\documentclass{article}
\pdfoutput=1

\usepackage[narrow]{arxiv}

\usepackage[utf8]{inputenc} 
\usepackage[T1]{fontenc}    
\usepackage[hypertexnames=false]{hyperref}       
\usepackage{url}            
\usepackage{booktabs}       
\usepackage{amsfonts}       
\usepackage{nicefrac}       
\usepackage{microtype}      

\usepackage{lipsum}         
\usepackage{graphicx}
\usepackage{natbib}
\usepackage{doi}

\title{Stochastic Non-Smooth Convex Optimization with Unbounded Gradients}
\renewcommand{\shorttitle}{Stochastic Non-Smooth Convex Optimization with Unbounded Gradients}

\newif\ifuniqueAffiliation
\uniqueAffiliationtrue

\ifuniqueAffiliation 
\author{
  Dmitry Kovalev \\
  Yandex Research \\
  \texttt{dakovalev1@gmail.com}
}
\else
\usepackage{authblk}

\author[1]{%
  Dmitry Kovalev\thanks{\texttt{dakovalev1@gmail.com}}%
}
\affil[1]{Yandex Research}
\fi

\usepackage{algorithm,algpseudocodex}
\usepackage{nicematrix}

\usepackage[textsize=tiny,color=orange!20,bordercolor=orange,]{todonotes}
\presetkeys{todonotes}{inline}{}

\usepackage{xcolor}
\hypersetup{
  colorlinks=true,
  linkcolor=red!75!black,
  citecolor=blue!50!black
}

\usepackage{mathtools,amsthm,amssymb}
\allowdisplaybreaks[4]

\usepackage{enumitem}
\setlist{topsep=0pt,partopsep=0pt,itemsep=0pt}

\usepackage{cleveref}
\usepackage{crossreftools}
\crefname{problem}{problem}{problems}
\crefname{condition}{condition}{conditions}
\crefname{case}{case}{cases}
\crefname{property}{property}{properties}
\creflabelformat{problem}{(#2#1#3)}
\creflabelformat{condition}{(#2#1#3)}

\usepackage{annotate}
\usepackage{mythm}
\usepackage{mymath}

\usepackage{graphicx}

\usepackage{bm}

\newcommand{\normo}{\gennorm{\mathrm{op}}{}}

\renewcommand{\norm}{\gennorm{2}{}}
\renewcommand{\sqn}{\gennorm{2}{2}}

\newcommand{\Rd}{\R^d}

\newcommand{\ox}{\overline{x}}

\newcommand{\df}{\mathrm{d}}

\usepackage{nccmath}

\newcommand{\mint}{\medint\int}
\newcommand{\tlim}{{\textstyle\lim}}
\newcommand{\tlimsup}{{\textstyle\limsup}}

\newcommand{\N}{\mathbb{N}}
\newcommand{\reg}{\mathrm{Reg}}
\newcommand{\las}{\overset{\smash{\text{a.s.}}}{\leq}}
\newcommand{\hC}{{\vphantom{\cC}\smash{\hat{\cC}}}}
\newcommand{\hF}{{\vphantom{\cF}\smash{\hat{\cF}}}}

\newcommand{\hK}{\vphantom{K}\smash{\hat{K}}}

\newcommand{\hg}{\hat{g}}
\newcommand{\hG}{\hat{G}}

\newcommand{\hepsilon}{\hat{\epsilon}}

\newcommand{\gpartial}{\partial^\circ}
\newcommand{\gprime}{^\circ}

\NewDocumentCommand{\clip}{sO{}m}{\operatorname{clip}_{#2}\IfBooleanTF{#1}{\brr*{#3}}{\brr{#3}}}

\usepackage{tikz}
\usepackage{pgfplots}

\pgfplotsset{compat=1.18}

\makeatletter

\newdimen\mysqrt@tot
\newdimen\mysqrt@bar

\newcommand{\mysqrt}[1]{%
  \mathchoice
  {\mysqrt@{\displaystyle}{#1}}%
  {\mysqrt@{\textstyle}{#1}}%
  {\mysqrt@{\scriptstyle}{#1}}%
  {\mysqrt@{\scriptscriptstyle}{#1}}%
}
\def\mysqrt@max#1#2{\ifdim #1>#2 #1\else #2\fi}

\newcommand{\mysqrt@}[2]{%
  \begingroup
  \setbox0=\hbox{$#1#2$}%
  \mysqrt@tot=\dimexpr\ht0+\dp0\relax
  \mysqrt@bar = \dimexpr\mysqrt@max{\dimexpr1.1\mysqrt@tot\relax}{\dimexpr1.5pt+\mysqrt@tot\relax}\relax
  \tikz[baseline=(R.base),inner sep=0pt,outer sep=0pt]{%
    \node[inner sep=0pt,outer sep=0pt] (R) {$#1#2$};%
    \draw[
      line width=0.45pt,
      line cap=round,
      line join=round
    ]
    ([xshift=-0.68\mysqrt@tot,yshift=0.5\mysqrt@tot]R.south west)
    --
    ([xshift=-0.6\mysqrt@tot,yshift=0.6\mysqrt@tot]R.south west);
    \draw[
      line width=0.45pt,
      line cap=round,
      line join=round
    ]
    ([xshift=-0.4\mysqrt@tot,yshift=0]R.south west)
    --
    ([xshift=0,yshift=\mysqrt@bar]R.south west)
    --
    ([xshift=0em,yshift=\mysqrt@bar]R.south east);
    \draw[
      line width=0.45pt,
      line cap=round,
      line join=round,
      fill
    ]
    ([xshift=-0.6\mysqrt@tot,yshift=0.6\mysqrt@tot]R.south west)
    --
    ([xshift=-0.38\mysqrt@tot,yshift=0.05\mysqrt@bar]R.south west)
    --
    ([xshift=-0.4\mysqrt@tot,yshift=0]R.south west)
    --
    ([xshift=-0.63\mysqrt@tot,yshift=0.5625\mysqrt@tot]R.south west);
    \useasboundingbox ([xshift=-0.6\mysqrt@tot,yshift=0]R.south west)
    rectangle
    ([xshift=0em,yshift=\mysqrt@bar]R.south east);
  }%
  \endgroup
}

\makeatother

\begin{document}
\maketitle

\begin{abstract}

  Much of the existing theory on first-order non-smooth optimization is built on a restrictive assumption that the gradients of the objective function are uniformly bounded. We introduce a much more realistic class of generalized Lipschitz functions, where the gradient norms are bounded by an affine function of the optimality gap. We then ask a natural question: what algorithm achieves the best global convergence rates for solving convex stochastic generalized Lipschitz optimization problems? To address this, we develop a new convergence analysis for several existing algorithms and find that AdamW with clipped updates, provably outperforms other popular stochastic optimization methods, such as SGD and AdaGrad. Moreover, our analysis establishes the critical role of AdamW's exponentially weighted gradient accumulation, as opposed to simple averaging. We further show that clipped AdamW is universal and achieves improved rates under the popular generalized smoothness assumption, analyze the convergence of clipped AdamW with diagonal and matrix preconditioners, and extend our results to the quasar-convex setting.
\end{abstract}

\section{Introduction}\label{sec:intro}

\textbf{Motivation.}
Stochastic non-smooth first-order optimization methods play a central role in modern machine learning, particularly in training deep neural networks with non-smooth components such as ReLU activations and max-pooling layers \citep{goodfellow2016deep,he2015delving,glorot2011deep}. The convergence properties of these methods are typically analyzed under the assumption that the objective function is globally Lipschitz, or equivalently, that its gradients (or their appropriate generalizations) are uniformly bounded \citep{nesterov2018lectures,shor2012minimization,zhang2020complexity}. Although this assumption is mathematically convenient and leads to clean worst-case convergence guarantees, it also imposes significant limitations. In particular, global Lipschitz continuity may fail even for the training loss of simple two-layer ReLU networks. Moreover, theory based on bounded gradients provides limited insight into the empirical success of Adam-type methods. Despite extensive efforts to develop better optimizers for deep learning, Adam and AdamW \citep{kingma2014adam,loshchilov2017decoupled} remain widely used and highly competitive, often outperforming other adaptive and non-adaptive methods such as AdaGrad \citep{duchi2011adaptive}, SGD \citep{robbins1951stochastic}, and Signum \citep{bernstein2018signsgd}. Yet, under the bounded gradients assumption, simple SGD or its momentum-based variants already attain optimal convergence rates in both convex \citep{nemirovski2009robust} and non-convex \citep{cutkosky2023optimal} settings. Therefore, even in the fundamental setting of non-smooth convex optimization, which is the focus of this paper, bounded-gradient theory leaves a clear gap in the understanding of the advantages of Adam-type algorithms over other first-order methods. This motivates the search for a suitable relaxation or generalization of the bounded-gradient assumption, together with an appropriate convergence theory capable of closing this gap.

\textbf{Summary of contributions.} Motivated by the above background, we present the following main contributions.
\begin{enumerate}[label={\bf(\roman{*})}]
  \item In \Cref{sec:M01}, we introduce the class of $(\cM_0,\cM_1)$-Lipschitz, or generalized Lipschitz, functions and derive their basic properties.
  \item In \Cref{sec:basic_alg}, we analyze the convergence of several deterministic and stochastic first-order methods for convex generalized Lipschitz functions. The results are summarized in \Cref{tab:GM01} and highlight the limitations of SGD and AdaGrad for stochastic generalized Lipschitz problems.
  \item In \Cref{sec:alg}, we show that scalar-stepsize AdamW with clipped updates and exponential decay significantly improves upon the convergence guarantees for SGD and AdaGrad obtained in \Cref{sec:basic_alg} for stochastic convex generalized Lipschitz problems.
  \item In \Cref{sec:lower}, we establish lower complexity bounds for SGD and AdaGrad on stochastic convex generalized Lipschitz problems, showing that the improvement achieved by scalar-stepsize AdamW with clipped updates is genuine rather than an artifact of the analysis.
  \item In \Cref{sec:uni}, we show that clipped scalar-stepsize AdamW is also effective under the generalized H\"older smoothness assumption, which interpolates between $(\cM_0,\cM_1)$-Lipschitzness and the popular $(\cL_0,\cL_1)$-smoothness assumption \citep{zhang2019gradient}.
  \item In \Cref{sec:quasar}, we extend our results to the quasar-convex setting \citep{hinder2020near}, which relaxes the standard convexity assumption.
  \item In \Cref{sec:preconditioning}, we show that our results are not specific to scalar stepsizes and can be extended to diagonal and matrix preconditioners.
\end{enumerate}

\textbf{Notations.}
In this paper we use the following notation:
$\norm{}$ denotes the standard Euclidean norm, induced by the standard inner product $\<,>$;
$\normi{}$ denotes the infinity vector norm;
$\normo{}$ denotes the matrix spectral norm;
for $x \in \Rd$, $\clip[2]{x}$ denotes the Euclidean norm clipping with unit radius, i.e., $\clip[2]{x} = x\cdot \min\brf{1,1/\norm{x}}$, and $\clip[\infty]{x}$ denotes the coordinate-wise clipping with unit radius;
$\odot$ and $\oslash$ denote elementwise multiplication and division, respectively;
$\cO\brs{}$ and $\Omega\brs{}$ hide universal constant multiplicative factors;
$\tilde\cO\brs{}$ hides universal constant and polylogarithmic multiplicative factors, depending on the assumption constants and target accuracy;
$\N$ and $\N_0$ denote the sets of positive and non-negative integers, respectively.
We also use the conventions $1/0 = +\infty$ and $0/0 = 0$; whenever an expression defines a function $\psi(t)$ for $t > 0$ that admits a continuous extension to $t = 0$, we set $\psi(0) = \lim_{t\to +0}\psi(t)$.

\section{Generalized Lipschitz Optimization}\label{sec:M01}

Throughout this paper, we focus on the problem of minimizing a locally Lipschitz, possibly non-smooth, objective function $f(x)\colon \Rd \to \R$. We assume that at least one global solution $x^* \in \Rd$ exists and denote the corresponding optimal function value as $f^* = f(x^*)$. We also assume that $\norm{x^*} \leq \cR$ for some positive constant $\cR > 0$.

\subsection{Generalized Lipschitz Functions}

As discussed in \Cref{sec:intro}, assuming global Lipschitz continuity is restrictive. Therefore, we introduce a relaxed condition, which we refer to as $(\cM_0,\cM_1)$-Lipschitzness or generalized Lipschitzness. In the case where the objective function $f(x)$ is differentiable, this condition can be stated as follows:
\begin{equation}\label[condition]{eq:M01_grad}
  \norm{\nabla f(x)} \leq \cM_0 + \cM_1\brs{f(x) - f^*}
  \quad\text{for all}\;\; x \in \Rd,
\end{equation}
which clearly generalizes the standard bounded gradients assumption. In the case where the function $f(x)$ is not differentiable, we provide even more general \Cref{ass:M01}.
\begin{assumption}<ass:M01>
  The function $f(x)$ is $(\cM_0,\cM_1)$-Lipschitz. That is, the following inequality holds:
  \begin{equation}
    \tlimsup_{u \to 0}\smash{\tmfrac{1}{\norm{u}}}\abs{f(x + u) - f(x)}
    \leq \cM_0 + \cM_1\brs{f(x) - f^*}
    \quad\text{for all}\;\; x \in \R^d.
  \end{equation}
\end{assumption}

\textbf{Motivation.} Our main theoretical motivation for \cref{eq:M01_grad} comes from its parallel with the following $(\cL_0,\cL_1)$-smoothness condition, also known as generalized smoothness, introduced by \citet{zhang2019gradient} for twice continuously differentiable functions:
\begin{equation}\label[condition]{eq:L01_grad}
  \normo{\nabla^2 f(x)} \leq \cL_0 + \cL_1 \norm{\nabla f(x)}
  \quad\text{for all}\;\; x \in \Rd.
\end{equation}
This condition was introduced to relax the standard smoothness assumption commonly used in smooth first-order optimization \citep{nesterov2018lectures}, namely, the global Lipschitz continuity of the gradient. Observe that \cref{eq:L01_grad} bounds the second derivatives by an affine function of the first derivatives. This suggests an analogous condition for non-smooth first-order optimization: reducing the order of the derivatives by one and recalling that the zeroth derivative is the function itself leads precisely to \cref{eq:M01_grad}. Additionally, Lemma~2.2 of \citet{gorbunov2024methods} shows that the generalized smoothness implies the more general bound
\begin{equation}\label{eq:gorbunov}
  \norm{\nabla f(x)} \leq \cO\big[\mysqrt{\cL_0\brs{f(x) - f^*}} + \cL_1\brs{f(x) - f^*}\big].
\end{equation}
On the other hand, standard global $\cL_0$-smoothness and $\cM_0$-Lipschitzness imply, respectively, $\norm{\nabla f(x)} \leq \cO\big[\mysqrt{\cL_0\brs{f(x) - f^*}}\big]$ and $\norm{\nabla f(x)} \leq \cM_0$. Thus, replacing the ``$\cL_0$-part'' in \cref{eq:gorbunov} with its ``$\cM_0$-analogue'' leads to our \cref{eq:M01_grad}. This observation also motivates the generalized H\"older smoothness condition studied in \Cref{sec:uni}, which interpolates between \cref{eq:M01_grad,eq:L01_grad}.

\textbf{Basic examples.} To show formally that \Cref{ass:M01} strictly generalizes the standard global Lipschitz continuity assumption, we provide two simple examples. Let $p > 1$, $A \in \R^{n\times d}$, and $b \in \R^n$. Then, the function $f(x) = \normi{Ax - b}^p$ is $(\cM_0,\cM_1)$-Lipschitz with arbitrary $\cM_1 > 0$ and $\cM_0 = \normo{A}^p\brr{\frac{p-1}{\cM_1}}^{p-1} + \cM_1 f^*$, and the function $f(x) = \exp(\normi{Ax - b})$ is $(\cM_0,\cM_1)$-Lipschitz with $\cM_1 = \normo{A}$ and $\cM_0 = \cM_1 f^*$.

\subsection{Basic Properties of Generalized Lipschitz Functions}\label{sec:M01_properties}

In this section, we derive the basic properties of generalized Lipschitz functions. We start with the equivalent characterization of $(\cM_0,\cM_1)$-Lipschitz functions in the following \Cref{lem:M01}.

\begin{lemma}<lem:M01>
  \Cref{ass:M01} holds if and only if the following inequality holds for all $x,x' \in \Rd$:
  \begin{equation}\label{eq:M01}
    \begin{aligned}
      \abs{f(x') - f(x)} &\leq
      \brr*{\brs{\cM_0/\cM_1} + \brs*{f(x) - f^*}}\brs{\exp\brr{\cM_1\norm{x'-x}} - 1}
      \\&\leq
      \brr*{\cM_0 + \cM_1\brs*{f(x) - f^*}}\exp\brr{\cM_1\norm{x'-x}}\norm{x'-x}
    \end{aligned}
  \end{equation}
\end{lemma}

In the case $\cM_1 = 0$, the inequality in \cref{eq:M01} matches the standard definition of Lipschitz functions, where the function value difference $\abs{f(x') - f(x)}$ is bounded by a multiple of the distance between the arguments $\norm{x' - x}$. However, in the general case $\cM_1 > 0$, the bound includes an additional exponential factor $\exp\brr{\cM_1\norm{x' - x}}$. This factor cannot be removed in general, which can be illustrated by the basic examples mentioned above.

\Cref{lem:M01} also raises the question of whether the generalized Lipschitzness condition offers any advantage over the standard global Lipschitz continuity. Indeed, one could consider minimizing the objective function $f(x)$ over the ball $Q_\cR = \brf{x : \norm{x} \leq \cR}$ of radius $\cR$ centered at the origin, which is equivalent to the original problem due to $x^* \in Q_\cR$. This approach is common, for instance, in adaptive methods such as AdaGrad, where projection steps on $Q_\cR$ are used to ensure the almost sure boundedness of the iterates \citep{duchi2011adaptive}, which is necessary for convergence guarantees. \Cref{lem:M01} implies that the function $f(x)$ is globally $\cM$-Lipschitz on $Q_\cR$, where
\begin{equation}\label{eq:MF}
  \cM = \cM_0 + \cM_1 \cF,\qquad
  \cF = \tmax_{x \in Q_\cR} \brs{f(x) - f^*}.
\end{equation}
However, this approach implies convergence rates with extra exponential factors, since the constants $\cM$ and $\cF$ may scale exponentially, according to \Cref{lem:M01} and the discussion following it:
\begin{equation}\label{eq:MF_bound}
  \cM \leq \cM_0\exp\brr{\cO\brs{\cR\cM_1}},
  \qquad
  \cF \leq \brs{\cM_0/\cM_1}\brs{\exp\brr{\cO\brs{\cR\cM_1}} - 1}.
\end{equation}
Consequently, such convergence rates depending on $\cM$ and $\cF$ can be overly pessimistic, especially when $\cR\cM_1 \gg 1$. In the upcoming sections, we will show that \Cref{ass:M01} can be utilized more directly, leading to fast convergence rates without any exponential factors.

\section{Algorithms for Generalized Lipschitz Convex Optimization}\label{sec:basic_alg}

Throughout most of this paper, we focus on the case where the objective function $f(x)$ is convex. By standard results in convex analysis \citep{rockafellar1997convex}, the subdifferential $\partial f(x)$ is non-empty at every point $x \in \Rd$. We denote an arbitrary subgradient by $\nabla f(x) \in \partial f(x)$ and, with a slight abuse of language, refer to it as the gradient of the function $f(x)$ at the point $x \in \Rd$. Consequently, it is easy to verify that \Cref{ass:M01} holds if and only if the inequality in \cref{eq:M01_grad} holds.

\begin{table}[t]
  \centering
  \caption{Summary of the main complexity results for solving non-stochastic and stochastic generalized Lipschitz convex optimization problems (\Cref{sec:M01,sec:alg}). Universal and logarithmic constants are omitted. The conditions $\cR\cM_1 \geq 1$ and $\cR\cG_1 \geq 1$ are assumed.}
  \label{tab:GM01}
  \begin{NiceTabular}{|c|c|r|}
    \toprule
    \bf Algorithm & \bf Theorem & \Block[c]{1-1}{\bf Iteration Complexity} \\
    \midrule
    \Block[c]{1-3}{\bf Generalized Lipschitz Problems (\Cref{ass:M01})}&&\\
    \midrule
    GD & \Cref{cor:GD} & $\exp\brr{\cO\brs{\cR\cM_1}} \cdot \brs{\cR\cM_0 / \epsilon} + \brs{\cR\cM_0 / \epsilon}^2$\\
    \midrule
    Normalized/Polyak GD & \Cref{thm:normalized_polyak} & $\brs{\cR\cM_1}^2 + \brs{\cR\cM_0/\epsilon}^2$\\
    \midrule
    \Block[c]{1-3}{\bf Stochastic Generalized Lipschitz Problems (\Cref{ass:G01})}&&\\
    \midrule
    SGD & \Cref{thm:SGD} & $\exp\brr{\cO\brs{\cR\cG_1}} \cdot \brs{\cR\cG_0 / \epsilon} + \brs{\cR\cG_0 / \epsilon}^2$\\
    \midrule
    AdaGrad-Norm & \Cref{thm:AdaGrad} & $\exp\brr{\cO\brs{\cR\cG_1}} \cdot \brs{\cR\cG_0 / \epsilon} + \brs{\cR\cG_0 / \epsilon}^2$\\
    \midrule
    Clipped AdamW & \Cref{cor:exp2} & $\brs{\cR\cG_1}^4 + \brs{\cR\cG_0 / \epsilon}^2$\\
    \bottomrule
  \end{NiceTabular}
\end{table}

\subsection{Deterministic Algorithms}\label{sec:det}

In this section, we analyze the convergence of projected gradient descent (GD) under \Cref{ass:M01}. This algorithm iterates according to the following update rule:
\begin{equation}\label{eq:GD}
  x_{k+1} = \proj[Q_\cR]{x_k - \eta_k \nabla f(x_k)},\quad x_0 = 0,
\end{equation}
where $\eta_k > 0$ is a possibly time-varying stepsize. We first consider the constant-stepsize case in \Cref{thm:GD}. Then, using the upper bound on $\cF$ in \cref{eq:MF_bound}, we derive in \Cref{cor:GD} an explicit iteration complexity in terms of the constants $\cR$, $\cM_0$, $\cM_1$, and the target accuracy $\epsilon$.

\begin{theorem}<thm:GD>
  Let \Cref{ass:M01} hold and let the iterates $\brf{x_k}_{k=0}^K$ be generated by \cref{eq:GD}. Let the stepsizes $\eta_k > 0$ and the number of iterations $K \in \N$ be defined as follows:
  \begin{equation}\label{eq:GD_params}
    1/\eta_k = \max\brf{2\cM_1^2\cF, \cM_0\mysqrt{K+1}/\cR},\qquad K = \ceil*{\brs{\cR\cM_1}^2 \cdot \brs{\cF /\epsilon} + \brs{\cR\cM_0 / \epsilon}^2}.
  \end{equation}
  Then, the inequality $f(\ox_K) - f^* \leq \cO\brs{\epsilon}$ holds, where $\ox_K = \smash{\frac{1}{K+1}\sum_{k=0}^K x_k}$.
\end{theorem}

\begin{corollary}<cor:GD>
  Under the conditions of \Cref{thm:GD}, it is sufficient to perform the following number of iterations to reach the precision $f(\ox_K) - f^* \leq \epsilon$:
  \begin{equation}
    K = \cO\brs*{
      \exp\brr{\cO\brs{\cR\cM_1}} \cdot \brs{\cR\cM_0 / \epsilon} + \brs{\cR\cM_0 / \epsilon}^2
    }.
  \end{equation}
\end{corollary}

The complexity bound in \Cref{cor:GD} improves upon the iteration complexity that could be obtained by naively using the global $\cM$-Lipschitzness of the function $f(x)$ over $Q_\cR$, as discussed in \Cref{sec:M01_properties}. It also admits the following informal interpretation: the bound in \cref{eq:GD_rate} has an initial regime with the sublinear rate $\cO\brs{[\cR\cM_1]^2\cF/K}$ lasting up to $K_\text{init} = \cO\brs{\brs{\cR\cM_1}^2 \cdot \brs{\cM_1\cF/\cM_0}^2}$ iterations, followed by the slower rate $\cO\brs{\cR\cM_0/\mysqrt{K}}$, matching the standard rate of GD for $\cM_0$-Lipschitz convex functions. Although the latter rate cannot be improved in general due to the standard lower complexity bounds \citep{nesterov2018lectures}, the initial phase before reaching it can be exponentially long: by \cref{eq:MF_bound}, $K_\text{init} = \cO\brs{\exp\brr{\cO\brs{\cR\cM_1}}}$. Thus, despite the improvement over the naive global-Lipschitz reduction, the overall iteration complexity can still be poor, especially when $\cR\cM_1 \gg 1$.

It turns out that we can further improve the convergence rate of GD by utilizing adaptive stepsizes. In particular, we consider normalized GD \citep{shor2012minimization} and GD with the Polyak stepsizes \citep{polyak1969minimization}, Option~A and Option~B in \cref{eq:normalized_polyak_eta}, respectively. By \Cref{thm:normalized_polyak}, both methods require only $K_{\text{init}} = \cO\brs{\brs{\cR\cM_1}^2}$ initial iterations before reaching the optimal rate $\cO\brs{\cR\cM_0 / \mysqrt{K+1}}$. Consequently, the overall iteration complexity of normalized GD and Polyak GD contains no exponential factors and significantly improves upon the nonadaptive GD rate in \Cref{thm:GD}.

\begin{theorem}<thm:normalized_polyak>
  Let \Cref{ass:M01} hold and let the iterates $\brf{x_k}_{k=0}^K$ be generated by \cref{eq:GD}. Let the stepsizes $\eta_k > 0$ be defined according to the following options:
  \begin{equation}\label{eq:normalized_polyak_eta}
    \text{Option A:}\;\;
    \eta_k = \cR/\brs{\norm{\nabla f(x_k)}\mysqrt{K+1}}
    ,\quad\text{Option B:}\;\;
    \eta_k = \brs{f(x_k) - f^*}/\sqn{\nabla f(x_k)},
  \end{equation}
  and let the number of iterations $K \in \N$ be defined as follows:
  \begin{equation}\label{eq:normalized_polyak_K}
    K = 4\cdot \ceil*{\brs{\cR\cM_1}^2 + \brs{\cR\cM_0/\epsilon}^2}.
  \end{equation}
  Then, the inequality $\min_{k \in \brf{0,\dots,K}}\brs*{f(x_k) - f^*} \leq \epsilon$ holds.
\end{theorem}

\subsection{Stochastic Algorithms}\label{sec:stoch}

In this section, we analyze the convergence of projected stochastic gradient descent (SGD):
\begin{equation}\label{eq:SGD}
  x_{k+1} = \proj[Q_\cR]{x_k - \eta_k \nabla_{\xi_k} f(x_k)},\quad x_0 = 0,
\end{equation}
where $\xi_k \sim \cD$ are independent samples and $\nabla_\xi f(x)$ is the stochastic gradient estimator defined in \Cref{ass:G01}, a stochastic version of \Cref{ass:M01}. The unbiasedness requirement is standard in stochastic first-order methods. The decomposition in \cref{eq:G01_grad} can be viewed as a stochastic analogue of \cref{eq:M01_grad}; in fact, it implies \cref{eq:M01_grad} with $\cM_0 = \cG_0$ and $\cM_1 = \cG_1$, as stated in \Cref{rem:GM01}. For technical reasons, the condition in \cref{eq:G01_grad} is slightly stronger than merely imposing an upper bound, analogous to \cref{eq:M01_grad}, on the root second moment $\mysqrt{\E{\norm{\nabla_\xi f(x)}^2}}$. Nevertheless, it remains substantially more general than the standard bounded-second-moment assumption recovered by setting $\cG_1 = 0$, which implies bounded full gradients and is commonly used in the stochastic non-smooth convex optimization literature \citep{nemirovski2009robust,bubeck2015convex}.

\begin{assumption}<ass:G01>
  There exists a stochastic estimator $\nabla_\xi f(x)$ of the gradient $\nabla f(x)$, where $\xi \sim \cD$ is a random variable sampled from the distribution $\cD$. The gradient estimator $\nabla_\xi f(x)$ is unbiased, i.e., $\E[\xi \sim \cD]{\nabla_\xi f(x)} = \nabla f(x)$, and allows the decomposition $\nabla_\xi f(x) = u_\xi(x) + v_\xi(x)$, satisfying
  \begin{equation}\label{eq:G01_grad}
    \E[\xi \sim \cD]{\sqn{u_\xi(x)}}\leq \cG_0^2
    \quad\text{and}\quad
    \norm{v_\xi(x)} \las \cG_1\brs{f(x) - f^*},
    \quad\text{where\;\;}
    \cG_0,\cG_1 \geq 0.
  \end{equation}
\end{assumption}

\begin{remark}<rem:GM01>
  \Cref{ass:G01} implies \Cref{ass:M01} with $\cM_0 = \cG_0$ and $\cM_1 = \cG_1$.
\end{remark}

Next, under \Cref{ass:G01}, we analyze SGD with constant stepsizes and AdaGrad-Norm stepsizes \citep{streeter2010less}, a scalar variant of AdaGrad. The corresponding results are stated in \Cref{thm:SGD,thm:AdaGrad}, respectively. The convergence rates for the two algorithms coincide. This is not particularly surprising: it is consistent with existing theory for stochastic convex globally Lipschitz optimization, where the rates for SGD and AdaGrad-Norm also coincide \citep{nemirovski2009robust,duchi2011adaptive}.

\begin{theorem}<thm:SGD>
  Let \Cref{ass:G01} hold and let the iterates $\brf{x_k}_{k=0}^K$ be generated by \cref{eq:SGD}.  Let the stepsizes $\eta_k > 0$ and the number of iterations $K \in \N$ be defined as follows:
  \begin{equation}\label{eq:SGD_params}
    \eta_k = 1/\max\brf{2\cG_1^2\cF, \cG_0\mysqrt{K+1}/\cR}
    ,\quad
    K = \cO\brs*{
      \exp\brr{\cO\brs{\cR\cG_1}} \cdot \brs{\cR\cG_0 / \epsilon} + \brs{\cR\cG_0 / \epsilon}^2
    }.
  \end{equation}
  Then, the inequality $\E*{f(\ox_K) - f^*} \leq \epsilon$ holds, where $\ox_K = \smash{\frac{1}{K+1}\sum_{k=0}^K x_k}$.
\end{theorem}

\begin{theorem}<thm:AdaGrad>
  Let \Cref{ass:G01} hold and let the iterates $\brf{x_k}_{k=0}^K$ be generated by \cref{eq:SGD}.  Let the stepsizes $\eta_k > 0$ and the number of iterations $K \in \N$ be defined as follows:
  \begin{equation}\label{eq:AdaGrad_params}
    \eta_k = \cR/\mysqrt{2\tsum_{i=0}^{k}\sqn{\nabla_{\xi_i}f(x_i)}}
    ,\quad
    K = \cO\brs*{
      \exp\brr{\cO\brs{\cR\cG_1}} \cdot \brs{\cR\cG_0 / \epsilon} + \brs{\cR\cG_0 / \epsilon}^2
    }.
  \end{equation}
  Then, the inequality $\E*{f(\ox_K) - f^*} \leq \epsilon$ holds, where $\ox_K = \smash{\frac{1}{K+1}\sum_{k=0}^K x_k}$.
\end{theorem}

As in the deterministic GD result in \Cref{thm:GD}, the iteration complexities in \Cref{thm:SGD,thm:AdaGrad} contain exponential factors and therefore can be very large. This raises the natural question of whether gradient normalization or Polyak stepsizes can yield improvements analogous to those obtained in the deterministic setting in \Cref{thm:normalized_polyak}. Unfortunately, the analysis in \Cref{thm:normalized_polyak} does not directly extend to the stochastic setting: gradient normalization and Polyak stepsizes generally break the unbiasedness property, which is crucial for the convergence analysis. In fact, normalized SGD can diverge even on simple one-dimensional convex problems \citep[Counterexample~1]{karimireddy2019error}. However, in the next \Cref{sec:alg}, we show that it is still possible to significantly improve upon the convergence rates of SGD and AdaGrad-Norm.

\section{Better Algorithms for Stochastic Convex Generalized Lipschitz Problems}\label{sec:alg}

\subsection{Algorithm Construction}\label{sec:wd}

In this section, we show how to construct an algorithm with convergence guarantees that significantly improve upon those in \Cref{sec:stoch}. We begin by defining the functions $f_k(x)\colon \Rd \to \R$ for $k \in \N$:
\begin{equation}\label{eq:fk}
  f_k(x) = \pi_k f\brr{\alpha_kx + \brr{1-\alpha_k}x_{k-1}},
  \quad
  \text{where}\quad
  \pi_k(1-\alpha_k) = \pi_{k-1}.
\end{equation}
Here, $\alpha_k \in [0,1]$ and $\pi_{k-1} \geq 0$ for $k \in \N$.
We then consider the following iterations for $k \in \N$:
\begin{equation}\label{eq:wd}
  x_k = (1-\alpha_k)x_{k-1} + \alpha_k z_k,\quad x_0 = 0,
\end{equation}
where the points $z_k \in Q_\cR$ are generated by an online optimization algorithm applied to an online linear optimization problem with regret
\begin{equation}\label{eq:regret}
  \reg_K = \tmax_{z \in Q_\cR}\tsum_{k=1}^K \<z_k - z,g_k>,
  \quad\text{where}\;\;
  g_k = \nabla_{\xi_k} f_k(z_k).
\end{equation}
Here, $\xi_k \sim \cD$ are random samples independent of the iterates $z_1,\dots,z_k$. When $1/\alpha_k = \pi_k = k$, the iterations in \cref{eq:wd} can be seen as a form of primal averaging \citep{nesterov2015quasi,tao2018primal} or online-to-batch conversion \citep{cutkosky2019anytime}. These techniques yield convergence guarantees for the last iterate $x_K$, rather than for the averaged iterate $\ox_K$ as in \Cref{thm:SGD,thm:AdaGrad}. However, we will later show that this choice of parameters leads to significantly suboptimal convergence rates under \Cref{ass:G01}. Consequently, in \Cref{lem:wd}, we analyze the iterations in \cref{eq:wd} for a general choice of parameters.

\begin{lemma}<lem:wd>
  Under \Cref{ass:G01}, the following inequality holds:
  \begin{equation}
    \pi_K \E{f(x_K) - f^*}\leq \pi_0\brs{f(x_0) - f^*} + \E{\reg_K}.
  \end{equation}
\end{lemma}

We further focus on online optimization algorithms that guarantee the regret bound stated in \Cref{ass:regret}. This bound is optimal when no assumptions are imposed on the gradients $g_k$ \citep{orabona2018scale}. Moreover, we will show that it enables adaptation to the unbounded gradients allowed under \Cref{ass:G01}. Several algorithms achieve this bound, including AdaGrad-Norm \citep{streeter2010less}, SOLO-FTRL \citep{orabona2018scale}, and the scalar-stepsize variant of Leon \citep{jiang2026adaptive,kovalev2026optimal}.

\begin{assumption}<ass:regret>
  The iterates $z_k \in Q_\cR$ satisfy the following, where $\cC \geq 1$ is a universal constant:
  \begin{equation}
    \reg_K \las \cC\cR\mysqrt{\tsum_{k=1}^K \sqn{g_k}}.
  \end{equation}
\end{assumption}

\subsection{Convergence Analysis with Default Averaging}\label{sec:avg}

\begin{figure}[t]
  \vspace{-2em}
  \begin{minipage}{0.48\linewidth}
    \begin{algorithm}[H]
      \caption{Clipped AdamW (\Cref{thm:avg})}\label{alg:avg}
      \begin{algorithmic}[1]
        \State $m_0 = 0$, $v_0 = 0$, $x_1 = 0$
        \For{$k=1,\dots,K-1$}
        \State $\hg_k = \nabla_{\xi_k} f(x_k)$, where $\xi_k \sim \cD$
        \State $m_k = m_{k-1}+\hg_k \vphantom{m_k = (1-\alpha)m_{k-1}+ \alpha \hg_k}$\label{line:avg:m}
        \State $v_k = v_{k-1} +  \sqn{\hg_k}\vphantom{v_k = (1-\alpha)^2v_{k-1} +  \alpha^2\sqn{\hg_k}}$\label{line:avg:v}
        \State $x_{k+1} = \frac{k}{k+1}x_k - \frac{\cR}{k+1}\clip[2]{m_k / \mysqrt{v_k}}$
        \EndFor
      \end{algorithmic}
    \end{algorithm}
  \end{minipage}
  \hfill
  \begin{minipage}{0.51\linewidth}
    \begin{algorithm}[H]
      \caption{Clipped AdamW (\Cref{thm:exp1})}\label{alg:exp}
      \begin{algorithmic}[1]
        \State $m_0 = 0$, $v_0 = 0$, $x_1 = 0$
        \For{$k=1,\dots,K-1$}
        \State $\hg_k = \nabla_{\xi_k} f(x_k)$, where $\xi_k \sim \cD$
        \State $m_k = (1-\alpha)m_{k-1}+ \alpha \hg_k$\label{line:exp:m}
        \State $v_k = (1-\alpha)^2v_{k-1} +  \alpha^2\sqn{\hg_k}$\label{line:exp:v}
        \State $x_{k+1} = (1-\alpha)x_k - \alpha \cR \clip[2]{m_k / \mysqrt{v_k}}
        \vphantom{x_{k+1} = \frac{k}{k+1}x_k - \frac{\cR}{k+1}\clip[2]{m_k / \mysqrt{v_k}}}$
        \label{line:exp:x}
        \EndFor
      \end{algorithmic}
    \end{algorithm}
  \end{minipage}
\end{figure}

In this section, we analyze the convergence of the iterations in \cref{eq:wd} under the default choice of parameters: $1/\alpha_k = \pi_k = k$. When SOLO-FTRL is used as the online learning algorithm to guarantee the regret bound in \Cref{ass:regret}, we obtain \Cref{alg:avg}, which is very similar to AdamW with scalar stepsizes and clipped updates. Further, we mostly focus on this particular algorithm, although the results and discussions apply to the iterations in \cref{eq:wd} combined with other online learning algorithms that generate the iterates $z_k$ satisfying \Cref{ass:regret}.

We begin the convergence analysis with an upper bound on the expected regret in \Cref{lem:regret}. Next, in \Cref{thm:avg}, we establish the main convergence result under the default choice of parameters. Finally, using the inequality in \cref{eq:MF_bound}, we derive an explicit iteration complexity in \Cref{cor:avg}.

\begin{lemma}<lem:regret>
  Let \Cref{ass:G01,ass:regret} hold. Then, the following inequality holds:
  \begin{equation}
    \E{\reg_K} \leq
    \cC\cR\cG_0\mysqrt{\tsum_{k=1}^K \alpha_k^2\pi_k^2}
    + \E*{\cC\cR\cG_1\mysqrt{\tsum_{k=1}^K \alpha_k^2\pi_k^2\brs*{f(x_k) - f^*}^2}}.
  \end{equation}
\end{lemma}

\begin{theorem}<thm:avg>
  Let \Cref{ass:regret,ass:G01} hold, and assume $\cR\cG_1 \geq 1$. Let $\alpha_k = 1/k$ and $\pi_0 = 0$. Let $K \in \N$ be defined as follows:
  \begin{equation}\label{eq:K0}
    K = \ceil*{
      \brs{\cR\cG_0/\epsilon}^2
      +
      \brs*{\brs{\cR\cG_1}^2\cdot \brs{\cF / \epsilon}}
      \cdot \ln\brs*{e+\brs{\cR\cG_0/\epsilon} + \brs{\cR\cG_1}^2\cdot\brs{\cF/\epsilon}}
    }.
  \end{equation}
  Then, the inequality $\E*{f(x_K) - f^*}\leq \cO\brs{\epsilon}$ holds.
\end{theorem}
\begin{corollary}<cor:avg>
  Under the conditions of \Cref{thm:avg}, to reach the precision $\E*{f(x_K) - f^*}\leq \epsilon$, it is sufficient to perform the following number of iterations:
  \begin{equation}
    K = \tilde\cO\brs*{
      \exp\brr{\cO\brs{\cR\cG_1}} \cdot \brs{\cR\cG_0/\epsilon}
      + \brs{\cR\cG_0/\epsilon}^2
    }.
  \end{equation}
\end{corollary}

Unfortunately, the complexity bound for \Cref{alg:avg} in \Cref{cor:avg} does not improve upon the corresponding bounds for SGD and AdaGrad-Norm in \Cref{thm:SGD,thm:AdaGrad}. As in the discussion of GD convergence in \Cref{sec:det}, this bound admits an informal two-phase interpretation, with slow sublinear rates in both phases. The main limiting factor is the default choice of parameters for the iterations in \cref{eq:wd}, which corresponds to simple averaging of the gradients and squared gradient norms in \Cref{line:avg:m,line:avg:v} of \Cref{alg:avg}. This simple averaging is the key difference from practical Adam-type methods, which use exponential decay instead. In the next \Cref{sec:exp}, we show that exponential decay is not merely a practical heuristic, but can substantially improve upon the complexity bound in \Cref{cor:avg}.

\subsection{Improved Rates via Exponential Decay}\label{sec:exp}

In this section, we show that replacing the simple averaging in \Cref{alg:avg} with exponential decay improves the convergence rate. Specifically, we replace the default schedule $1/\alpha_k = \pi_k = k$ in \cref{eq:wd} with a constant parameter $\alpha_k = \alpha$. This yields \Cref{alg:exp}, which coincides with scalar-stepsize AdamW with clipped updates. The constant choice of $\alpha_k$ is exactly what induces the exponential decay in \Cref{line:exp:m,line:exp:v} of \Cref{alg:exp}.

We begin the convergence analysis with \Cref{lem:regret_exp1}, which gives the key improved upper bound on the regret. It exploits the fact that, with a sufficiently small choice of the parameters $\alpha_k$, the iterates in \cref{eq:wd} remain stable, i.e., $x_k \approx x_{k-1}$. Intuitively, this may imply the stability of the gradients, i.e., $g_k \approx g_{k-1}$, which makes the method behave more like normalized GD in \Cref{thm:normalized_polyak} by accumulating stable squared gradient norms in \Cref{line:exp:v}. Formally, the argument uses the stability of the function values $f(x_k)$, which accounts for the improvement of the bound in \Cref{lem:regret_exp1} over the bound in \Cref{lem:regret}. Note that \citet{nesterov2015quasi,cutkosky2019anytime} also made informal stability-related observations about primal averaging and online-to-batch conversion, but did not provide any formal justification that these schemes actually improve convergence rates. In contrast, our analysis in \Cref{proof:lem:regret_exp1} is substantially different: it formally exploits the iterate stability and leads to improved convergence guarantees.

\begin{lemma}<lem:regret_exp1>
  Let \Cref{ass:regret,ass:G01} hold and assume $\cR\cG_1 \geq 1$. Let $\alpha_k$ satisfy the following:
  \begin{equation}\label{eq:alpha_T}
    \alpha_k = \alpha
    ,\quad\text{where}\;\;
    \alpha T \cR\cG_1\leq 1
    \quad\text{and}\quad
    T=\ceil*{\hC\brs{\cR\cG_1}^2},
  \end{equation}
  where $\hC$ is a sufficiently large universal constant.
  Then, the following inequality holds for all $K \geq T$:
  \begin{equation}\label{eq:regret_bound}
    \E*{\reg_K}
    \leq
    \tfrac{1}{2}\tsum_{k=1}^K\alpha_k\pi_k\E*{f(x_k) - f^*}
    +
    \cO\brs*{
      \cR\cG_0\mysqrt{\tsum_{k=1}^K \alpha_k^2\pi_k^2}
      +
      T\cR\cG_0\tsum_{k=1}^K\alpha_k^2\pi_k
    }.
  \end{equation}
\end{lemma}

Armed with \Cref{lem:regret_exp1,lem:wd}, we derive the main convergence result in \Cref{thm:exp1} and the explicit iteration complexity in \Cref{cor:exp1}. As for \Cref{alg:avg} in \Cref{sec:avg}, the complexity bound for \Cref{alg:exp} in \Cref{cor:exp1} admits a two-phase interpretation. However, while \Cref{alg:avg} and the other stochastic methods from \Cref{sec:stoch} achieve slow sublinear convergence rates in both phases, \Cref{alg:exp} can achieve fast linear convergence in the first phase, up to the precision $\hepsilon = \cO\brs{\cG_0/\cG_1}$. This improvement can be significant when $\cR \cG_1\gg 1$ and, more importantly, eliminates exponential factors from the iteration complexity.

\begin{theorem}<thm:exp1>
  Let \Cref{ass:regret,ass:G01} hold, and assume $\cR\cG_1 \geq 1$. Let $\pi_0 = 1$, and let $\alpha_k$ and $K \in \N$ be defined as follows:
  \begin{equation}\label{eq:alpha_K1}
    1/\alpha_k = 1/\alpha = \max\brf*{T\cR\cG_1, \brs{\cR\cG_0/\epsilon}^2, \brs{T\cR\cG_0/\epsilon}},
    \quad
    K =  \ceil*{\brs{2/\alpha_1}\ln \brs*{e + \brs{\cF/\epsilon}}},
  \end{equation}
  where $T$ is defined in \cref{eq:alpha_T}.
  Then, the inequality $\E*{f(x_K) - f^*}\leq \cO\brs{\epsilon}$ holds.
\end{theorem}

\begin{corollary}<cor:exp1>
  Under the conditions of \Cref{thm:exp1}, to reach the precision $\E*{f(x_K) - f^*}\leq \epsilon$, it is sufficient to perform the following number of iterations:
  \begin{equation}
    K = \tilde\cO\brs*{[\cR\cG_1]^4 + \brs{\cR\cG_1}^3 \cdot \brs{\cR\cG_0/\epsilon} + \cR\cG_1 \cdot \brs{\cR\cG_0 / \epsilon}^2}
  \end{equation}
\end{corollary}

Unfortunately, a constant choice of $\alpha_k$ is no longer optimal when the target precision is below $\hepsilon$. However, this limitation can be addressed with a simple two-stage schedule for $\alpha_k$. We extend the regret bound in \cref{eq:regret_bound} to this schedule in \Cref{lem:regret_exp2}, and then derive the corresponding convergence result and explicit complexity bound in \Cref{thm:exp2} and \Cref{cor:exp2}. The resulting guarantee combines the best features of \Cref{alg:avg,alg:exp}: it achieves fast linear convergence in the first phase and optimal sublinear convergence in the second phase.

\begin{lemma}<lem:regret_exp2>
  Under the conditions of \Cref{thm:exp2}, the inequality in \cref{eq:regret_bound} holds for $K \geq T$.
\end{lemma}

\begin{theorem}<thm:exp2>
  Let \Cref{ass:regret,ass:G01} hold, and assume $\cR\cG_1 \geq 1$. Let $\alpha_k$ be defined as follows:
  \begin{equation}\label{eq:alpha2}
    1/\alpha_k =
    \begin{cases}
      T\cR\cG_1 & k \leq S T \\
      \max\brf*{T\cR\cG_1, \brs{\cR\cG_0/\epsilon}^2, \brs{T\cR\cG_0/\epsilon}}  & k > S T
    \end{cases},
  \end{equation}
  where $T$ is defined in \cref{eq:alpha_T}, and let $\pi_0 = 1$. Let $K,S \in \N$ be defined as follows:
  \begin{equation}\label{eq:KS2}
    K = TS + \ceil*{\brs{2/\alpha_{TS+1}} \ln \brs{e + \brs{\cG_0/(\cG_1\epsilon)}}}
    ,\quad
    S = \ceil*{\brs{2/(\alpha_1T)} \ln \brs{e + \brs{\cF\cG_1/\cG_0}}}.
  \end{equation}
  Then, the inequality $\E*{f(x_K) - f^*}\leq \cO\brs{\epsilon}$ holds.
\end{theorem}

\begin{corollary}<cor:exp2>
  Under the conditions of \Cref{thm:exp2}, to reach the precision $\E*{f(x_K) - f^*}\leq \epsilon$, it is sufficient to perform the following number of iterations:
  \begin{equation}
    K = \tilde\cO\brs*{[\cR\cG_1]^4 + \brs{\cR\cG_0 / \epsilon}^2}.
  \end{equation}
\end{corollary}

\subsection{Discussion}

We conclude this section by summarizing the results obtained above and discussing them in the context of modern literature.

\textbf{Advanced vs classical averaging.}
Several advanced averaging schemes for convex optimization have been developed in recent years, including primal averaging \citep{nesterov2015quasi,tao2018primal}, online-to-batch conversion \citep{cutkosky2019anytime}, and other schemes \citep{defazio2024road}. Although these schemes are aesthetically appealing \citep{cutkosky2019anytime} and can improve practical performance \citep{defazio2024road}, there is currently no theoretical evidence that they can lead to improved convergence rates compared to the classical Polyak--Ruppert averaging \citep{polyak1990new,ruppert1988efficient} which is used, for instance, in \Cref{thm:SGD,thm:AdaGrad}; refer to \citet[eqs.~1 and~2]{defazio2024road} for the definition. To the best of our knowledge, \Cref{thm:exp2} is the first result of this kind: it shows that the averaging scheme in \cref{eq:wd} yields significantly improved convergence guarantees under \Cref{ass:G01}, which allows unbounded gradients.

\textbf{Exponential decay in Adam/AdamW.} The original Adam paper \citep{kingma2014adam} interprets the exponentially weighted accumulations in \Cref{line:exp:m,line:exp:v} of \Cref{alg:exp} as estimates of the first and second moments of the gradient obtained via exponential moving averages. Unfortunately, this interpretation does not provide meaningful insight into the convergence properties of Adam/AdamW.

An alternative interpretation was proposed by \citet{defossez2020simple} and later revisited by \citet{chezhegov2024clipping}. To analyze the convergence of Adam, they compare AdaGrad with its counterpart, RMSProp \citep{tieleman2012lecture}, which replaces AdaGrad's cumulative sum of squared gradients (see, e.g., \cref{eq:AdaGrad_params}) with an exponentially weighted sum. \citet[Section~4.3]{defossez2020simple} suggest that RMSProp relates to AdaGrad in the same way that constant-stepsize SGD relates to decaying-stepsize SGD. This interpretation yields meaningful convergence guarantees for RMSProp; however, these guarantees do not improve upon those for AdaGrad and therefore do not explain the benefits of exponential decay. In contrast, our analysis in \Cref{sec:exp} shows that understanding Adam or AdamW through the RMSProp--AdaGrad analogy may be inaccurate: in the unbounded-gradient setting covered by \Cref{ass:G01}, AdamW obtains the improved convergence guarantees in \Cref{thm:exp2} by essentially exploiting iterate stability, which is absent in AdaGrad and RMSProp.

\textbf{Non-convex analysis of Adam.}
Another, more accurate interpretation of Adam with exponential decay was proposed by \citet{ahn2024understanding,ahn2024adam}. They analyze the convergence of Adam for non-convex functions using the exponentiated online-to-non-convex conversion \citep{zhang2024random}. Their approach shares some similarities with ours, as we also use a form of conversion from convex optimization to online optimization, as described in \Cref{sec:wd}. However, their analysis does not show that Adam can improve upon the convergence rates of SGD with momentum \citep{cutkosky2023optimal}. One reason is that SGD with momentum is already optimal for globally Lipschitz functions \citep{cutkosky2023optimal}, which highlights the importance of realistic assumptions like \Cref{ass:G01}. Another reason is that their analysis is limited to general non-convex functions. Therefore, they can only obtain vague first-order stationarity guarantees, which typically depend on the initial function gap $f(x_0) - f^*$. This gap can be exponential due to \Cref{lem:M01}, making exponential factors in the resulting iteration complexities unavoidable. In contrast, in \Cref{thm:exp1,thm:exp2}, we obtain global convergence rates for convex functions with only logarithmic dependence on the function gap $\cF$, which is essential for avoiding exponential factors in the complexity and improving upon other methods. We highlight that such theoretical results for convex functions are relevant to machine learning practice, as convexity-like properties have been observed in neural-network training despite global nonconvexity \citep{kleinberg2018alternative,zhou2019sgd,tran2024reevaluating}. We also extend our results to the quasar-convex setting \citep{hinder2020near} in \Cref{sec:quasar}, bringing them even closer to practice.

It is also worth noting that general non-convex analysis can limit the possibility of understanding why Adam outperforms other methods even without assuming that the objective function is globally Lipschitz. For instance, \citet{li2023convergence} derived the convergence guarantees for Adam on non-convex functions under the generalized smoothness assumption, but acknowledged that their results do not improve upon the convergence rate of SGD \citep{li2023convex}. In contrast, in \Cref{sec:uni}, we derive convergence guarantees for \Cref{alg:exp} on convex functions under generalized H\"older smoothness, and these guarantees also improve upon the results for other methods.

\textbf{Non-monotone stepsizes.}
It is well known that AdaGrad has the limitation that its stepsize are always non-increasing \citep{defazio2022grad}. On the other hand, the stepsize in \Cref{line:exp:x} of \Cref{alg:exp} is proportional to $1/\mysqrt{v_k}$ and can be non-monotone due to the exponential decay in \Cref{line:exp:v}. Although common intuition suggests that such non-monotonicity can be beneficial because it allows the method to ``forget'' gradients that are too ``old'', to the best of our knowledge, there are no theoretical results supporting this intuition. Our results in \Cref{thm:exp1,thm:exp2} support this intuition: the stepsize in \Cref{alg:exp} may be exponentially small in early iterations, because the gradients may be exponentially large due to the bound in \cref{eq:MF_bound}; however, exponential decay allows the method to adapt to these large gradients and yields convergence rates without exponential factors, as discussed in \Cref{sec:exp}.

\newpage

\cite*{}

\bibliographystyle{unsrtnat}
\bibliography{references}

@book{rockafellar1997convex,
  title     = {Convex analysis},
  author    = {Rockafellar, R Tyrrell},
  volume    = {28},
  year      = {1997},
  publisher = {Princeton university press}
}

@article{hagood2006recovering,
  title     = {Recovering a function from a Dini derivative},
  author    = {Hagood, John W and Thomson, Brian S},
  journal   = {The American Mathematical Monthly},
  volume    = {113},
  number    = {1},
  pages     = {34--46},
  year      = {2006},
  publisher = {Taylor \& Francis}
}

@article{duchi2011adaptive,
  title   = {Adaptive subgradient methods for online learning and stochastic optimization.},
  author  = {Duchi, John and Hazan, Elad and Singer, Yoram},
  journal = {Journal of machine learning research},
  volume  = {12},
  number  = {7},
  year    = {2011}
}

@article{streeter2010less,
  title   = {Less regret via online conditioning},
  author  = {Streeter, Matthew and McMahan, H Brendan},
  journal = {arXiv preprint arXiv:1002.4862},
  year    = {2010}
}

@inproceedings{cutkosky2019anytime,
  title        = {Anytime online-to-batch, optimism and acceleration},
  author       = {Cutkosky, Ashok},
  booktitle    = {International conference on machine learning},
  pages        = {1446--1454},
  year         = {2019},
  organization = {PMLR}
}

@article{orabona2018scale,
  title     = {Scale-free online learning},
  author    = {Orabona, Francesco and P{\'a}l, D{\'a}vid},
  journal   = {Theoretical Computer Science},
  volume    = {716},
  pages     = {50--69},
  year      = {2018},
  publisher = {Elsevier}
}

@article{ahn2024understanding,
  title   = {Understanding Adam optimizer via online learning of updates: Adam is FTRL in disguise},
  author  = {Ahn, Kwangjun and Zhang, Zhiyu and Kook, Yunbum and Dai, Yan},
  journal = {arXiv preprint arXiv:2402.01567},
  year    = {2024}
}

@article{jiang2026adaptive,
  title   = {Adaptive Matrix Online Learning through Smoothing with Guarantees for Nonsmooth Nonconvex Optimization},
  author  = {Jiang, Ruichen and Mhammedi, Zakaria and Mohri, Mehryar and Mokhtari, Aryan},
  journal = {arXiv preprint arXiv:2602.08232},
  year    = {2026}
}

@article{kovalev2026optimal,
  title   = {Optimal Projection-Free Adaptive SGD for Matrix Optimization},
  author  = {Kovalev, Dmitry},
  journal = {arXiv preprint arXiv:2604.02505},
  year    = {2026}
}

@inproceedings{hinder2020near,
  title        = {Near-optimal methods for minimizing star-convex functions and beyond},
  author       = {Hinder, Oliver and Sidford, Aaron and Sohoni, Nimit},
  booktitle    = {Conference on learning theory},
  pages        = {1894--1938},
  year         = {2020},
  organization = {PMLR}
}

@book{clarke1990optimization,
  title     = {Optimization and nonsmooth analysis},
  author    = {Clarke, Frank H},
  year      = {1990},
  publisher = {SIAM}
}

@inproceedings{zhang2020complexity,
  title        = {Complexity of finding stationary points of nonconvex nonsmooth functions},
  author       = {Zhang, Jingzhao and Lin, Hongzhou and Jegelka, Stefanie and Sra, Suvrit and Jadbabaie, Ali},
  booktitle    = {International Conference on Machine Learning},
  pages        = {11173--11182},
  year         = {2020},
  organization = {PMLR}
}

@article{shapiro1990concepts,
  title     = {On concepts of directional differentiability},
  author    = {Shapiro, Alexander},
  journal   = {Journal of optimization theory and applications},
  volume    = {66},
  number    = {3},
  pages     = {477--487},
  year      = {1990},
  publisher = {Springer}
}

@inproceedings{glorot2011deep,
  title        = {Deep sparse rectifier neural networks},
  author       = {Glorot, Xavier and Bordes, Antoine and Bengio, Yoshua},
  booktitle    = {Proceedings of the fourteenth international conference on artificial intelligence and statistics},
  pages        = {315--323},
  year         = {2011},
  organization = {JMLR Workshop and Conference Proceedings}
}

@article{loshchilov2017decoupled,
  title   = {Decoupled weight decay regularization},
  author  = {Loshchilov, Ilya and Hutter, Frank},
  journal = {arXiv preprint arXiv:1711.05101},
  year    = {2017}
}

@article{kingma2014adam,
  title   = {Adam: A method for stochastic optimization},
  author  = {Kingma, Diederik P and Ba, Jimmy},
  journal = {arXiv preprint arXiv:1412.6980},
  year    = {2014}
}

@article{nemirovski2009robust,
  title     = {Robust stochastic approximation approach to stochastic programming},
  author    = {Nemirovski, Arkadi and Juditsky, Anatoli and Lan, Guanghui and Shapiro, Alexander},
  journal   = {SIAM Journal on optimization},
  volume    = {19},
  number    = {4},
  pages     = {1574--1609},
  year      = {2009},
  publisher = {SIAM}
}

@article{tran2024reevaluating,
  title   = {Reevaluating Theoretical Analysis Methods for Optimization in Deep Learning},
  author  = {Tran, Hoang and Zhang, Qinzi and Cutkosky, Ashok},
  journal = {arXiv preprint arXiv:2407.01825},
  year    = {2024}
}

@book{nesterov2018lectures,
  title     = {Lectures on convex optimization},
  author    = {Nesterov, Yurii and others},
  volume    = {137},
  year      = {2018},
  publisher = {Springer}
}

@inproceedings{cutkosky2023optimal,
  title        = {Optimal stochastic non-smooth non-convex optimization through online-to-non-convex conversion},
  author       = {Cutkosky, Ashok and Mehta, Harsh and Orabona, Francesco},
  booktitle    = {International Conference on Machine Learning},
  pages        = {6643--6670},
  year         = {2023},
  organization = {PMLR}
}

@book{shor2012minimization,
  title     = {Minimization methods for non-differentiable functions},
  author    = {Shor, Naum Zuselevich},
  year      = {2012},
  publisher = {Springer Science \& Business Media}
}

@book{goodfellow2016deep,
  title     = {Deep Learning},
  author    = {Goodfellow, Ian and Bengio, Yoshua and Courville, Aaron},
  publisher = {MIT Press},
  year      = {2016}
}

@inproceedings{he2015delving,
  title     = {Delving deep into rectifiers: Surpassing human-level performance on imagenet classification},
  author    = {He, Kaiming and Zhang, Xiangyu and Ren, Shaoqing and Sun, Jian},
  booktitle = {Proceedings of the IEEE international conference on computer vision},
  pages     = {1026--1034},
  year      = {2015}
}

@article{tieleman2012lecture,
  title   = {Lecture 6.5-rmsprop: Divide the gradient by a running average of its recent magnitude},
  author  = {Tieleman, Tijmen},
  journal = {COURSERA: Neural networks for machine learning},
  volume  = {4},
  number  = {2},
  pages   = {26},
  year    = {2012}
}

@inproceedings{bernstein2018signsgd,
  title        = {signSGD: Compressed optimisation for non-convex problems},
  author       = {Bernstein, Jeremy and Wang, Yu-Xiang and Azizzadenesheli, Kamyar and Anandkumar, Animashree},
  booktitle    = {International conference on machine learning},
  pages        = {560--569},
  year         = {2018},
  organization = {PMLR}
}

@article{robbins1951stochastic,
  title     = {A stochastic approximation method},
  author    = {Robbins, Herbert and Monro, Sutton},
  journal   = {The annals of mathematical statistics},
  pages     = {400--407},
  year      = {1951},
  publisher = {JSTOR}
}

@article{zhang2019gradient,
  title   = {Why gradient clipping accelerates training: A theoretical justification for adaptivity},
  author  = {Zhang, Jingzhao and He, Tianxing and Sra, Suvrit and Jadbabaie, Ali},
  journal = {arXiv preprint arXiv:1905.11881},
  year    = {2019}
}

@article{gorbunov2024methods,
  title   = {Methods for convex $(l\_0, l\_1) $-smooth optimization: Clipping, acceleration, and adaptivity},
  author  = {Gorbunov, Eduard and Tupitsa, Nazarii and Choudhury, Sayantan and Aliev, Alen and Richt{\'a}rik, Peter and Horv{\'a}th, Samuel and Tak{\'a}{\v{c}}, Martin},
  journal = {arXiv preprint arXiv:2409.14989},
  year    = {2024}
}

@article{polyak1969minimization,
  title     = {Minimization of unsmooth functionals},
  author    = {Polyak, Boris Teodorovich},
  journal   = {USSR Computational Mathematics and Mathematical Physics},
  volume    = {9},
  number    = {3},
  pages     = {14--29},
  year      = {1969},
  publisher = {Elsevier}
}

@article{bubeck2015convex,
  title     = {Convex optimization: Algorithms and complexity},
  author    = {Bubeck, S{\'e}bastien},
  journal   = {Foundations and trends in Machine Learning},
  volume    = {8},
  number    = {3-4},
  pages     = {231--357},
  year      = {2015},
  publisher = {Emerald Publishing limited}
}

@article{nesterov2015quasi,
  title     = {Quasi-monotone subgradient methods for nonsmooth convex minimization},
  author    = {Nesterov, Yu and Shikhman, Vladimir},
  journal   = {Journal of Optimization Theory and Applications},
  volume    = {165},
  number    = {3},
  pages     = {917--940},
  year      = {2015},
  publisher = {Springer}
}

@article{tao2018primal,
  title     = {Primal averaging: A new gradient evaluation step to attain the optimal individual convergence},
  author    = {Tao, Wei and Pan, Zhisong and Wu, Gaowei and Tao, Qing},
  journal   = {IEEE transactions on cybernetics},
  volume    = {50},
  number    = {2},
  pages     = {835--845},
  year      = {2018},
  publisher = {IEEE}
}

@article{ahn2024adam,
  title   = {Adam with model exponential moving average is effective for nonconvex optimization},
  author  = {Ahn, Kwangjun and Cutkosky, Ashok},
  journal = {Advances in Neural Information Processing Systems},
  volume  = {37},
  pages   = {94909--94933},
  year    = {2024}
}

@article{defazio2024road,
  title   = {The road less scheduled},
  author  = {Defazio, Aaron and Yang, Xingyu and Mehta, Harsh and Mishchenko, Konstantin and Khaled, Ahmed and Cutkosky, Ashok},
  journal = {Advances in Neural Information Processing Systems},
  volume  = {37},
  pages   = {9974--10007},
  year    = {2024}
}

@article{defossez2020simple,
  title   = {A simple convergence proof of adam and adagrad},
  author  = {D{\'e}fossez, Alexandre and Bottou, L{\'e}on and Bach, Francis and Usunier, Nicolas},
  journal = {arXiv preprint arXiv:2003.02395},
  year    = {2020}
}

@article{chezhegov2024clipping,
  title   = {Clipping improves adam-norm and adagrad-norm when the noise is heavy-tailed},
  author  = {Chezhegov, Savelii and Klyukin, Yaroslav and Semenov, Andrei and Beznosikov, Aleksandr and Gasnikov, Alexander and Horv{\'a}th, Samuel and Tak{\'a}{\v{c}}, Martin and Gorbunov, Eduard},
  journal = {arXiv preprint arXiv:2406.04443},
  year    = {2024}
}

@techreport{ruppert1988efficient,
  title       = {Efficient estimations from a slowly convergent Robbins-Monro process},
  author      = {Ruppert, David},
  year        = {1988},
  institution = {Cornell University Operations Research and Industrial Engineering}
}

@article{polyak1990new,
  title   = {New stochastic approximation type procedures},
  author  = {Polyak, Boris T},
  journal = {Automat. i Telemekh},
  volume  = {7},
  number  = {98-107},
  pages   = {2},
  year    = {1990}
}

@article{defazio2022grad,
  title   = {Grad-GradaGrad? A non-monotone adaptive stochastic gradient method},
  author  = {Defazio, Aaron and Zhou, Baoyu and Xiao, Lin},
  journal = {arXiv preprint arXiv:2206.06900},
  year    = {2022}
}

@article{zhang2024random,
  title   = {Random scaling and momentum for non-smooth non-convex optimization},
  author  = {Zhang, Qinzi and Cutkosky, Ashok},
  journal = {arXiv preprint arXiv:2405.09742},
  year    = {2024}
}

@article{zhou2019sgd,
  title   = {Sgd converges to global minimum in deep learning via star-convex path},
  author  = {Zhou, Yi and Yang, Junjie and Zhang, Huishuai and Liang, Yingbin and Tarokh, Vahid},
  journal = {arXiv preprint arXiv:1901.00451},
  year    = {2019}
}

@inproceedings{kleinberg2018alternative,
  title        = {An alternative view: When does SGD escape local minima?},
  author       = {Kleinberg, Bobby and Li, Yuanzhi and Yuan, Yang},
  booktitle    = {International conference on machine learning},
  pages        = {2698--2707},
  year         = {2018},
  organization = {PMLR}
}

@article{li2023convergence,
  title   = {Convergence of adam under relaxed assumptions},
  author  = {Li, Haochuan and Rakhlin, Alexander and Jadbabaie, Ali},
  journal = {Advances in Neural Information Processing Systems},
  volume  = {36},
  pages   = {52166--52196},
  year    = {2023}
}

@article{li2023convex,
  title   = {Convex and non-convex optimization under generalized smoothness},
  author  = {Li, Haochuan and Qian, Jian and Tian, Yi and Rakhlin, Alexander and Jadbabaie, Ali},
  journal = {Advances in Neural Information Processing Systems},
  volume  = {36},
  pages   = {40238--40271},
  year    = {2023}
}

@article{kovalev2025non,
  title   = {Non-Euclidean SGD for Structured Optimization: Unified Analysis and Improved Rates},
  author  = {Kovalev, Dmitry and Borodich, Ekaterina},
  journal = {arXiv preprint arXiv:2511.11466},
  year    = {2025}
}

@article{kovalev2025sgd,
  title   = {Sgd with adaptive preconditioning: Unified analysis and momentum acceleration},
  author  = {Kovalev, Dmitry},
  journal = {arXiv preprint arXiv:2506.23803},
  year    = {2025}
}

@inproceedings{karimireddy2019error,
  title        = {Error feedback fixes signsgd and other gradient compression schemes},
  author       = {Karimireddy, Sai Praneeth and Rebjock, Quentin and Stich, Sebastian and Jaggi, Martin},
  booktitle    = {International conference on machine learning},
  pages        = {3252--3261},
  year         = {2019},
  organization = {PMLR}
}

@article{crawshaw2022robustness,
  title   = {Robustness to unbounded smoothness of generalized signsgd},
  author  = {Crawshaw, Michael and Liu, Mingrui and Orabona, Francesco and Zhang, Wei and Zhuang, Zhenxun},
  journal = {Advances in neural information processing systems},
  volume  = {35},
  pages   = {9955--9968},
  year    = {2022}
}

@article{lobanov2026avoiding,
  title   = {Avoiding Bias in Clipped SGD for Overparameterized Models under Generalized Smoothness},
  author  = {Lobanov, Aleksandr and Koloskova, Anastasia},
  journal = {arXiv preprint arXiv:2605.14800},
  year    = {2026}
}

@article{orabona2023normalized,
  title   = {Normalized gradients for all},
  author  = {Orabona, Francesco},
  journal = {arXiv preprint arXiv:2308.05621},
  year    = {2023}
}

@article{nesterov2015universal,
  title     = {Universal gradient methods for convex optimization problems},
  author    = {Nesterov, Yu},
  journal   = {Mathematical Programming},
  volume    = {152},
  number    = {1},
  pages     = {381--404},
  year      = {2015},
  publisher = {Springer}
}

@article{rodomanov2024universal,
  title   = {Universal gradient methods for stochastic convex optimization},
  author  = {Rodomanov, Anton and Kavis, Ali and Wu, Yongtao and Antonakopoulos, Kimon and Cevher, Volkan},
  journal = {arXiv preprint arXiv:2402.03210},
  year    = {2024}
}

@article{rodomanov2024universality,
  title   = {Universality of adagrad stepsizes for stochastic optimization: Inexact oracle, acceleration and variance reduction},
  author  = {Rodomanov, Anton and Jiang, Xiaowen and Stich, Sebastian},
  journal = {Advances in Neural Information Processing Systems},
  volume  = {37},
  pages   = {26770--26813},
  year    = {2024}
}

\newpage
\appendix

\section{Lower Bounds}\label{sec:lower}

The complexity results for SGD and AdaGrad-Norm established in \Cref{sec:stoch} for stochastic convex optimization under \Cref{ass:G01} are quite pessimistic. This naturally raises the question of whether the exponential factors in the complexities obtained in \Cref{thm:SGD,thm:AdaGrad} are unavoidable or merely artifacts of a loose convergence analysis. In \Cref{thm:lower_SGD,thm:lower_AdaGrad}, we provide lower bounds showing that these factors are indeed unavoidable. Therefore, in view of the improved complexity bound in \Cref{cor:exp2}, we conclude that scalar-stepsize AdamW with clipped updates provably outperforms SGD and AdaGrad in this setting.

\begin{theorem}<thm:lower_SGD>
  Let the constants $\cG_0,\cG_1,\cR,\epsilon > 0$ satisfy the following conditions:
  \begin{equation}\label{eq:lower_SGD_params}
    \epsilon \leq \brs{\cG_0/\cG_1}\exp\brr{\tfrac{1}{8}\cR\cG_1},\qquad \cR\cG_1 \geq 8.
  \end{equation}
  Let $\eta_k = \eta>0$ be an arbitrary stepsize, and let $\ox_K = \frac{1}{K+1}\sum_{k=0}^K x_k$, where $x_0,\dots,x_k$ are defined by \cref{eq:SGD}.
  There exists a convex function $f(x)\colon \Rd \to \R$ and corresponding stochastic gradient oracle $\nabla_\xi f(x) \colon \Rd \to \Rd$ satisfying \Cref{ass:G01}, such that at least the following number of iterations is required to reach the precision $\E{\min\brf{f(x_K), f(\ox_K)} - f^*} \leq \epsilon$:
  \begin{equation}
    K = \Omega\brs*{\exp\brr{\Omega\brs{\cR\cG_1}} \cdot \brs{\cR\cG_0/\epsilon}}.
  \end{equation}
\end{theorem}

\begin{theorem}<thm:lower_AdaGrad>
  Let the constants $\cG_0,\cG_1,\cR,\epsilon > 0$ satisfy the following conditions:
  \begin{equation}\label{eq:lower_AdaGrad_params}
    \epsilon \leq \brs{\cG_0/\cG_1}\exp\brr{\tfrac{1}{64}\cR\cG_1},\qquad \cR\cG_1 \geq 32.
  \end{equation}
  Let $\eta>0$ be an arbitrary parameter, and let $\ox_K = \frac{1}{K+1}\sum_{k=0}^K x_k$, where $x_0,\dots,x_k$ are defined by \cref{eq:SGD} with the AdaGrad-Norm stepsizes:
  \begin{equation}
    \eta_k = \eta / \mysqrt{\tsum_{i=0}^{k} \sqn{\nabla_{\xi_i} f(x_i)}}
  \end{equation}
  There exists a convex function $f(x)\colon \Rd \to \R$ and corresponding stochastic gradient oracle $\nabla_\xi f(x) \colon \Rd \to \Rd$ satisfying \Cref{ass:G01}, such that at least the following number of iterations is required to reach the precision $\E{\min\brf{f(x_K), f(\ox_K)} - f^*} \leq \epsilon$:
  \begin{equation}
    K = \Omega\brs*{\exp\brr{\Omega\brs{\cR\cG_1}}}.
  \end{equation}
\end{theorem}

\newpage

\section{Universality: Generalized H\"older Smooth Stochastic Convex Problems}\label{sec:uni}

\subsection{Generalized H\"older Smoothness}\label{sec:SL01}

In this section, we assume that the objective function $f(x)$ is differentiable and convex. Next, we introduce \Cref{ass:SL01}, which we call stochastic generalized H\"older smoothness.

\begin{assumption}<ass:SL01>
  There exists a stochastic estimator $\nabla_\xi f(x)$ of the gradient $\nabla f(x)$, where $\xi \sim \cD$ is a random variable sampled from the distribution $\cD$. The gradient estimator $\nabla_\xi f(x)$ is unbiased, i.e., $\E[\xi \sim \cD]{\nabla_\xi f(x)} = \nabla f(x)$, and allows the decomposition $\nabla_\xi f(x) = u_\xi(x) + v_\xi(x)$, satisfying
  \begin{equation}
    \E[\xi \sim \cD]{\sqn{u_\xi(x)}}\leq \sigma^2 + \cL_0^{\smash[t]{\frac{2}{1+\nu}}}\brs{f(x) - f^*}^{\smash[t]{\frac{2\nu}{1+\nu}}}
    \quad\text{and}\quad
    \norm{v_\xi(x)} \las \cL_1\brs{f(x) - f^*},
  \end{equation}
  where $\nu \in [0,1]$ and $\sigma,\cL_0,\cL_1 \geq 0$, and where $\sigma = 0$ whenever $\nu=0$.
\end{assumption}

We motivate \Cref{ass:SL01} by describing the following special cases.
\begin{enumerate}[label=\bf(\roman*)]
  \item In the case $(\nu,\sigma,\cL_0,\cL_1) = (0,0,\cG_0,\cG_1)$, \Cref{ass:SL01} matches \Cref{ass:G01}, which is a stochastic variant of the generalized Lipschitzness discussed in \Cref{sec:M01}.
  \item\label[case]{case:L01} In the noiseless case with $(\nu,\sigma) = (1,0)$, \Cref{ass:SL01} recovers the generalization of the $(\cL_0,\cL_1)$-smoothness up to universal constants; see \cref{eq:gorbunov} in \Cref{sec:M01}.
  \item\label[case]{case:Holder} In the noiseless case with $(\sigma,\cL_0,\cL_1) = (0,\cL,0)$, \Cref{ass:SL01} is implied by the standard $(\nu,\cL)$-H\"older smoothness assumption \citep{nesterov2015universal} up to universal constants; see Corollary~1 of \citet{orabona2023normalized}.
  \item Assuming the stochastic setting with globally bounded variance of the stochastic gradient and allowing $\sigma \neq 0$ in \cref{case:L01,case:Holder} recovers the stochastic variants of the corresponding assumptions.
\end{enumerate}
By analyzing these cases, one can conclude that \Cref{ass:SL01} interpolates between the stochastic variants of the generalized Lipschitzness and generalized smoothness, just like the standard H\"older smoothness interpolates between uniform Lipschitzness and uniform smoothness.

We also establish an important property of the generalized H\"older smooth functions in the following \Cref{lem:SL01}. It can be seen as the extension of \Cref{lem:M01} to the stochastic generalized H\"older smoothness setting.

\begin{lemma}<lem:SL01>
  Let  \Cref{ass:SL01} hold and let the function $h(x)\colon \Rd \to \R$ be defined as follows:
  \begin{equation}\label{eq:h}
    h(x) = \brs{\beta + f(x) - f^*}^{\smash[t]{\frac{1}{1+\nu}}}
    ,\quad\text{where}\;\; \beta = \min\brf*{\brs{\sigma^{1+\nu}/\cL_0}^{\smash{\frac{1}{\nu}}}, \sigma/\cL_1}
  \end{equation}
  Then, the following inequality holds for all $x,x' \in \Rd$:
  \begin{equation}
    \abs{h(x') - h(x)}^{1+\nu} \leq \cO\brs*{\brr*{ \cL_0 + \brs{\cL_1 h(x)}^{1+\nu}}\exp\brr*{\cL_1\norm{x'-x}}\norm{x'-x}^{1+\nu}}.
  \end{equation}
\end{lemma}

\subsection{Convergence Analysis}

In this section, we analyze the convergence of the iterates generated by \cref{eq:wd} under \Cref{ass:regret,ass:SL01}. Our goal is to extend the two-stage convergence guarantee of \Cref{thm:exp2} to the stochastic generalized H\"older smooth setting. We first prove the corresponding generalized regret bound in \Cref{lem:regret_uni}, which extends \cref{eq:regret_bound}, and then derive the main convergence result and explicit iteration complexity bound in \Cref{thm:uni} and \Cref{cor:uni}.

\begin{lemma}<lem:regret_uni>
  Under the conditions of \Cref{thm:uni}, the following inequality holds for all $K \geq T$:
  \begin{equation}
    \begin{aligned}
      \E{\reg_K}
      &\leq
      \tfrac{1}{2}\tsum_{k=1}^K \alpha_k\pi_k\E*{f(x_k) - f^*}
      +\cO\brs*{
        \cR\sigma\mysqrt{\tsum_{k=1}^K \alpha_k^2\pi_k^2}
        +
        T\cR\sigma\tsum_{k=1}^K\alpha_k^2\pi_k
      }
      \\&
      +\cO\brs*{
        \cR^{1+\nu}\cL_0\brs*{\tsum_{k=1}^K [\alpha_k\pi_k]^{\frac{2}{1-\nu}}}^{\frac{1-\nu}{2}}
        +
        \brs{T\cR}^{1+\nu}\cL_0\tsum_{k=1}^K \alpha_k^{2+\nu}\pi_k
      }.
    \end{aligned}
  \end{equation}
\end{lemma}

\begin{theorem}<thm:uni>
  Let \Cref{ass:regret,ass:SL01} hold, and assume $\cR\cL_1 \geq 1$. Let $T = \ceil*{\hC\brs{\cR\cL_1}^2}$, where $\hC$ is a sufficiently large universal constant, and let $\pi_0 = 1$. Let $\alpha_k$ be defined as follows:
  \begin{equation}\label{eq:alpha_uni}
    \resizebox{0.93\textwidth}{!}{
      \hspace{-0.7em}
      $
      \displaystyle
      \frac{1}{\alpha_k} =
      \begin{cases}
        T\cR\cL_1 & k \leq S T \\
        \max\brf*{
          T\cR\cL_1,
          \brs*{{\cR\sigma}/{\epsilon}}^2,
          \brs*{{\cR^{1+\nu}\cL_0}/{\epsilon}}^{\smash[t]{\frac{2}{1+\nu}}},
          \brs{{T\cR\sigma}/{\epsilon}},
          \brs*{{[T\cR]^{1+\nu}\cL_0}/{\epsilon}}^{\smash[b]{\frac{1}{1+\nu}}}
        }  & k > S T
      \end{cases}.
      $
    }
  \end{equation}
  Let $\hF = \smash{\cL_0/\cL_1^{1+\nu} + \sigma/\cL_1}$, and let $K,S \in \N$ be defined as follows:
  \begin{equation}\label{eq:KS_uni}
    K = TS + \ceil*{\brs{2/\alpha_{TS+1}} \ln\brs*{e+ \brs{\hF/\epsilon}}}
    ,\quad
    S = \ceil*{\brs{2/(\alpha_1T)} \ln \brs*{e+ \brs{\cF/\hF}}}.
  \end{equation}
  Then, the inequality $\E*{f(x_K) - f^*}\leq \cO\brs{\epsilon}$ holds.
\end{theorem}

\begin{corollary}<cor:uni>
  Under the conditions of \Cref{thm:uni}, to reach the precision $\E*{f(x_K) - f^*}\leq \epsilon$, it is sufficient to perform the following number of iterations:
  \begin{equation}
    K = \tilde\cO\brs*{[\cR\cL_1]^4 + \brs{\cR^{1+\nu}\cL_0 / \epsilon}^{\smash[t]{\frac{2}{1+\nu}}} + \brs{\cR\sigma/ \epsilon}^2}.
  \end{equation}
\end{corollary}

As with the corresponding result for the generalized Lipschitz setting in \Cref{sec:exp}, the guarantee in \Cref{cor:uni} admits an informal two-phase interpretation: fast linear convergence in the first phase up to the precision $\hepsilon = \cO\brs{\max\brf{\cR\sigma/\brs{\cR\cL_1}^2,\cR^{1+\nu}\cL_0 / \brs{\cR\cL_1}^{2(1+\nu)}}}$, followed by the phase with the sublinear convergence rate, which matches the standard convergence rate of SGD for convex $(\nu,\cL_0)$-H\"older smooth functions with bounded stochastic gradient variance \citep{rodomanov2024universal}. It is worth highlighting that concurrent work by \citet{lobanov2026avoiding} obtains a similar first-phase linear convergence result in the generalized smooth case $\nu=1$; however, their guarantees rely on significantly stronger or even unrealistic assumptions, such as increasing batch sizes or the interpolation condition.

The result in \Cref{thm:uni} also implies that, up to logarithmic factors, the schedule of $\alpha_k$ can be parameterized by the overall number of iterations $K$ and the constants $\cR$ and $\cL_1$, without prior knowledge of the H\"older exponent $\nu$. This parameter independence across smoothness levels is called universality \citep{nesterov2015universal} and is also a feature of many other adaptive algorithms such as AdaGrad \citep{rodomanov2024universality}.

\newpage

\section{Extension to Quasar-Convex Setting}\label{sec:quasar}

\subsection{Quasar Convexity and Generalized Gradients}

In this section, we no longer assume the convexity of the objective function $f(x)$. Instead, we assume that the function $f(x)$ is $\gamma$-quasar convex for $\gamma \in (0,1]$, that is, the following inequality holds:
\begin{equation}\label{eq:quasar}
  f(t x^* + (1-t)x) \leq \gamma t f^* + (1-\gamma t)f(x)
  \;\;\text{for all}\;\;x \in \Rd,\; t \in [0,1].
\end{equation}
Refer to \citet{hinder2020near} for more details on quasar-convex functions.
Consequently, the function $f(x)$ may no longer be subdifferentiable at every point $x \in \Rd$. However, we can define the generalized subdifferential $\gpartial f(x) = \partial \phi_x(0)$, where $\phi_x(w) = f\gprime (x;w)$. Here, $f\gprime (x;w)$ is the generalized directional derivative \citep{clarke1990optimization}, which is a convex function with respect to the second argument and is defined for all $x, w \in \Rd$ as follows:
\begin{equation}\label{eq:gprime}
  f\gprime (x;w) = \tlimsup_{x'\to x, t\to +0} \tfrac{1}{t}\brs{f(x' + tw) - f(x')}.
\end{equation}
Note that both $\gpartial f(x)$ and $f\gprime (x;w)$ exist because the function $f(x)$ is assumed to be locally Lipschitz. Furthermore, we assume that the objective function $f(x)$ is directionally differentiable in the sense of Hadamard \citep{shapiro1990concepts}, that is, the following limit exists for all $x,w \in \Rd$:
\begin{equation}
  f'(x;w) = \tlim_{w'\to w, t \to +0} \tfrac{1}{t}\brs{f(x + tw') - f(x)}.
\end{equation}
Now, we are ready to describe the generalized gradient oracle that we will further use in this section. Specifically, we assume that there exists $\nabla f(x;w) \in \Rd$, which we refer to as the generalized gradient and which satisfies the following properties for all $x, w \in \Rd$:
\begin{equation}\label{eq:ggrad}
  \nabla f(x;w) \in \gpartial f(x)
  \quad\text{and}\quad
  \<w,\nabla f(x;w)> = f'(x;w).
\end{equation}
The existence of such a generalized gradient is guaranteed by \citet[Lemma~4]{zhang2020complexity}. Refer to \citet[Section~2]{zhang2020complexity} for more details on generalized gradients. Moreover, in the following \Cref{lem:quasar}, we show that the generalized gradient can be used to define quasar-convexity via the first-order criterion, just as it is for differentiable functions.
\begin{lemma}<lem:quasar>
  The inequality in \cref{eq:quasar} is equivalent to the following condition:
  \begin{equation}\label{eq:quasar_grad}
    \gamma\brs{f(x) - f^*}\leq \<x-x^*, \nabla f(x;w)>
    \quad\text{for all}\;\;
    x,w \in \Rd.
  \end{equation}
\end{lemma}

We also introduce the following \Cref{ass:qG01}, which is an extension of \Cref{ass:G01} for the generalized gradient setting.

\begin{assumption}<ass:qG01>
  There exists a stochastic estimator $\nabla_\xi f(x;w)$ of the generalized gradient $\nabla f(x;w)$, where $\xi \sim \cD$ is a random variable sampled from the distribution $\cD$. The generalized gradient estimator $\nabla_\xi f(x;w)$ is unbiased, i.e., $\E[\xi \sim \cD]{\nabla_\xi f(x;w)} = \nabla f(x;w)$, and allows the decomposition $\nabla_\xi f(x;w) = u_\xi(x;w) + v_\xi(x;w)$, satisfying
  \begin{equation}
    \E[\xi \sim \cD]{\sqn{u_\xi(x;w)}}\leq \cG_0^2
    \quad\text{and}\quad
    \norm{v_\xi(x;w)} \las \cG_1\brs{f(x) - f^*},
    \quad\text{where\;\;}
    \cG_0,\cG_1 > 0.
  \end{equation}
\end{assumption}

\begin{lemma}<rem:qGM01>
  \Cref{ass:qG01} implies \Cref{ass:M01} with $\cM_0 = \cG_0$ and $\cM_1 = \cG_1$.
\end{lemma}

\subsection{Algorithm and Convergence Analysis}

Similar to \Cref{sec:alg}, we consider the iterations in \cref{eq:wd}, where the iterates $z_k \in Q_\cR$ are obtained by an online learning algorithm for bounding the regret $\reg_K$ in \cref{eq:regret}, satisfying \Cref{ass:regret}, but with a different definition of $g_k$:
\begin{equation}\label{eq:g_quasar}
  g_k = \alpha_k\pi_k\nabla_{\xi_k} f(\ox_k;x_k - x_{k-1}),\quad \ox_k = x_{k-1} + \zeta_k(x_k - x_{k-1}),
\end{equation}
where the samples $\xi_k \sim \cD$ and $\zeta_k \sim \text{u.a.r.}([0,1])$ are independent of each other and of $z_1,\dots,z_k$. Here, the generalized stochastic gradient is computed at the point $\ox_k$ sampled uniformly at random on the segment $[x_k, x_{k-1}]$ which is inspired by the approach of \citet{zhang2020complexity,cutkosky2023optimal} for non-smooth non-convex optimization. Below, we derive the main convergence result and explicit iteration complexity bound in \Cref{thm:quasar} and \Cref{cor:quasar}.

\begin{theorem}<thm:quasar>
  Let \Cref{ass:qG01,ass:regret} hold and let $\cR\cG_1 \geq 1$. Let $\pi_k$ be defined as follows:
  \begin{equation}\label{eq:pi_quasar}
    \pi_0 = 1,\qquad (1-\gamma\alpha_k)\pi_k = \pi_{k-1}.
  \end{equation}
  Let $T = \smash{\ceil*{\hC\brs{\cR\cG_1/\gamma}^2}}$ for a sufficiently large universal constant $\hC$, and let $\alpha_k$ be defined as follows:
  \begin{equation}\label{eq:alpha_quasar}
    1/\alpha_k =
    \begin{cases}
      T\cR\cG_1 & k \leq S T \\
      \max\brf*{T\cR\cG_1, [1/\gamma] \cdot \brs{\cR\cG_0/\epsilon}^2, [1/\gamma] \cdot \brs{T\cR\cG_0/\epsilon}}  & k > S T
    \end{cases}.
  \end{equation}
  Let $K,S \in \N$ be defined as follows:
  \begin{equation}\label{eq:KS_quasar}
    K = TS + \ceil*{\brs{2/(\gamma\alpha_{TS+1})} \ln \brs{e + \brs{\cG_0/(\gamma\cG_1\epsilon)}}}
    ,\quad
    S = \ceil*{\brs{2/(\gamma\alpha_1T)} \ln \brs{e + \brs{\gamma\cF\cG_1/\cG_0}}}.
  \end{equation}
  Then, the inequality $\E*{f(x_K) - f^*}\leq \cO\brs{\epsilon}$ holds.
\end{theorem}
\begin{corollary}<cor:quasar>
  Under the conditions of \Cref{thm:quasar}, to reach the precision $\E*{f(x_K) - f^*}\leq \epsilon$, it is sufficient to perform the following number of iterations:
  \begin{equation}
    K = \tilde\cO\brs*{[1/\gamma]^6 \cdot [\cR\cG_1]^4 + [1/\gamma]^2 \cdot\brs{\cR\cG_0 / \epsilon}^2}.
  \end{equation}
\end{corollary}

\newpage

\section{Beyond Scalar stepsizes}\label{sec:preconditioning}

\begin{algorithm}[t]
  \caption{AdamW/LeonW (Diagonal Preconditioning)}
  \label{alg:diag}
  \begin{algorithmic}[1]
    \State $m_0 = 0$, $v_0 = \delta \ones$, $x_1 = 0$
    \For{$k=1,\dots,K-1$}
    \State $\hg_k = \nabla_{\xi_k} f(x_k)$, where $\xi_k \sim \cD$
    \State $m_k = (1-\alpha_k)m_{k-1}+ \alpha_k \hg_k$,\;\; $v_k = (1-\alpha_k)^2v_{k-1} +  \alpha_k^2 \brr{\hg_k\odot \hg_k}$
    \State Option I (AdamW): $x_{k+1} = (1-\alpha_{k+1})x_k - \alpha_{k+1}\cR \clip[\infty]{m_k \oslash \mysqrt{v_k}}$
    \label{line:diag:adam}
    \State Option II (LeonW): $x_{k+1} = (1-\alpha_{k+1})x_k - \alpha_{k+1}\cR \brr{m_k \oslash \mysqrt{m_k\odot m_k + v_k}}$
    \label{line:diag:leon}
    \EndFor
  \end{algorithmic}
\end{algorithm}

\begin{algorithm}[t]
  \caption{LeonW (Matrix Preconditioning)}
  \label{alg:matrix}
  \begin{algorithmic}[1]
    \State $M_0 = 0$, $V_0 = \delta I$, $X_1 = 0$
    \For{$k=1,\dots,K-1$}
    \State $\hG_k = \nabla_{\xi_k} f(X_k)$, where $\xi_k \sim \cD$
    \State $M_k = (1-\alpha_k)M_{k-1}+ \alpha_k \hG_k$,\;\; $V_k = (1-\alpha_k)^2V_{k-1} +  \alpha_k^2 \hG_k\hG_k^\top$
    \State $X_{k+1} = (1-\alpha_{k+1})X_k - \alpha_{k+1} \cR \mysqrt{\brr{M_kM_k^{\top} + V_k}^{-1}}M_k$
    \EndFor
  \end{algorithmic}
\end{algorithm}

In this section, we show that the results obtained in the main part of the paper are not specific to scalar stepsizes. First, we consider \Cref{alg:diag} with diagonal preconditioning, which has two options. Option~I in \Cref{line:diag:adam} solves the online linear optimization problem in \cref{eq:regret} using SOLO-FTRL, which leads to AdamW with the standard diagonal preconditioning and clipped updates. Option~II in \Cref{line:diag:leon} solves the problem in \cref{eq:regret} using a diagonal variant of Leon and leads to a slight variation of AdamW without clipping, which we call LeonW. We also consider a variant of LeonW with matrix preconditioning, \Cref{alg:matrix}, which solves the problem in \cref{eq:regret} using the matrix variant of Leon.

The convergence analysis for these algorithms is very similar to that in \Cref{sec:alg}, but it can exploit structural properties of the objective function, such as different coordinate scales (\Cref{alg:diag}) or low-rank gradients (\Cref{alg:matrix}). We will provide a unified analysis based on the framework of \citet{kovalev2025non,kovalev2025sgd} in the next version of the paper.

\newpage

\section[Proofs for \crtCref{sec:M01}]{Proofs for \Cref{sec:M01}}

\proofsubsection{lem:M01}

\textbf{\Cref{eq:M01} implies \Cref{ass:M01}.} We can upper $\limsup_{u \to 0}\frac{1}{\norm{u}}\abs{f(x + u) - f(x)}$ as follows:
\begin{align*}
  \limsup_{u \to 0}\tmfrac{1}{\norm{u}}\abs{f(x + u) - f(x)}
  &\at{uses \cref{eq:M01}}\leq
  \limsup_{u \to 0}\brr*{\cM_0 + \cM_1\brs*{f(x) - f^*}}\exp\brr{\cM_1\norm{u}}
  \\&=
  \cM_0 + \cM_1\brs*{f(x) - f^*},
\end{align*}
where \annotate.

\textbf{\Cref{ass:M01} implies \Cref{eq:M01}.}
Let the function $p(t)\colon \R_+ \to \R$ be defined as follows:
\begin{align*}
  p(t) = f(x + t(x'-x)) - f^*,
\end{align*}
let the function $q(t)\colon \R_+ \to \R$ be defined as follows:
\begin{align*}
  q(t) =  \norm{x'-x}\smash{\mint_0^t} \brr{\cM_0 + \cM_1p(\tau)}\df \tau,
\end{align*}
and let the function $r(t)\colon \R_+ \to \R$ be defined as follows:
\begin{align*}
  r(t) = \abs{p(t) - p(0)} - q(t).
\end{align*}
First, we need to show that $r(t) \leq 0$ for all $t \in \R_+$. Observe that $r(0) = 0$. Hence, it sufficient to show that the function $r(t)$ is non-increasing. Let $r^+(t)$ be the upper Dini derivative of the function $r(t)$, which is defined as follows:
\begin{align*}
  r^+(t) = \limsup_{t' \downarrow t}\tmfrac{1}{t' - t}\brs{r(t') - r(t)}.
\end{align*}
We can upper-bound $r^+(t)$ as follows:
\begin{align*}
  r^+(t)
  &\at{uses the definition of the functions $r(t)$ and $p(t)$}=
  \limsup_{t' \downarrow t}\tmfrac{1}{t' - t}\brs{\abs{p(t') - p(0)} - q(t') - \abs{p(t) - p(0)} + q (t)}
  \\&\leq
  \limsup_{t' \downarrow t}
  \tmfrac{1}{t' - t}\brs{\abs{p(t') - p(0)} - \abs{p(t) - p(0)}}
  -
  \lim_{t' \to t}
  \tmfrac{1}{t' - t}\brs{q(t') - q(t)}
  \\&\leq
  \limsup_{t' \downarrow t}
  \tmfrac{1}{t' - t}\abs{p(t') - p(t)}
  -
  \lim_{t' \to t}
  \tmfrac{1}{t' - t}\brs{q(t') - q(t)}
  \\&\at{uses the definition of the function $q(t)$ and the fact that the function $p(t)$ is continuous}=
  \limsup_{t' \downarrow t}
  \tmfrac{1}{t' - t}\abs{p(t') - p(t)}
  -
  \norm{x'-x}\brr{\cM_0 + \cM_1p(t)}
  \\&\leq
  \limsup_{t' \to t}
  \tmfrac{1}{\abs{t' - t}}\abs{p(t') - p(t)}
  -
  \norm{x'-x}\brr{\cM_0 + \cM_1p(t)}
  \\&\at{uses the definition of the function $p(t)$ and \Cref{ass:M01}}\leq
  0,
\end{align*}
where \annotate. Hence, by the Monotonicity Theorem \citep{hagood2006recovering}, we conclude that the function $r(t)$ is non-increasing, which implies $r(t) \leq 0$ for all $t \in \R_+$.

Next, since the function $p(t)$ is continuous, the derivative $\dot{q}(t)$ of the function $q(t)$ exists and can be upper-bounded as follows:
\begin{align*}
  \dot{q}(t)
  &\at{uses the definition of $q(t)$}=
  \brr*{\cM_0 + \cM_1 p(t)}\norm{x'-x}
  \\&\leq
  \brr*{\cM_0 + \cM_1 \brs*{p(0)+ \abs{p(t) - p(0)}}}\norm{x'-x}
  \\&\at{uses the inequality $r(t) \leq 0$}\leq
  \brr{\cM_0 + \cM_1 \brs{p(0)+ q(t)}}\norm{x'-x},
\end{align*}
where \annotate.
After rearranging, we obtain the following inequality:
\begin{align*}
  \dot{l}(t) \leq \cM_1\norm{x'-x},
\end{align*}
where $l(t) = \ln \brs{\cM_0 + \cM_1\brs*{p(0) + q(t)}}$.
By integrating the inequality, we obtain
\begin{align*}
  \ln \brs{\cM_0 + \cM_1\brs{p(0) + q(1)}} - \ln \brs{\cM_0 + \cM_1\brs{p(0) + q(0)}} \leq \cM_1\norm{x'-x}.
\end{align*}
By taking the exponent of both sides and rearranging, we obtain
\begin{align*}
  \cM_0 + \cM_1\brs{p(0) + q(1)} &\leq \brs*{\cM_0 + \cM_1\brs{p(0) + q(0)}}\exp\brr{\cM_1\norm{x'-x}}
  \\&\at{uses the definition of $q(t)$}=
  \brs*{\cM_0 + \cM_1p(0)}\exp\brr{\cM_1\norm{x'-x}}
\end{align*}
where \annotate.
After rearranging, we obtain the following
\begin{align*}
  q(1) &\leq \brs{\cM_0 + \cM_1p(0)}\tmfrac{1}{\cM_1}\brs{\exp\brr{\cM_1\norm{x'-x}} - 1}
  \\&\at{uses the convexity of the exponent}\leq
  \brs{\cM_0 + \cM_1p(0)}\exp\brr{\cM_1\norm{x'-x}}\norm{x'-x},
\end{align*}
where \annotate.
Finally, we use the inequality $r(1) \leq 0$.\qed

\newpage

\section[Proofs for \crtCref{sec:basic_alg}]{Proofs for \Cref{sec:basic_alg}}

\proofsubsection{thm:GD}

We define $\eta = \eta_k$ and proceed with the following inequality:
\begin{align*}
  \tfrac{1}{2}\sqn{x_{k+1} - x^*}
  &\at{uses the iterations in \cref{eq:GD} and the fact that $x^* \in Q_\cR$}=
  \tfrac{1}{2}\sqn{\proj[Q_\cR]{x_k - \eta\nabla f(x_k)} - \proj[Q_\cR]{x^*}}
  \\&\at{uses the non-expansiveness of the projection}\leq
  \tfrac{1}{2}\sqn{x_k - x^* - \eta\nabla f(x_k)}
  \\&=
  \tfrac{1}{2}\sqn{x_k - x^*}
  -\eta\<\nabla f(x_k),x_k - x^*>
  +\tfrac{1}{2}\eta^2\sqn{\nabla f(x_k)}
  \\&\at{uses the convexity of the function $f(x)$}\leq
  \tfrac{1}{2}\sqn{x_k - x^*}
  -\eta\brs*{f(x_k)  - f^*}
  +\tfrac{1}{2}\eta^2\sqn{\nabla f(x_k)}
  \\&\at{uses \Cref{ass:M01}}\leq
  \tfrac{1}{2}\sqn{x_k - x^*}
  -\eta\brs*{f(x_k)  - f^*}
  +\tfrac{1}{2}\eta^2\brr{\cM_0 + \cM_1\brs{f(x_k) - f^*}}^2
  \\&\at{uses the definitions in \cref{eq:MF} and the fact that $\norm{x_k}\leq\cR$}\leq
  \tfrac{1}{2}\sqn{x_k - x^*}
  -\eta\brs*{f(x_k)  - f^*}
  +\eta^2\cM_0^2
  +\eta^2\cM_1^2\cF\brs{f(x_k) - f^*}
  \\&\leq
  \tfrac{1}{2}\sqn{x_k - x^*}
  -\tfrac{1}{2}\eta\brs*{f(x_k)  - f^*}
  +\eta^2\cM_0^2
\end{align*}
where \annotate. Next, we can obtain the following:
\begin{align}
  \nonumber
  f(\ox_K) - f^*
  &\at{uses the convexity of the function $f(x)$ and the definition of $\ox_K$}\leq
  \tmfrac{1}{K+1}\msum_{k=0}^{K}\brs{f(x_k) - f^*}
  \at{uses the inequality above and the fact that $\norm{x^*} \leq \cR$ and $x_0 = 0$}\leq
  \tmfrac{\cR^2}{\eta(K+1)} + 2\eta\cM_0^2
  \\&\at{use the definitions in \cref{eq:GD_params}}\leq
  \cO\brs*{\tmfrac{\cR\cM_0}{\sqrt{K+1}} +  \tmfrac{\brs{\cR\cM_1}^2\cF}{K+1}}
  \at{use the definitions in \cref{eq:GD_params}}\leq
  \cO\brs{\epsilon},\label{eq:GD_rate}
\end{align}
where \annotate.\qed

\proofsubsection{thm:normalized_polyak}

We start with the following inequality:
\begin{align*}
  \tfrac{1}{2}\sqn{x_{k+1} - x^*}
  &\at{uses the iterations in \cref{eq:GD} and the fact that $x^* \in Q_\cR$}=
  \tfrac{1}{2}\sqn{\proj[Q_\cR]{x_k - \eta_k\nabla f(x_k)} - \proj[Q_\cR]{x^*}}
  \\&\at{uses the non-expansiveness of the projection}\leq
  \tfrac{1}{2}\sqn{x_k - x^* - \eta_k\nabla f(x_k)}
  \\&=
  \tfrac{1}{2}\sqn{x_k - x^*} - \eta_k\<\nabla f(x_k),x_k - x^*> + \tfrac{1}{2}\eta_k^2\sqn{\nabla f(x_k)}
  \\&\at{uses the convexity of the function $f(x)$}\leq
  \tfrac{1}{2}\sqn{x_k - x^*} - \eta_k\brs{f(x_k) - f^*} + \tfrac{1}{2}\eta_k^2\sqn{\nabla f(x_k)},
\end{align*}
where \annotate.
Further, for Option~A, we proceed as follows:
\begin{align*}
  \tfrac{1}{2}\sqn{x_{k+1} - x^*}
  &\at{uses Option~A in \cref{eq:normalized_polyak_eta}}\leq
  \tfrac{1}{2}\sqn{x_k - x^*} - \tmfrac{\cR\brs{f(x_k) - f^*}}{\sqrt{K+1}\norm{\nabla f(x_k)}} + \tmfrac{\cR^2}{2\brs{K+1}}
  \\&\at{uses \Cref{ass:M01}}\leq
  \tfrac{1}{2}\sqn{x_k - x^*} - \tmfrac{\cR\brs{f(x_k) - f^*}}{\sqrt{K+1}\brr{\cM_0 + \cM_1\brs{f(x_k) - f^*}}} + \tmfrac{\cR^2}{2\brs{K+1}}
  \\&\leq
  \tfrac{1}{2}\sqn{x_k - x^*}
  - \min\brf*{\tmfrac{\cR\brs{f(x_k) - f^*}}{2\sqrt{K+1}\cM_0},\tmfrac{\cR}{2\sqrt{K+1}\cM_1}}
  + \tmfrac{\cR^2}{2\brs{K+1}},
\end{align*}
where \annotate.
After summing these inequalities and rearranging, we obtain the following:
\begin{align*}
  \min\brf*{\tmin_{k \in \brf{0,\dots,K}}\brs*{f(x_k) - f^*},\tmfrac{\cM_0}{\cM_1}}
  \leq \tmfrac{2\cR\cM_0}{\sqrt{K+1}}.
\end{align*}
From the choice of $K$ in \cref{eq:normalized_polyak_K}, we have $\tmfrac{\cM_0}{\cM_1} > \tmfrac{2\cR\cM_0}{\sqrt{K+1}}$. Hence, the minimum in the inequality above is attained in the first term, which implies
\begin{align*}
  \tmin_{k \in \brf{0,\dots,K}}\brs*{f(x_k) - f^*} \leq \tmfrac{2\cR\cM_0}{\sqrt{K+1}} \leq \epsilon.
\end{align*}
Finally, for Option~B, we proceed as follows:
\begin{align*}
  \tfrac{1}{2}\sqn{x_{k+1} - x^*}
  &\at{uses Option~B in \cref{eq:normalized_polyak_eta}}\leq
  \tfrac{1}{2}\sqn{x_k - x^*} - \tmfrac{\brs{f(x_k) - f^*}^2}{2\sqn{\nabla f(x_k)}}
  \\&\at{uses \Cref{ass:M01}}\leq
  \tfrac{1}{2}\sqn{x_k - x^*} - \tmfrac{\brs{f(x_k) - f^*}^2}{2\brs{\cM_0 + \cM_1\brs{f(x_k) - f^*}}^2}
  \\&\leq
  \tfrac{1}{2}\sqn{x_k - x^*} - \min\brf*{\tmfrac{\brs{f(x_k) - f^*}^2}{8\cM_0^2},\tmfrac{1}{8\cM_1^2}},
\end{align*}
where \annotate.
After summing these inequalities and rearranging, we obtain the following:
\begin{align*}
  \min\brf*{\tmin_{k \in \brf{0,\dots,K}}\brs*{f(x_k) - f^*},\tmfrac{\cM_0}{\cM_1}}^2
  \leq \tmfrac{4\brs{\cR\cM_0}^2}{K+1}.
\end{align*}
The rest of the proof is identical to Option~A.\qed

\proofsubsection{thm:SGD}

We define $\eta = \eta_k$ and proceed with the following inequality:
\begin{align*}
  \E*{\tfrac{1}{2}\sqn{x_{k+1} - x^*}}
  &\at{uses the iterations in \cref{eq:SGD} and the fact that $x^*\in Q_\cR$}=
  \E*{\tfrac{1}{2}\sqn{\proj[Q_\cR]{x_k - \eta\nabla_{\xi_k} f(x_k)} - \proj[Q_\cR]{x^*}}}
  \\&\at{uses the non-expansiveness of the projection}\leq
  \E*{\tfrac{1}{2}\sqn{x_k - x^* - \eta\nabla_{\xi_k} f(x_k)}}
  \\&=
  \E*{\tfrac{1}{2}\sqn{x_k - x^*}
    -\eta\<\nabla_{\xi_k} f(x_k),x_k - x^*>
  +\tfrac{1}{2}\eta^2\sqn{\nabla_{\xi_k} f(x_k)}}
  \\&\at{uses the unbiasedness property in \Cref{ass:G01} and the convexity of the function $f(x)$}\leq
  \E*{\tfrac{1}{2}\sqn{x_k - x^*}
    -\eta\brs*{f(x_k)  - f^*}
  +\tfrac{1}{2}\eta^2\sqn{\nabla_{\xi_k} f(x_k)}}
  \\&\at{use \Cref{ass:G01}}\leq
  \E*{\tfrac{1}{2}\sqn{x_k - x^*}
    -\eta\brs*{f(x_k)  - f^*}
    +\eta^2\sqn{u_{\xi_k}(x_k)} + \eta^2\sqn{v_{\xi_k}(x_k)}
  }
  \\&\at{use \Cref{ass:G01}}\leq
  \E*{\tfrac{1}{2}\sqn{x_k - x^*}
    -\eta\brs*{f(x_k)  - f^*}
    +\eta^2\cG_0^2 + \eta^2\cG_1^2\brs{f(x_k) - f^*}^2
  }
  \\&\at{uses the definition in \cref{eq:MF} and the fact that $\norm{x_k}\leq\cR$}\leq
  \E*{\tfrac{1}{2}\sqn{x_k - x^*}
    -\eta\brs*{f(x_k)  - f^*}
    +\eta^2\cG_0^2 + \eta^2\cG_1^2\cF\brs{f(x_k) - f^*}
  }
  \\&\at{uses the definition in \cref{eq:SGD_params}}\leq
  \E*{\tfrac{1}{2}\sqn{x_k - x^*}
    -\tfrac{1}{2}\eta\brs*{f(x_k)  - f^*}
    +\eta^2\cG_0^2
  },
\end{align*}
where \annotate. After summing these inequalities, we obtain
\begin{align*}
  \E{f(\ox_K) - f^*}
  &\at{uses the convexity of the function $f(x)$ and the definition of $\ox_K$}\leq
  \tmfrac{1}{K+1}\msum_{k=0}^{K}\E*{f(x_k) - f^*}
  \at{uses the inequality above and the fact that $\norm{x^*} \leq \cR$ and $x_0 = 0$}\leq
  \tmfrac{\cR^2}{\eta(K+1)} + 2\eta\cG_0^2
  \\&\at{uses the definitions in \cref{eq:SGD_params}}\leq
  \cO\brs*{\tmfrac{\cR\cG_0}{\sqrt{K+1}} + \tmfrac{\brs{\cR\cG_1}^2\cF}{K+1}}
  \at{uses the definition in \cref{eq:SGD_params}, \Cref{rem:GM01}, and the inequality in \cref{eq:MF_bound}}\leq
  \cO\brs{\epsilon},
\end{align*}
where \annotate.\qed

\proofsubsection{thm:AdaGrad}

Using the definition of $\ox_K$ and the convexity of $f(x)$, we can upper-bound $\E{f(\ox_K) - f^*}$ as follows:
\begin{align*}
  \E{f(\ox_K) - f^*}
  \leq
  \tmfrac{1}{K+1}\msum_{k=0}^K\E*{f(x_k) - f^*}.
\end{align*}
Next, we can upper-bound $\sum_{k=0}^K\E*{f(x_k) - f^*}$ as follows:
\begin{align*}
  \msum_{k=0}^K\E*{f(x_k) - f^*}
  &\at{uses the convexity of the function $f(x)$}\leq
  \msum_{k=0}^K\E*{\<\nabla f(x_k), x_k - x^*>}
  \\&\at{uses the unbiasedness in \Cref{ass:G01}}=
  \E*{\msum_{k=0}^K\<\nabla_{\xi_k} f(x_k), x_k - x^*>}
  \\&\at{uses the standard upper-bound on the regret for AdaGrad-Norm \citep[Theorem~2]{streeter2010less}}\leq
  \cO\brs*{\cR\E*{\sqrt{\vphantom{\rule{0pt}{8.5pt}}\smash{\tsum_{k=0}^K \sqn{\nabla_{\xi_k}f(x_k)}}}}}
  \\&\at{uses \Cref{ass:G01}}\leq
  \cO\brs*{\cR\E*{\sqrt{\vphantom{\rule{0pt}{8.5pt}}\smash{\tsum_{k=0}^K \sqn{u_{\xi_k}(x_k) + v_{\xi_k}(x_k)}}}}}
  \\&\leq
  \cO\brs*{\cR\E*{
      \sqrt{\vphantom{\rule{0pt}{8.5pt}}\smash{\tsum_{k=0}^K \sqn{u_{\xi_k}(x_k)}}}
      +
      \sqrt{\vphantom{\rule{0pt}{8.5pt}}\smash{\tsum_{k=0}^K \sqn{v_{\xi_k}(x_k)}}}
  }}
  \\&\at{uses \Cref{ass:G01}, Jensen's inequality, and the concavity of the square root}\leq
  \cO\brs*{
    \cR\cG_0\sqrt{K+1}
    +
    \cR\cG_1\E*{\sqrt{\vphantom{\rule{0pt}{8.5pt}}\smash{\tsum_{k=0}^K \brs{f(x_k) - f^*}^2}}}
  }
  \\&\at{uses the definition in \cref{eq:MF} and the fact that $\norm{x_k}\leq\cR$}\leq
  \cO\brs*{
    \cR\cG_0\sqrt{K+1}
    +
    \cR\cG_1\E*{\sqrt{\vphantom{\rule{0pt}{8.5pt}}\smash{\tsum_{k=0}^K \cF\brs{f(x_k) - f^*}}}}
  }
  \\&\at{uses Young's inequality}\leq
  \tfrac{1}{2}\msum_{k=0}^K \E{f(x_k) - f^*}
  +
  \cO\brs*{\cR\cG_0\sqrt{K+1} + \brs{\cR\cG_1}^2\cF},
\end{align*}
where \annotate. After rearranging, we obtain the following:
\begin{align*}
  \tmfrac{1}{K+1}\msum_{k=0}^K\E*{f(x_k) - f^*}
  &\leq
  \cO\brs*{
    \tmfrac{\cR\cG_0}{\sqrt{K+1}} + \tmfrac{\brs{\cR\cG_1}^2\cF}{K+1}
  }
  \at{uses the definition in \cref{eq:AdaGrad_params}, \Cref{rem:GM01}, and the inequality in \cref{eq:MF_bound}}\leq
  \cO\brs{\epsilon},
\end{align*}
where \annotate.\qed

\newpage

\section[Proofs for \crtCref{sec:alg}]{Proofs for \Cref{sec:alg}}

\subsection{Technical Lemma}

In this section, we present a technical lemma that will be used in the proofs of \Cref{thm:avg,,thm:exp1,,thm:exp2}.

\begin{lemma}<lem:tech>
  Let $p \in (0,1/2]$, and let $\brf{\Delta_k}_{k \in \N} \subset \R_+$ and $\brf{\delta_k}_{k \in \N} \subset \R_+$ be two sequences of non-negative numbers satisfying the following condition for all $k \in \N$:
  \begin{equation}
    \pi_k \Delta_k \leq \delta_k + p\tsum_{i=1}^k \alpha_i\pi_i \Delta_i,
  \end{equation}
  where $\alpha_k,\pi_k$ satisfy \cref{eq:fk}. Then, the following inequalities hold for all $k \in \N$:
  \begin{equation}
    \tsum_{i=1}^k\alpha_i\pi_i \Delta_i
    \leq
    \tsum_{i=2}^k \alpha_i\delta_i\brs*{\pi_k/\pi_{i-1}}^p
    +\alpha_1\delta_1\min\brf*{\brs*{\pi_k/\pi_0}^p, 2\brs*{\pi_k/\pi_1}^p}.
  \end{equation}
\end{lemma}

\begin{proof}
  We define $\tau_k \geq 0$ for $k \in \N$ as follows:
  \begin{align*}
    \tau_k = \tsum_{i=1}^k \alpha_i\pi_i\Delta_i.
  \end{align*}
  We can upper-bound $\tau_k$ as follows:
  \begin{align*}
    \tau_k
    \at{uses the definition of $\tau_k$}=
    \alpha_k\pi_k\Delta_k + \tau_{k-1}
    \at{uses the assumption of \Cref{lem:tech}}\leq
    \alpha_k\delta_k + p\alpha_k\tau_k + \tau_{k-1}
    \at{uses the inequality $pt \leq 1 - (1-t)^p$ for $t \leq 1$ which is implied by the concavity of the function $t\mapsto t^p$}\leq
    \alpha_k\delta_k + \brs*{1 - (1-\alpha_k)^p}\tau_k + \tau_{k-1},
  \end{align*}
  where \annotate.
  After rearranging, we obtain the following:
  \begin{align*}
    \alpha_k\delta_k + \tau_{k-1}
    \geq
    (1-\alpha_k)^p\tau_k
    \at{uses the definition in \cref{eq:fk}}=
    \brs*{\pi_{k-1}/\pi_k}^p\tau_k
  \end{align*}
  where \annotate.
  After rearranging one more time, we obtain the following:
  \begin{align*}
    \brs*{1/\pi_k}^p\tau_k \leq \brs*{1/\pi_{k-1}}^p\brs*{\tau_{k-1} + \alpha_k\delta_k}.
  \end{align*}
  In the case where $\pi_0 \neq 0$, we can upper-bound $\tau_k$ as follows:
  \begin{align*}
    \tau_k \leq \tsum_{i=1}^k \alpha_i\delta_i\brs*{\pi_k/\pi_{i-1}}^p.
  \end{align*}
  In the case $\pi_0 = 0$, we can upper-bound $\tau_k$ as follows:
  \begin{align*}
    \tau_k
    &\leq
    \tsum_{i=2}^k \alpha_i\delta_i\brs*{\pi_k/\pi_{i-1}}^p + \tau_1\brs*{\pi_k/\pi_1}^p
    \\&\at{uses the definition of $\tau_k$}=
    \tsum_{i=2}^k \alpha_i\delta_i\brs*{\pi_k/\pi_{i-1}}^p + \alpha_1 \pi_1 \Delta_1 \brs*{\pi_k/\pi_1}^p
    \\&\at{uses the fact that $p\alpha_1 \leq 1/2$ and the assumption of \Cref{lem:tech}}\leq
    \tsum_{i=2}^k \alpha_i\delta_i\brs*{\pi_k/\pi_{i-1}}^p
    +  2\alpha_1\delta_1\brs*{\pi_k/\pi_1}^p,
  \end{align*}
  where \annotate.
\end{proof}

\proofsubsection{lem:wd}
We can lower-bound $\E{\reg_K}$ as follows
\begin{align*}
  \E{\reg_K}
  &\at{uses the definition in \cref{eq:regret}}\geq
  \msum_{k=1}^K \E{\<\nabla_{\xi_k} f_k(z_k),z_k - x^*>}
  \\&\at{uses the unbiasedness in \Cref{ass:G01}, the fact that $\xi_k$ is independent of $z_k$ and $x_{k-1}$, and the definition in \cref{eq:fk}}=
  \msum_{k=1}^K \E{\<\nabla f_k(z_k),z_k - x^*>}
  \\&\at{uses the convexity of $f(x)$ and the definition in \cref{eq:fk}}\geq
  \msum_{k=1}^K \E{f_k(z_k) - f_k(x^*)}
  \\&\at{use the definitions in \cref{eq:fk}}=
  \msum_{k=1}^K \pi_k\E{f(\alpha_k z_k + (1-\alpha_k)x_{k-1}) - f(\alpha_k x^* + (1-\alpha_k)x_{k-1})}
  \\&\at{uses the iterations in \cref{eq:wd}}=
  \msum_{k=1}^K \pi_k\E{f(x_k) - f(\alpha_k x^* + (1-\alpha_k)x_{k-1})}
  \\&\at{use the convexity of $f(x)$}\geq
  \msum_{k=1}^K \brr{\pi_k\E{f(x_k) - f^*} - (1-\alpha_k)\pi_k\E{f(x_{k-1}) - f^*}}
  \\&\at{use the definitions in \cref{eq:fk}}=
  \msum_{k=1}^K \brr{\pi_k\E{f(x_k) - f^*} - \pi_{k-1}\E{f(x_{k-1}) - f^*}}
  \\&=
  \pi_K \E{f(x_K) - f^*} - \pi_0 \brs{f(x_0) - f^*},
\end{align*}
where \annotate. Rearranging concludes the proof.\qed

\proofsubsection{lem:regret}

We can upper-bound $\E{\reg_K}$ as follows:
\begin{align*}
  \E{\reg_K}
  &\at{uses \Cref{ass:regret} and the definition in \cref{eq:regret}}\leq
  \E*{\cC\cR\sqrt{\tsum_{k=1}^K \sqn{\nabla_{\xi_k} f_k(z_k)}}}
  \\&\at{uses \cref{eq:wd} and the definition in \cref{eq:fk}}=
  \E*{\cC\cR\sqrt{\tsum_{k=1}^K \alpha_k^2\pi_k^2\sqn{\nabla_{\xi_k} f(x_k)}}}
  \\&\at{use \Cref{ass:G01}}\leq
  \E*{\cC\cR\sqrt{\tsum_{k=1}^K \alpha_k^2\pi_k^2\sqn{u_{\xi_k}(x_k) + v_{\xi_k}(x_k)}}}
  \\&\leq
  \E*{\cC\cR\sqrt{\tsum_{k=1}^K \alpha_k^2\pi_k^2\sqn{u_{\xi_k}(x_k)}}
  + \cC\cR\sqrt{\tsum_{k=1}^K \alpha_k^2\pi_k^2\sqn{v_{\xi_k}(x_k)}}}
  \\&\at{uses Jensen's inequality and the concavity of the square root}\leq
  \cC\cR\sqrt{\tsum_{k=1}^K \alpha_k^2\pi_k^2\E*{\sqn{u_{\xi_k}(x_k)}}}
  + \E*{\cC\cR\sqrt{\tsum_{k=1}^K \alpha_k^2\pi_k^2\sqn{v_{\xi_k}(x_k)}}}
  \\&\at{use \Cref{ass:G01}}\leq
  \cC\cR\cG_0\sqrt{\tsum_{k=1}^K \alpha_k^2\pi_k^2}
  + \E*{\cC\cR\cG_1\sqrt{\tsum_{k=1}^K \alpha_k^2\pi_k^2\brs*{f(x_k) - f^*}^2}},
\end{align*}
where \annotate.\qed

\proofsubsection{thm:avg}

From the definition in \cref{eq:fk} and the definition $\alpha_k = 1/k$, it follows that $\pi_k = k$.
Next, let $p \in (0,1/2]$ and let $\Delta_k = \E{f(x_k) - f^*}$. Then, we can upper-bound $\pi_k\Delta_k = k\Delta_k$ as follows:
\begin{align*}
  \pi_k \Delta_k
  &\at{uses \Cref{lem:wd}}\leq
  \E{\reg_k}
  \\&\at{uses \Cref{lem:regret}}\leq
  \cC\cR\cG_0\sqrt{\tsum_{i=1}^k \alpha_i^2\pi_i^2}
  +
  \E*{\cC\cR\cG_1\sqrt{\tsum_{i=1}^k \alpha_i^2\pi_i^2\brs*{f(x_i) - f^*}^2}}
  \\&\at{uses the fact that $\alpha_k \pi_k = 1$}=
  \cC\cR\cG_0\sqrt{k}
  +
  \E*{\cC\cR\cG_1\sqrt{\tsum_{i=1}^k \alpha_i\pi_i\brs*{f(x_i) - f^*}^2}}
  \\&\at{uses the definition of $\cF$, Jensen's inequality, and the definition of $\Delta_i$}\leq
  \cC\cR\cG_0\sqrt{k}
  +
  \cC\cR\cG_1\sqrt{\tsum_{i=1}^k \cF\alpha_i\pi_i\Delta_i}
  \\&\at{uses Young's inequality}\leq
  \cC\cR\cG_0\sqrt{k}
  +
  \tmfrac{1}{4p}\brs{\cC\cR\cG_1}^2\cF
  +
  p\tsum_{i=1}^k \alpha_i\pi_i\Delta_i
\end{align*}
where \annotate.
Now, we can define $\delta_k = \cC\cR\cG_0\sqrt{k}+\tmfrac{1}{4p}\brs{\cC\cR\cG_1}^2\cF$ and $p = \min\big\{\frac{1}{\ln K},\frac{1}{4}\big\}$, and obtain the following inequality:
\begin{align*}
  \mi{2}\tsum_{i=1}^k\alpha_i\pi_i \Delta_i
  \\&\at{uses \Cref{lem:tech}}\leq
  \tsum_{i=2}^k \alpha_i\delta_i\brs*{\tmfrac{\pi_k}{\pi_{i-1}}}^p
  +2\alpha_1\delta_1\brs*{\tmfrac{\pi_k}{\pi_1}}^p
  \\&\at{uses the definition of $\delta_i$}=
  \tsum_{i=2}^k \alpha_i\brs*{\tmfrac{\pi_k}{\pi_{i-1}}}^p\brs*{\cC\cR\cG_0\sqrt{i}+\tmfrac{1}{4p}\brs{\cC\cR\cG_1}^2\cF}
  +2\alpha_1\brs*{\tmfrac{\pi_k}{\pi_1}}^p\brs*{\cC\cR\cG_0+\tmfrac{\brs{\cC\cR\cG_1}^2\cF}{4p}}
  \\&\at{uses the fact that $\alpha_i = 1/i$ and $\pi_i=i$}=
  \tsum_{i=2}^k \tmfrac{k^p}{i(i-1)^p}\brs*{\cC\cR\cG_0\sqrt{i}+\tmfrac{1}{4p}\brs{\cC\cR\cG_1}^2\cF}
  +2k^p\brs*{\cC\cR\cG_0+\tmfrac{\brs{\cC\cR\cG_1}^2\cF}{4p}}
  \\&\leq
  k^p\tsum_{i=1}^{k-1} \brs*{\tmfrac{\cC\cR\cG_0}{i^{p+1/2}} + \tmfrac{\brs{\cC\cR\cG_1}^2\cF}{4pi^{p+1}}}
  +2k^p\brs*{\cC\cR\cG_0+\tmfrac{\brs{\cC\cR\cG_1}^2\cF}{4p}}
  \\&\at{uses the fact that $p+1 > 1$ and $p+1/2\in (0,1)$}\leq
  k^p \brs*{
    \cC\cR\cG_0\brs*{1 + \tmfrac{k^{1/2-p} - 1}{1/2 - p}}
    +
    \tmfrac{(p+1)\brs{\cC\cR\cG_1}^2\cF}{4p^2}
    +
    2\cC\cR\cG_0
    +
    \tmfrac{\brs{\cC\cR\cG_1}^2\cF}{2p}
  }
  \\&\at{uses the fact that $p \leq 1/4$}\leq
  \brs*{4\cC\cR\cG_0\sqrt{k} + \tmfrac{\brs{\cC\cR\cG_1}^2\cF k^p}{2p^2}}
\end{align*}
where \annotate.
Next, we upper-bound $\Delta_K$ as follows:
\begin{align*}
  \Delta_K
  &\at{uses the previously obtained upper bound on $\pi_K\Delta_K$ and the fact that $\pi_K = K$}\leq
  \tmfrac{\cC\cR\cG_0}{\sqrt{K}}
  +
  \tmfrac{\cC\cR\cG_1}{K}\sqrt{\tsum_{k=1}^K \cF\alpha_k\pi_k\Delta_k}
  \\&\at{uses the previously obtained upper bound on $\sum_{k=1}^K\alpha_k\pi_k \Delta_k$}\leq
  \tmfrac{\cC\cR\cG_0}{\sqrt{K}}
  +
  \tmfrac{\cC\cR\cG_1}{K}\sqrt{
    4\cC\cR\cG_0\cF K^{1/2} + \tmfrac{\brs{\cC\cR\cG_1\cF}^2K^p}{2p^2}
  }
  \\&\at{uses the fact that $p \leq 1/\ln K$}\leq
  \tmfrac{\cC\cR\cG_0}{\sqrt{K}}
  +
  \tmfrac{\cC\cR\cG_1}{K}\sqrt{
    4\cC\cR\cG_0\cF K^{1/2} + \tmfrac{\brs{\cC\cR\cG_1\cF}^2e}{2p^2}
  }
  \\&\leq
  \tmfrac{\cC\cR\cG_0}{\sqrt{K}}
  +
  \tmfrac{2\sqrt{\cC\cR\cG_0\cF}\cC\cR\cG_1}{K^{3/4}}
  +
  \tmfrac{\sqrt{e}\brs{\cC\cR\cG_1}^2\cF}{\sqrt{2}pK}
  \\&\at{uses Young's inequality and the choice of $p$}\leq
  \cO\brs*{
    \tmfrac{\cR\cG_0}{\sqrt{K}}
    +
    \tmfrac{\brs{1+\ln K}\brs{\cR\cG_1}^2\cF}{K}
  },
\end{align*}
where \annotate.
Finally, we define $\kappa = e+\brs{\cR\cG_0/\epsilon} + \brs{\cR\cG_1}^2\cdot\brs{\cF/\epsilon}$ and upper-bound $\Delta_K$ as follows:
\begin{align*}
  \Delta_K
  &\leq
  \cO\brs*{
    \tmfrac{\cR\cG_0}{\sqrt{K}}
    +
    \tmfrac{\brs{1+\ln K}\brs{\cR\cG_1}^2\cF}{K}
  }
  \\&\at{uses the definition in \cref{eq:K0}}\leq
  \cO\brs*{
    \epsilon
    +\epsilon\ln^{-1}[\kappa]\brs*{
      1+ \ln\brs*{1 + \brs{\cR\cG_0/\epsilon}^2 + \brs{\cR\cG_1}^2\cdot\brs{\cF/\epsilon}\cdot\ln[\kappa]}
    }
  }
  \\&\leq
  \cO\brs*{
    \epsilon
    +\epsilon\ln^{-1}[\kappa]
    \brs*{1+ \ln\brs*{\kappa^2\ln[\kappa]}}
  }
  \\&\leq
  \cO\brs*{
    \epsilon
    +
    \epsilon\ln^{-1}[\kappa]
    \brr*{
      1
      +2\ln [\kappa]
      +\ln\brs*{\ln[\kappa]}
    }
  }
  \\&\leq
  \cO\brs*{
    \epsilon
    +
    \epsilon\ln^{-1}[\kappa]
    \brr*{
      1
      +3\ln [\kappa]
    }
  }
  \\&\leq
  \cO\brs{\epsilon},
\end{align*}
where \annotate.\qed

\proofsubsection{lem:regret_exp1}

First, let $J = \floor{K/T} \geq 1$ and define $\cK_s \subset \N$ for $s \in \brf{1,\dots,J}$ as follows:
\begin{align*}
  \cK_s =
  \begin{cases}
    \brf{(s-1)T+1,\dots, sT} & s \in \brf{1,\dots,J-1}\\
    \brf{(s-1)T+1,\dots, K} & s = J
  \end{cases},
\end{align*}
and let $k_s = \min \cK_s$.
It is easy to observe that $\sqcup_{s=1}^J \cK_s = \brf{1,\dots,K}$ and $T \leq \abs{\cK_s} \leq 2T-1$.
Further, we can upper-bound $\norm{x_k - x_{k_s}}$ for $k \in \cK_s$ as
\begin{align*}
  \norm{x_k - x_{k_s}}
  &\at{uses the triangle inequality}\leq
  \msum_{i=k_s}^{k-1}\norm{x_{i+1} - x_{i}}
  \at{uses the definition in \cref{eq:wd} and the assumption in \cref{eq:alpha_T}}\leq
  \msum_{i=k_s}^{k-1}\alpha\norm{z_{i+1} - x_{i}}
  \at{uses the fact that $z_{i+1},x_i \in Q_\cR$}\leq
  \msum_{i=k_s}^{k-1}2\alpha\cR
  \at{use the fact that $k,k_s \in \cK_s$ and the properties of $\cK_s$}\leq
  \cO\brs{\alpha T\cR}
  \at{uses the assumptions in \cref{eq:alpha_T}}\leq
  \cO\brs{1/\cG_1},
\end{align*}
where \annotate.
Next, we can upper-bound $\pi_k$ for $k \in \cK_s$ as follows:
\begin{align*}
  \pi_k
  &\at{uses the definition in \cref{eq:fk}}\leq
  \pi_{k_s}\mprod_{i=k_s+1}^k(1-\alpha_i)^{-1}
  \at{uses the assumptions in \cref{eq:alpha_T} and the fact that $k \in \cK_s$}\leq
  \pi_{k_s}(1-1/T)^{1-2T}
  =
  \pi_{k_s}(1+1/(T-1))^{2T-1}
  \leq
  \cO\brs{\pi_{k_s}},
\end{align*}
where \annotate.
Next, we can obtain the following inequality:
\begin{align*}
  \mi{0}\sqrt{\tsum_{k=1}^K \alpha_k^2\pi_k^2\brs*{f(x_k) - f^*}^2}
  =
  \sqrt{\tsum_{s=1}^J\tsum_{k\in\cK_s} \alpha_k^2\pi_k^2\brs*{f(x_k) - f^*}^2}
  \\&\leq
  \sqrt{\tsum_{s=1}^J\tsum_{k\in\cK_s} \alpha_k^2\pi_k^2\brs*{f(x_{k_s}) - f^*}^2}
  +\sqrt{\tsum_{s=1}^J\tsum_{k\in\cK_s} \alpha_k^2\pi_k^2\brs*{f(x_k) - f(x_{k_s})}^2}
  \\&\at{uses \Cref{lem:M01}, \Cref{rem:GM01}, and the upper bound on $\norm{x_k - x_{k_s}}$ for $k \in \cK_s$ above}\leq
  \sqrt{\tsum_{s=1}^J\tsum_{k\in\cK_s} \alpha_k^2\pi_k^2\brs*{f(x_{k_s}) - f^*}^2}
  +\cO\brs*{\sqrt{\tsum_{s=1}^J\tsum_{k\in\cK_s} \alpha_k^2\pi_k^2\brr*{\brs{\cG_0/\cG_1} + \brs*{f(x_{k_s}) - f^*}}^2}}
  \\&\leq
  \cO\brs*{\sqrt{\tsum_{s=1}^J\tsum_{k\in\cK_s} \alpha_k^2\pi_k^2\brs*{f(x_{k_s}) - f^*}^2}
  +\brs{\cG_0/\cG_1}\sqrt{\tsum_{k=1}^K \alpha_k^2\pi_k^2}}
  \\&\at{uses the assumption in \cref{eq:alpha_T} and the upper bound on $\pi_k$ above}=
  \cO\brs*{\sqrt{\tsum_{s=1}^J\tsum_{k\in\cK_s} \alpha^2\pi_{k_s}^2\brs*{f(x_{k_s}) - f^*}^2}
  +\brs{\cG_0/\cG_1}\sqrt{\tsum_{k=1}^K \alpha_k^2\pi_k^2}}
  \\&\at{uses the properties of $\cK_s$}\leq
  \cO\brs*{\tsum_{s=1}^J \sqrt{T}\alpha\pi_{k_s}\brs{f(x_{k_s}) - f^*}
  +\brs{\cG_0/\cG_1}\sqrt{\tsum_{k=1}^K \alpha_k^2\pi_k^2}},
\end{align*}
where \annotate.
Furthermore, we can obtain the following inequality:
\begin{align*}
  \mi{1}\tsum_{s=1}^J \sqrt{T}\alpha\pi_{k_s}\brs*{f(x_{k_s}) - f^*}
  \at{use the properties of $\cK_s$}\leq
  \cO\brs*{\tsum_{s=1}^J\tsum_{k \in \cK_s}\tmfrac{\alpha\pi_{k_s}}{\sqrt{T}} \brs*{f(x_{k_s}) - f(x_k) + f(x_k) - f^*}}
  \\&\at{uses \Cref{lem:M01}, \Cref{rem:GM01}, and the upper bounds on $\norm{x_k - x_{k_s}}$ for $k \in \cK_s$ above}\leq
  \cO\brs*{\tsum_{s=1}^J\tsum_{k \in \cK_s}  \tmfrac{\alpha\pi_{k_s}}{\sqrt{T}}\brr*{\alpha T\cR\cG_0 + \brs{f(x_k) - f^*}}}
  \\&\leq
  \cO\brs*{
    \tsum_{s=1}^J\tsum_{k \in \cK_s}  \tmfrac{\alpha\pi_{k_s}}{\sqrt{T}}\brs{f(x_k) - f^*}
    +\sqrt{T}\cR\cG_0\tsum_{s=1}^J\tsum_{k \in \cK_s} \alpha^2\pi_{k_s}
  }
  \\&\at{uses the assumption in \cref{eq:alpha_T} and the definition in \cref{eq:fk}}\leq
  \cO\brs*{
    \tsum_{s=1}^J\tsum_{k \in \cK_s}  \tmfrac{\alpha_k\pi_k}{\sqrt{T}}\brs{f(x_k) - f^*}
    +\sqrt{T}\cR\cG_0\tsum_{s=1}^J\tsum_{k \in \cK_s} \alpha_k^2\pi_k
  }
  \\&\at{uses the properties of $\cK_s$}\leq
  \cO\brs*{\tmfrac{1}{\sqrt{T}}\tsum_{k=1}^K  \alpha_k\pi_k\brs{f(x_k) - f^*}
  +\sqrt{T}\cR\cG_0\tsum_{k=1}^K\alpha_k^2\pi_k},
\end{align*}
where \annotate.
Finally, we can upper-bound $\E*{\reg_K}$ as follows:
\begin{align*}
  \E*{\reg_K}
  &\at{uses \Cref{lem:regret}}\leq
  \cO\brs*{\cR\cG_0\sqrt{\tsum_{k=1}^K \alpha_k^2\pi_k^2}
  + \E*{\cR\cG_1\sqrt{\tsum_{k=1}^K \alpha_k^2\pi_k^2\brs*{f(x_k) - f^*}^2}}}
  \\&\at{use the inequalities above}\leq
  \cO\brs*{\cR\cG_0\sqrt{\tsum_{k=1}^K \alpha_k^2\pi_k^2}
  +\cR\cG_1\tsum_{s=1}^J \sqrt{T}\alpha_{k_s}\pi_{k_s}\E*{f(x_{k_s}) - f^*}}
  \\&\at{use the inequalities above}\leq
  \cO\brs*{\cR\cG_0\sqrt{\tsum_{k=1}^K \alpha_k^2\pi_k^2}
    +\sqrt{T}\cR^2\cG_0\cG_1\tsum_{k=1}^K\alpha_k^2\pi_k
  +\tmfrac{\cR\cG_1}{\sqrt{T}}\tsum_{k=1}^K  \alpha_k\pi_k\E*{f(x_k) - f^*}}
  \\&\at{uses the definition in \cref{eq:alpha_T}}\leq
  \tfrac{1}{2}\tsum_{k=1}^K  \alpha_k\pi_k\E*{f(x_k) - f^*}
  +\cO\brs*{\cR\cG_0\sqrt{\tsum_{k=1}^K \alpha_k^2\pi_k^2}
  +T\cR\cG_0\tsum_{k=1}^K\alpha_k^2\pi_k},
\end{align*}
where \annotate.\qed

\proofsubsection{thm:exp1}

Let $\Delta_k = \E{f(x_k) - f^*}$, and let $\delta_k$ be defined as follows:
\begin{align*}
  \delta_k = e\cF + \cO\brs*{\cR\cG_0\sqrt{\tsum_{i=1}^k \alpha_i^2\pi_i^2}
  +T\cR\cG_0\tsum_{i=1}^k\alpha_i^2\pi_i}.
\end{align*}
For $k \leq T-1$, we can upper-bound $\Delta_k$ as follows:
\begin{align*}
  \pi_k\Delta_k
  &\at{uses the definition in \cref{eq:MF} and the fact that $x_k \in Q_\cR$}\leq
  \pi_k\cF
  \at{uses the definitions in \cref{eq:fk,eq:alpha_K1}}\leq
  (1-1/T)^{1-T}\cF
  \leq
  e\cF
  \at{uses the definition of $\delta_k$}\leq
  \delta_k
  \leq
  \delta_k
  +\tfrac{1}{2}\tsum_{i=1}^k \alpha_i\pi_i \Delta_i,
\end{align*}
where \annotate.
For $k \geq T$, we can upper-bound $\Delta_k$ as follows:
\begin{align*}
  \pi_k\Delta_k
  &\at{uses \Cref{lem:wd,lem:regret_exp1}}\leq
  \Delta_0 +  \tfrac{1}{2}\tsum_{i=1}^k \alpha_i\pi_i \Delta_i
  +\cO\brs*{
    \cR\cG_0\mysqrt{\tsum_{i=1}^k \alpha_i^2\pi_i^2}
    +
    T\cR\cG_0\tsum_{i=1}^k\alpha_i^2\pi_i
  }
  \\&\at{uses the definition in \cref{eq:MF} and the definition of $\delta_k$}\leq
  \delta_k
  +\tfrac{1}{2}\tsum_{i=1}^k \alpha_i\pi_i \Delta_i,
\end{align*}
where \annotate. Hence, the conditions of \Cref{lem:tech} with $p=1/2$ are satisfied.
Further, we can upper-bound $\delta_k$ as follows:
\begin{align*}
  \delta_k
  &\at{uses the definitions in \cref{eq:alpha_K1,eq:fk}}=
  e\cF + \cO\brs*{\cR\cG_0\alpha\pi_k\sqrt{\tsum_{i=0}^{k-1} (1-\alpha)^{2i}}
  +T\cR\cG_0\alpha^2\pi_k\tsum_{i=0}^{k-1}(1-\alpha)^i}
  \\&\leq
  e\cF + \cO\brs*{\cR\cG_0\sqrt{\alpha}\pi_k
  +T\cR\cG_0\alpha\pi_k}
  \\&\at{uses the definition in \cref{eq:alpha_K1}}\leq
  e\cF + \cO\brs*{\epsilon\pi_k},
\end{align*}
where \annotate.
Finally, we can upper-bound $\Delta_K$ as follows:
\begin{align*}
  \Delta_K
  &\at{uses \Cref{lem:tech} with $p=1/2$}\leq
  \cO\brs*{\tmfrac{\delta_K}{\pi_K} + \tsum_{k=1}^K\tmfrac{\alpha_k\delta_k}{\sqrt{\pi_K\pi_{k-1}}}}
  \\&\at{uses the inequality above}\leq
  \cO\brs*{
    \tmfrac{\cF}{\pi_K} + \epsilon
    +
    \tsum_{k=1}^K\tmfrac{\alpha_k\cF}{\sqrt{\pi_K\pi_{k-1}}}
    +
    \tsum_{k=1}^K\tmfrac{\alpha_k\pi_k\epsilon}{\sqrt{\pi_K\pi_{k-1}}}
  }
  \\&\at{uses the definitions in \cref{eq:fk,eq:alpha_K1}}\leq
  \cO\brs*{
    \cF[1-\alpha_1]^K + \epsilon
    +
    \alpha_1\cF
    \tsum_{k=0}^{K-1}[1-\alpha_1]^{[K+k]/2}
    +
    \alpha_1\epsilon
    \tsum_{k=0}^{K-1}[1-\alpha_1]^{[K-k]/2-1}
  }
  \\&\leq
  \cO\brs*{
    \cF[1-\alpha_1]^{K/2}
    +
  \epsilon}
  \\&\at{uses the definition in \cref{eq:alpha_K1}}\leq
  \cO\brs{\epsilon},
\end{align*}
where \annotate.\qed

\proofsubsection{lem:regret_exp2}

We define $J \in \N$, $k_s \in \N$ and $\cK_s \subset \N$ for $s \in \brf{1,\dots,J}$ similarly to the proof of \Cref{lem:regret_exp1} in \Cref{proof:lem:regret_exp1}. Consequently, it holds that $\pi_k \leq \cO\brs{\pi_{k_s}}$. Furthermore, for $k \in \cK_s$, using the fact that $\alpha_k \leq \alpha_{k_s}$ due to the definition in \cref{eq:alpha2}, we can upper-bound $\norm{x_k - x_{k_s}}$ as follows:
\begin{align*}
  \norm{x_k - x_{k_s}} \leq \cO\brs{\alpha_{k_s}T\cR} \leq \cO\brs{1/\cG_1}.
\end{align*}
Next, following the proof of \Cref{lem:regret_exp1} in \Cref{proof:lem:regret_exp1}, we can obtain the following inequality:
\begin{align*}
  \sqrt{\tsum_{k=1}^K \alpha_k^2\pi_k^2\brs*{f(x_k) - f^*}^2}
  &\leq
  \cO\brs*{\sqrt{\tsum_{s=1}^J\tsum_{k\in\cK_s} \alpha_k^2\pi_k^2\brs*{f(x_{k_s}) - f^*}^2}
  +\brs{\cG_0/\cG_1}\sqrt{\tsum_{k=1}^K \alpha_k^2\pi_k^2}}
  \\&\at{uses the fact that $\alpha_k \leq \alpha_{k_s}$ and $\pi_k \leq \cO\brs{\pi_{k_s}}$, and the properties of $\cK_s$}\leq
  \cO\brs*{\tsum_{s=1}^J\sqrt{T}\alpha_{k_s}\pi_{k_s}\brs*{f(x_{k_s}) - f^*}
  +\brs{\cG_0/\cG_1}\sqrt{\tsum_{k=1}^K \alpha_k^2\pi_k^2}},
\end{align*}
where \annotate.
Furthermore, from the definition in \cref{eq:alpha2} it follows that for $s \leq J-1$ we have $\alpha_k = \alpha_{k_s}$ for all $k \in \cK_s$, which implies $\alpha_{k_s} = \sum_{k \in \cK_s} \brs{\alpha_k /T}$. In addition, for $s = J$, we have
\begin{align*}
  \alpha_{k_J} = \tsum_{k = (J-1)T+1}^{JT} \brs{\alpha_k /T} \leq \tsum_{k \in \cK_J}\brs{\alpha_k /T}.
\end{align*}
Using similar arguments, we can show that
\begin{align*}
  \tsum_{k \in \cK_s} \alpha_{k_s}^2 \leq \cO\brs*{\tsum_{k \in \cK_s} \alpha_{k}^2}
  \quad\text{for all}\;\;
  s \in \brf{1,\dots,J}.
\end{align*}
Hence, we can obtain the following inequality:
\begin{align*}
  \mi{1}\tsum_{s=1}^J \sqrt{T}\alpha_{k_s}\pi_{k_s}\brs*{f(x_{k_s}) - f^*}
  \leq
  \cO\brs*{\tsum_{s=1}^J\tsum_{k \in \cK_s}\tmfrac{\alpha_k\pi_{k_s}}{\sqrt{T}} \brs*{f(x_{k_s}) - f(x_k) + f(x_k) - f^*}}
  \\&\at{uses \Cref{lem:M01}, \Cref{rem:GM01}, and the upper bounds on $\norm{x_k - x_{k_s}}$ for $k \in \cK_s$ above}\leq
  \cO\brs*{\tsum_{s=1}^J\tsum_{k \in \cK_s}\tmfrac{\alpha_k\pi_{k_s}}{\sqrt{T}} \brs*{\alpha_{k_s}T\cR\cG_0 + \brs{f(x_k) - f^*}}}
  \\&\leq
  \cO\brs*{
    \tsum_{s=1}^J\tsum_{k \in \cK_s}\tmfrac{\alpha_k\pi_{k_s}}{\sqrt{T}} \brs{f(x_k) - f^*}
    +
    \sqrt{T}\cR\cG_0\tsum_{s=1}^J\tsum_{k \in \cK_s}\alpha_k\alpha_{k_s}\pi_{k_s}
  }
  \\&\leq
  \cO\brs*{
    \tsum_{s=1}^J\tsum_{k \in \cK_s}\tmfrac{\alpha_k\pi_{k_s}}{\sqrt{T}} \brs{f(x_k) - f^*}
    +
    \sqrt{T}\cR\cG_0\tsum_{s=1}^J\tsum_{k \in \cK_s}\brs{\alpha_{k_s}^2 + \alpha_k^2}\pi_{k_s}
  }
  \\&\at{uses the previously obtained inequalities}\leq
  \cO\brs*{
    \tsum_{s=1}^J\tsum_{k \in \cK_s}\tmfrac{\alpha_k\pi_{k_s}}{\sqrt{T}} \brs{f(x_k) - f^*}
    +
    \sqrt{T}\cR\cG_0\tsum_{s=1}^J\tsum_{k \in \cK_s}\alpha_k^2\pi_{k_s}
  }
  \\&\at{uses the definition in \cref{eq:fk} and the properties of $\cK_s$}\leq
  \cO\brs*{
    \tsum_{k=1}^K\tmfrac{\alpha_k\pi_k}{\sqrt{T}} \brs{f(x_k) - f^*}
    +
    \sqrt{T}\cR\cG_0\tsum_{k=1}^K\alpha_k^2\pi_k
  },
\end{align*}
where \annotate.
Finally, we can upper-bound $\E*{\reg_K}$ as follows:
\begin{align*}
  \mi{3}\E*{\reg_K}
  \at{uses \Cref{lem:regret}}\leq
  \cO\brs*{\cR\cG_0\sqrt{\tsum_{k=1}^K \alpha_k^2\pi_k^2}
  + \E*{\cR\cG_1\sqrt{\tsum_{k=1}^K \alpha_k^2\pi_k^2\brs*{f(x_k) - f^*}^2}}}
  \\&\at{use the inequalities above}\leq
  \cO\brs*{\cR\cG_0\sqrt{\tsum_{k=1}^K \alpha_k^2\pi_k^2}
  +\cR\cG_1\tsum_{s=1}^J \sqrt{T}\alpha_{k_s}\pi_{k_s}\E*{f(x_{k_s}) - f^*}}
  \\&\at{use the inequalities above}\leq
  \cO\brs*{\cR\cG_0\sqrt{\tsum_{k=1}^K \alpha_k^2\pi_k^2}
    +\sqrt{T}\cR^2\cG_0\cG_1\tsum_{k=1}^K\alpha_k^2\pi_k
  +\tmfrac{\cR\cG_1}{\sqrt{T}}\tsum_{k=1}^K  \alpha_k\pi_k\E*{f(x_k) - f^*}}
  \\&\at{uses the definition in \cref{eq:alpha_T}}\leq
  \tfrac{1}{2}\tsum_{k=1}^K  \alpha_k\pi_k\E*{f(x_k) - f^*}
  +\cO\brs*{\cR\cG_0\sqrt{\tsum_{k=1}^K \alpha_k^2\pi_k^2}
  +T\cR\cG_0\tsum_{k=1}^K\alpha_k^2\pi_k},
\end{align*}
where \annotate.\qed

\proofsubsection{thm:exp2}

First, observe that the statement of \Cref{lem:regret_exp2} holds for arbitrary $K \geq T$.
Next, we define $\Delta_k$ and $\delta_k$ similarly to the proof of \Cref{thm:exp1} in \Cref{proof:thm:exp1}. Consequently, using \Cref{lem:regret_exp2,lem:wd} and using similar arguments to \Cref{proof:thm:exp1}, we can show that the following inequality holds for all $k \in \N$:
\begin{align*}
  \pi_k \Delta_k \leq \delta_k
  +\tfrac{1}{2}\tsum_{i=1}^k \alpha_i\pi_i \Delta_i.
\end{align*}
Hence, the conditions of \Cref{lem:tech} with $p=1/2$ are satisfied.
Furthermore, for $k \leq S T$, we can upper-bound $\delta_k$ as follows:
\begin{align*}
  \delta_k
  &\at{uses the definitions in \cref{eq:alpha2,eq:fk}}=
  e\cF + \cO\brs*{\cR\cG_0\alpha_1\pi_k\sqrt{\tsum_{i=0}^{k-1} (1-\alpha_1)^{2i}}
  +T\cR\cG_0\alpha_1^2\pi_k\tsum_{i=0}^{k-1}(1-\alpha_1)^i}
  \\&\leq
  e\cF + \cO\brs*{\cR\cG_0\sqrt{\alpha_1}\pi_k
  +T\cR\cG_0\alpha_1\pi_k}
  \\&\at{uses the definition in \cref{eq:alpha2}}\leq
  e\cF + \cO\brs*{\brs{\cG_0/\cG_1}\pi_k},
\end{align*}
where \annotate.
For $k \geq ST+1$, we can upper-bound $\delta_k$ as follows:
\begin{align*}
  \delta_k
  &\at{uses the definitions in \cref{eq:alpha2,eq:fk}}=
  \delta_{ST} + \cO\brs*{\cR\cG_0\alpha_{ST+1}\pi_k\sqrt{\tsum_{i=0}^{\infty} (1-\alpha_{ST+1})^{2i}}
  +T\cR\cG_0\alpha_{ST+1}^2\pi_k\tsum_{i=0}^{\infty}(1-\alpha_{ST+1})^i}
  \\&\leq
  \delta_{ST} + \cO\brs*{\cR\cG_0\sqrt{\alpha_{ST+1}}\pi_k
  +T\cR\cG_0\alpha_{ST+1}\pi_k}
  \\&\at{uses the definition in \cref{eq:alpha2}}\leq
  \delta_{ST} + \cO\brs*{\epsilon \pi_k},
\end{align*}
where \annotate.
Finally, we can upper-bound $\Delta_K$ as follows:
\begin{align*}
  \Delta_K
  &\at{uses \Cref{lem:tech} with $p=1/2$}\leq
  \cO\brs*{
    \tmfrac{\delta_K}{\pi_K} + \tsum_{k=1}^K\tmfrac{\alpha_k\delta_k}{\sqrt{\pi_K\pi_{k-1}}}
  }
  \\&\at{uses the fact that $K \geq TS+1$}=
  \cO\brs*{
    \tmfrac{\delta_K}{\pi_K}
    +
    \tsum_{k=1}^{ST}\tmfrac{\alpha_k\delta_k}{\sqrt{\pi_K\pi_{k-1}}}
    +
    \tsum_{k=ST+1}^{K}\tmfrac{\alpha_k\delta_k}{\sqrt{\pi_K\pi_{k-1}}}
  }
  \\&\at{uses the inequalities above}\leq
  \cO\brs*{
    \tmfrac{\cF}{\pi_K} + \tmfrac{\cG_0\pi_{ST}}{\cG_1\pi_K} + \epsilon
    +
    \tsum_{k=1}^{ST}\brs*{
      \tmfrac{\alpha_k\cF}{\sqrt{\pi_K\pi_{k-1}}}
      +
      \tmfrac{\cG_0\alpha_k\pi_k}{\cG_1\sqrt{\pi_K\pi_{k-1}}}
    }
  }
  \\&
  +\cO\brs*{
    \tsum_{k=ST+1}^{K}\brs*{
      \tmfrac{\alpha_k\cF}{\sqrt{\pi_K\pi_{k-1}}}
      +
      \tmfrac{\cG_0\alpha_k\pi_{ST}}{\cG_1\sqrt{\pi_K\pi_{k-1}}}
      +
      \tmfrac{\epsilon\alpha_k\pi_k}{\sqrt{\pi_K\pi_{k-1}}}
    }
  }
  \\&\at{use the definitions in \cref{eq:fk,eq:alpha2}}\leq
  \cO\brs*{
    \tmfrac{\cF}{\pi_K} + \tmfrac{\cG_0\pi_{ST}}{\cG_1\pi_K} + \epsilon
    +
    \tmfrac{\cF\alpha_1}{\sqrt{\pi_K}}\tsum_{k=0}^{ST-1}\brs{1-\alpha_1}^{k/2}
    +
    \tmfrac{\cG_0\alpha_1}{\cG_1\sqrt{\pi_K}}\tsum_{k=0}^{ST-1}\brs{1-\alpha_1}^{-k/2-1}
  }
  \\&
  +
  \cO\brs*{
    \brs*{
      \tmfrac{\alpha_{ST+1}\cF}{\sqrt{\pi_K\pi_{ST}}}
      +
      \tmfrac{\cG_0\alpha_{ST+1}\sqrt{\pi_{ST}}}{\cG_1\sqrt{\pi_K}}
    }
    \tsum_{k=0}^{K-ST-1}[1-\alpha_{ST+1}]^{k/2}
  }
  \\&
  +
  \cO\brs*{
    \tmfrac{\epsilon\alpha_{ST+1}\sqrt{\pi_{ST}}}{2\sqrt{\pi_K}}
    \tsum_{k=0}^{K-ST-1}
    [1-\alpha_{ST+1}]^{-k/2-1}
  }
  \\&\at{use the definitions in \cref{eq:fk,eq:alpha2}}\leq
  \cO\brs*{
    \cF[1-\alpha_1]^{ST/2}[1-\alpha_{ST+1}]^{[K-ST]/2}
    +
    \brs{\cG_0/\cG_1}[1-\alpha_{ST+1}]^{[K-ST]/2}
    +
    \epsilon
  }
  \\&\at{uses the definitions in \cref{eq:KS2}}\leq
  \cO\brs*{\epsilon},
\end{align*}
where \annotate.\qed

\newpage

\section[Proofs for \crtCref{sec:lower}]{Proofs for \Cref{sec:lower}}

\proofsubsection{thm:lower_SGD}

In this proof, we consider one-dimensional example functions $f(x)$. The construction is inspired by the counterexamples for the generalized smooth optimization setting \citep{crawshaw2022robustness,zhang2019gradient}. We choose deterministic gradient oracle $\nabla_\xi f(x) = \nabla f(x)$, where $\nabla f(x) \in \partial f(x)$ is an arbitrary subgradient.
We further consider the following two cases:
\begin{align*}
  \text{Case I:}\quad
  \eta > \tmfrac{2\cR}{\cG_0}\exp\brr{-\tfrac{1}{4}\cR\cG_1}
  ,\qquad
  \text{Case II:}\quad
  \eta \leq  \tmfrac{2\cR}{\cG_0}\exp\brr{-\tfrac{1}{4}\cR\cG_1}.
\end{align*}

\textbf{Case I.}
Consider the following example:
\begin{align*}
  f(x) =
  \begin{cases}
    \tmfrac{\cG_0}{\cG_1}\brs{\exp\brr{\cG_1\abs{x - \frac{3}{4}\cR}} - 1}
    & x \geq \frac{1}{2}\cR\\[0.5em]
    \cG_0\exp\brr{\frac{1}{4}\cR\cG_1}\brs{\frac{1}{2}\cR - x} +  \tmfrac{\cG_0}{\cG_1}\brs{\exp\brr{\frac{1}{4}\cR\cG_1} - 1}
    & x < \frac{1}{2}\cR
  \end{cases}
\end{align*}
It is easy to verify that the gradient $\nabla f(x)$ satisfies the following:
\begin{align*}
  \nabla f(0) = \nabla f(-\cR) = -\cG_0\exp\brr{\tfrac{1}{4}\cR\cG_1},\quad
  \nabla f(\cR) = \cG_0\exp\brr{\tfrac{1}{4}\cR\cG_1}.
\end{align*}
Since $x_0 = \ox_0 = 0$, using the lower bound on the stepsize $\eta$ and \cref{eq:SGD}, for $K \in \N$, we get
\begin{align*}
  x_K = (-1)^{K+1}\cdot \cR,\quad \ox_K =
  \begin{cases}
    0 & K \equiv 0 \bmod 2\\
    \cR/(K+1) & K \equiv 1 \bmod 2
  \end{cases}.
\end{align*}
Hence, for all $K \in \N$, it is easy to verify that $\ox_K \leq \frac{1}{2}\cR$ and
\begin{align*}
  \min\brf{f(x_K), f(\ox_K)} - f^*
  &\geq
  \tmfrac{\cG_0}{\cG_1}\brs{\exp\brr{\tfrac{1}{4}\cR\cG_1} - 1}
  \\&\at{use the assumptions in \cref{eq:lower_SGD_params}}\geq
  \tmfrac{\cG_0}{\cG_1}\brs{\exp\brr{\tfrac{1}{8}\cR\cG_1 + 1} - 1}
  \at{use the assumptions in \cref{eq:lower_SGD_params}}>
  \tmfrac{\cG_0}{\cG_1}\exp\brr{\tfrac{1}{8}\cR\cG_1}
  \at{use the assumptions in \cref{eq:lower_SGD_params}}\geq
  \epsilon,
\end{align*}
where \annotate. This concludes Case~I.

\textbf{Case II.}
Let $r = \max\big\{0,\frac{1}{\cG_1}\ln \big[\frac{8\epsilon}{\cR\cG_0}\big]\big\}$, and consider the following example:
\begin{align*}
  f(x) =
  \begin{cases}
    \tmfrac{\cG_0}{\cG_1}\brs{\exp\brr{\cG_1\abs{x - \frac{1}{2}\cR}} - 1}
    & \abs{x - \frac{1}{2}\cR} \leq r\\[0.5em]
    \tmfrac{8\epsilon}{\cR}\brs{\abs{x - \frac{1}{2}\cR} - r} + \tmfrac{\cG_0}{\cG_1}\brs{\exp\brr{\cG_1r} - 1}
    & \abs{x - \frac{1}{2}\cR} > r
  \end{cases}.
\end{align*}
One can verify that $f(x)$ is convex and $\nabla f(x)$ satisfies \Cref{ass:G01}. Indeed, in the case $r = 0$, the convexity is obvious and $\frac{8\epsilon}{\cR} \leq \cG_0$. In the case $r > 0$, we can show that $\nabla f(\frac{1}{2}\cR \pm r) = \pm \frac{8\epsilon}{\cR}$, and hence, the function $f(x)$ is differentiable at $x = \frac{1}{2}\cR \pm r$. Furthermore, using the assumptions in \cref{eq:lower_SGD_params}, we can show that $r \leq \frac{1}{8}\cR$. Hence, the following inequality holds:
\begin{align*}
  f(\tfrac{1}{8}\cR)
  =
  \tmfrac{8\epsilon}{\cR}\brs*{\tfrac{3}{8}\cR - r} + \tmfrac{\cG_0}{\cG_1}\brs{\exp\brr{\cG_1r} - 1}
  \geq
  2\epsilon + \tmfrac{\cG_0}{\cG_1}\brs{\exp\brr{\cG_1r} - 1}
  \geq
  2\epsilon.
\end{align*}
Consequently, we can show by induction that $x_K,\ox_K \leq \eta \cdot \frac{8\epsilon}{\cR}\cdot K$ for $K \leq \frac{\cR^2}{64\epsilon\eta}$. Hence, we can show that $\min\brf{f(x_K), f(\ox_K)} - f^* \geq 2\epsilon$ for all $K$ satisfying $ K \leq \frac{\cR^2}{64\epsilon\eta}$. Hence, to reach the desired precision, at least the following number of iterations is necessary:
\begin{align*}
  K \geq \tmfrac{\cR^2}{64\epsilon\eta}
  \at{uses the upper bound on $\eta$}\geq
  \exp\brr{\tfrac{1}{4}\cR\cG_1}\cdot \tmfrac{\cR\cG_0}{128\epsilon},
\end{align*}
where \annotate. This concludes Case~II.\qed

\proofsubsection{thm:lower_AdaGrad}

This proof is similar in many ways to the proof of \Cref{thm:lower_SGD} in \Cref{proof:thm:lower_SGD}, but is more involved because of the more complex AdaGrad-Norm stepsizes. We consider one-dimensional example functions $f(x)$. We choose deterministic gradient oracle $\nabla_\xi f(x) = \nabla f(x)$, where $\nabla f(x) \in \partial f(x)$ is an arbitrary subgradient.
We further consider the following three cases:
\begin{align*}
  \text{Case I:}\;\;
  \eta \geq \cR\exp\brr{\tfrac{1}{32}\cR\cG_1}
  ,\quad
  \text{Case II:}\;\;
  \tfrac{1}{4}\cR \leq \eta <  \cR\exp\brr{\tfrac{1}{32}\cR\cG_1}
  ,\quad
  \text{Case III:}\;\;
  \eta <  \tfrac{1}{4}\cR.
\end{align*}

\textbf{Case I.}
We consider the example from Case~I in the proof of \Cref{thm:lower_SGD} in \Cref{proof:thm:lower_SGD}. One can verify that $x_K,\ox_K$ coincide with the definitions obtained in \Cref{proof:thm:lower_SGD}, Case~I. Indeed, we can obtain the following inequality:
\begin{align*}
  \norm{\eta_{K-1} \nabla f(x_{K-1})} = \eta/\mysqrt{K} \geq 2\cR,
\end{align*}
as long as $K$ satisfies the following inequality:
\begin{align*}
  K \leq \brs*{\tmfrac{\eta}{2\cR}}^2.
\end{align*}
Consequently, to reach the precision $ \min\brf{f(x_K), f(\ox_K)} - f^* \leq \epsilon$, the number of iterations $K$ must satisfy the following inequality:
\begin{align*}
  K > \brs*{\tmfrac{\eta}{2\cR}}^2
  \at{uses the lower bound on $\eta$}\geq
  \tfrac{1}{4}\exp\brr{\tfrac{1}{16}\cR\cG_1},
\end{align*}
where \annotate. This concludes Case~I.

\textbf{Case II.}
Let $r = \max\big\{0,\frac{1}{\cG_1}\ln \big[\frac{32\epsilon}{\cR\cG_0}\big]\big\}$, and consider the following example:
\begin{align*}
  f(x) =
  \begin{cases}
    \tmfrac{\cG_0}{\cG_1}\brs{\exp\brr{\cG_1\abs{x - \frac{1}{16}\cR}} - 1}
    & x \leq  \frac{1}{16}\cR + r\\[0.5em]
    \tmfrac{32\epsilon}{\cR}\brs{x - \frac{1}{16}\cR - r} + \tmfrac{\cG_0}{\cG_1}\brs{\exp\brr{\cG_1r} - 1}
    & x > \frac{1}{16}\cR + r
  \end{cases}.
\end{align*}
Similar to the proof of Case~II in \Cref{proof:thm:lower_SGD}, one can verify that $f(x)$ is convex and $\nabla f(x)$ satisfies \Cref{ass:G01}. Additionally, from the assumptions in \cref{eq:lower_AdaGrad_params}, it follows that $r \leq \frac{1}{32}\cR$. Hence, the following inequality holds:
\begin{align*}
  f(\tfrac{1}{8}\cR)
  \geq
  \tmfrac{32\epsilon}{\cR}\brs{\tfrac{1}{16}\cR  - r} + \tmfrac{\cG_0}{\cG_1}\brs{\exp\brr{\cG_1r} - 1}
  =
  \epsilon +  \tmfrac{32\epsilon}{\cR}\brs{\tfrac{1}{32}\cR  - r} + \tmfrac{\cG_0}{\cG_1}\brs{\exp\brr{\cG_1r} - 1}
  >
  \epsilon.
\end{align*}
In addition, the following inequality holds:
\begin{align*}
  f(0)
  =
  \tmfrac{\cG_0}{\cG_1}\brs{\exp\brr{\tfrac{1}{16}\cR\cG_1} - 1}
  \at{use the assumptions in \cref{eq:lower_AdaGrad_params}}\geq
  \tmfrac{\cG_0}{\cG_1}\brs{\exp\brr{\tfrac{1}{32}\cR\cG_1 + 1} - 1}
  \at{use the assumptions in \cref{eq:lower_AdaGrad_params}}>
  \tmfrac{\cG_0}{\cG_1}\exp\brr{\tfrac{1}{32}\cR\cG_1}
  \at{use the assumptions in \cref{eq:lower_AdaGrad_params}}\geq
  \epsilon,
\end{align*}
where \annotate.
Furthermore, from the lower bound on $\eta$ it follows that $x_1 \geq \frac{1}{4}\cR$, and from the upper bound on $\eta$ it follows that
\begin{align*}
  x_K
  \geq
  \tfrac{1}{4}\cR - \tmfrac{\eta}{\norm{\nabla f(0)}} \cdot \tmfrac{32\epsilon}{\cR}\cdot (K-1)
  \geq
  \tfrac{1}{4}\cR - \tmfrac{32\epsilon}{\cG_0\exp\brr{\tfrac{1}{32}\cR\cG_1}} \cdot (K-1)
  \geq
  \tfrac{1}{8}\cR,
\end{align*}
and $\ox_K \geq \frac{1}{8}\cR$ for $K \in \N$,
as long as $K$ satisfies the following inequality:
\begin{align*}
  K \leq 1 + \exp\brr{\tfrac{1}{32}\cR\cG_1} \cdot \tmfrac{\cR\cG_0}{256\epsilon}.
\end{align*}
Hence, to reach the precision $\min\brf{f(x_K), f(\ox_K)} - f^* \leq \epsilon$, it is necessary that
\begin{align*}
  K
  \geq
  \exp\brr{\tfrac{1}{32}\cR\cG_1} \cdot \tmfrac{\cR\cG_0}{256\epsilon}
  \at{uses the assumptions in \cref{eq:lower_AdaGrad_params}}\geq
  \tmfrac{1}{8}\exp\brr{\tfrac{1}{64}\cR\cG_1},
\end{align*}
where \annotate.
This concludes Case~II.

\textbf{Case III.}
Let $r = \max\big\{0,\frac{1}{\cG_1}\ln \big[\frac{32\epsilon}{\cR\cG_0}\big]\big\}$, $m = \max\brf{\cG_0,\frac{32\epsilon}{\cR}}$, and consider the following example:
\begin{align*}
  f(x) =
  \begin{cases}
    0
    & x \geq \cR\\[0.5em]
    \tmfrac{\cG_0}{\cG_1}\brs{\exp\brr{\cG_1\brs{\cR - x}} - 1}
    &   x \in [\cR - r,\cR) \\[0.5em]
    \tmfrac{32\epsilon}{\cR}\brs{\cR - r - x} + \tmfrac{\cG_0}{\cG_1}\brs{\exp\brr{\cG_1r} - 1}
    & x \in [\frac{1}{4}\cR, \cR - r) \\[0.5em]
    \tmfrac{m}{\cG_1}\brs{\exp\brr{\cG_1\brs{\frac{1}{4}\cR - x}} - 1}
    +\tmfrac{32\epsilon}{\cR}\brs{\frac{3}{4}\cR - r} + \tmfrac{\cG_0}{\cG_1}\brs{\exp\brr{\cG_1r} - 1}
    & x < \frac{1}{4}\cR
  \end{cases}.
\end{align*}
One can verify that $f(x)$ is convex, $\nabla f(x)$ satisfies \Cref{ass:G01}, and $r \leq \frac{1}{32}\cR$. Moreover, from the upper bound on $\eta$ it follows that there exists $\hK \in \N$ such that $x_{\hK} \in (\frac{1}{4}\cR,\frac{1}{2}\cR]$, and $\ox_K \leq x_{\hK}$ for all $K \leq \hK$.
Moreover, we can upper-bound $x_K$ as follows:
\begin{align*}
  x_K
  &\leq
  \tfrac{1}{2}\cR + \tmfrac{\eta}{\norm{\nabla f(0)}} \cdot \tmfrac{32\epsilon}{\cR}\cdot (K - \hK)
  \\&\leq
  \tfrac{1}{2}\cR + \tmfrac{\cR}{4m\exp\brr{\tfrac{1}{4}\cR\cG_1}} \cdot \tmfrac{32\epsilon}{\cR}\cdot (K - \hK)
  \\&\leq
  \tfrac{1}{2}\cR +  \tmfrac{\cR}{4\cG_0\exp\brr{\tfrac{1}{4}\cR\cG_1}} \cdot \tmfrac{32\epsilon}{\cR}\cdot (K - \hK)
  \\&=
  \tfrac{1}{2}\cR +  \tmfrac{8\epsilon}{\cG_0\exp\brr{\tfrac{1}{4}\cR\cG_1}}\cdot (K - \hK)
  \leq
  \tfrac{3}{4}\cR,
\end{align*}
as long as $K$ satisfies the following inequality:
\begin{align*}
  K \leq \hK + \exp\brr{\tfrac{1}{4}\cR\cG_1} \cdot \tmfrac{\cR\cG_0}{32\epsilon},
\end{align*}
Consequently, for such $K \in \N$, we have
\begin{align*}
  \min\brf{f(x_K), f(\ox_K)} - f^*
  &\geq
  f(\tfrac{3}{4}\cR) - f^*
  \geq
  \tmfrac{32\epsilon}{\cR}\brs{\cR - r - \tfrac{3}{4}\cR} + \tmfrac{\cG_0}{\cG_1}\brs{\exp\brr{\cG_1r} - 1}
  \geq
  4\epsilon.
\end{align*}
Hence, to reach the precision $\min\brf{f(x_K), f(\ox_K)} - f^* \leq \epsilon$, at least the following number of iterations is necessary:
\begin{align*}
  K
  \geq
  \hK + \exp\brr{\tfrac{1}{4}\cR\cG_1} \cdot \tmfrac{\cR\cG_0}{32\epsilon}
  \at{uses the assumptions in \cref{eq:lower_AdaGrad_params}}\geq
  \exp\brr{\tfrac{1}{8}\cR\cG_1},
\end{align*}
where \annotate.
This concludes Case~III.\qed

\newpage

\section[Proofs for \crtCref{sec:uni}]{Proofs for \Cref{sec:uni}}

\proofsubsection{lem:SL01}

We can upper-bound $\norm{\nabla h(x)}$ as follows:
\begin{align*}
  \norm{\nabla h(x)}
  &\at{use the definitions in \cref{eq:h}}=
  \tmfrac{1}{(1+\nu)\brs{h(x)}^\nu}\norm{\nabla f(x)}
  \\&\at{uses \Cref{ass:SL01}, Jensen's inequality, and the triangle inequality}\leq
  \tmfrac{1}{(1+\nu)\brs{h(x)}^\nu}\E*[\xi\sim\cD]{\norm{u_\xi(x)} + \norm{v_\xi(x)}}
  \\&\at{uses \Cref{ass:SL01} and Jensen's inequality}\leq
  \tmfrac{1}{(1+\nu)\brs{h(x)}^\nu}
  \brs*{
    \sqrt{\sigma^2 + \cL_0^{\frac{2}{1+\nu}}\brs{f(x) - f^*}^{\frac{2\nu}{1+\nu}}}
    +\cL_1\brs{f(x)- f(x^*)}
  }
  \\&\at{use the definitions in \cref{eq:h}}=
  \tmfrac{1}{(1+\nu)\brs{h(x)}^\nu}
  \brs*{
    \sqrt{\sigma^2 + \cL_0^{\frac{2}{1+\nu}}\brs{h(x)^{1+\nu} - \beta}^{\frac{2\nu}{1+\nu}}}
    +\cL_1\brs{h(x)^{1+\nu} - \beta}
  }
  \\&\leq
  \tmfrac{1}{(1+\nu)\brs{h(x)}^\nu}
  \brs*{
    \sigma + \cL_0^{\frac{1}{1+\nu}}
    \brs{h(x)}^\nu
    -\cL_1 \beta
  }
  +\tfrac{1}{1+\nu}\cL_1 h(x)
  \\&\at{use the definitions in \cref{eq:h}}\leq
  \tmfrac{1}{(1+\nu)\brs{h(x)}^\nu}
  \brs*{
    \max\brf*{\cL_1\beta, \cL_0^{\frac{1}{1+\nu}}\beta^{\frac{\nu}{1+\nu}}} + \cL_0^{\frac{1}{1+\nu}}
    \brs{h(x)}^\nu
    -\cL_1 \beta
  }
  +\tfrac{1}{1+\nu}\cL_1 h(x)
  \\&\at{use the definitions in \cref{eq:h}}\leq
  \tmfrac{1}{(1+\nu)\brs{h(x)}^\nu}
  \brs*{
    \max\brf*{\cL_1\beta ,\cL_0^{\frac{1}{1+\nu}}\brs{h(x)}^\nu} + \cL_0^{\frac{1}{1+\nu}}
    \brs{h(x)}^\nu
    -\cL_1 \beta
  }
  +\tfrac{1}{1+\nu}\cL_1 h(x)
  \\&\leq
  \tfrac{2}{(1+\nu)}
  \cL_0^{\frac{1}{1+\nu}}
  +\tfrac{1}{1+\nu}\cL_1 h(x),
\end{align*}
where \annotate.
Next, we can upper-bound $\abs{h(x') - h(x)}^{1+\nu}$ as follows:
\begin{align*}
  \abs{h(x') - h(x)}^{1+\nu}
  &\at{uses the inequality above and similar arguments to the proof of \Cref{lem:M01} in \Cref{proof:lem:M01}}\leq
  \cO\brs*{\brr*{\cL_0^{\smash[t]{1/(1+\nu)}} + \cL_1 h(x)}^{1+\nu}\exp\brr{\cL_1\norm{x'-x}}\norm{x'-x}^{1+\nu}}
  \\&\leq
  \cO\brs*{\brr*{\cL_0 + \brs{\cL_1 h(x)}^{1+\nu}}\exp\brr{\cL_1\norm{x'-x}}\norm{x'-x}^{1+\nu}},
\end{align*}
where \annotate.\qed

\proofsubsection{lem:regret_uni}

We can upper-bound $\E{\reg_K}$ as follows:
\begin{align*}
  \E{\reg_K}
  &\at{uses \Cref{ass:regret} and the definition in \cref{eq:regret}}\leq
  \E*{\cC\cR\sqrt{\tsum_{k=1}^K \sqn{\nabla_{\xi_k} f_k(z_k)}}}
  \\&\at{uses \cref{eq:wd} and the definition in \cref{eq:fk}}=
  \E*{\cC\cR\sqrt{\tsum_{k=1}^K \alpha_k^2\pi_k^2\sqn{\nabla_{\xi_k} f(x_k)}}}
  \\&\at{use \Cref{ass:SL01}}\leq
  \E*{\cC\cR\sqrt{\tsum_{k=1}^K \alpha_k^2\pi_k^2\sqn{u_{\xi_k}(x_k) + v_{\xi_k}(x_k)}}}
  \\&\leq
  \E*{\cC\cR\sqrt{\tsum_{k=1}^K \alpha_k^2\pi_k^2\sqn{u_{\xi_k}(x_k)}}
  + \cC\cR\sqrt{\tsum_{k=1}^K \alpha_k^2\pi_k^2\sqn{v_{\xi_k}(x_k)}}}
  \\&\at{use Jensen's inequality}\leq
  \cC\cR\sqrt{\tsum_{k=1}^K \alpha_k^2\pi_k^2\E*{\sqn{u_{\xi_k}(x_k)}}}
  + \E*{\cC\cR\sqrt{\tsum_{k=1}^K \alpha_k^2\pi_k^2\sqn{v_{\xi_k}(x_k)}}}
  \\&\at{use \Cref{ass:SL01}}\leq
  \cC\cR\sigma\sqrt{\tsum_{k=1}^K \alpha_k^2\pi_k^2}
  +\cC\cR\cL_0^{\frac{1}{1+\nu}}\sqrt{\E*{\tsum_{k=1}^K \alpha_k^2\pi_k^2\brs{f(x_k) - f^*}^{\frac{2\nu}{1+\nu}}}}
  \\&
  +\E*{\cC\cR\cL_1\sqrt{\tsum_{k=1}^K \alpha_k^2\pi_k^2\brs*{f(x_k) - f^*}^2}}
  \\&\at{use Jensen's inequality}\leq
  \cC\cR\sigma\sqrt{\tsum_{k=1}^K \alpha_k^2\pi_k^2}
  +\cC\cR\cL_0^{\frac{1}{1+\nu}}\sqrt{\tsum_{k=1}^K \brs{\alpha_k\pi_k}^{\frac{2}{1+\nu}}\brs{\alpha_k\pi_k\E*{f(x_k) - f^*}}^{\frac{2\nu}{1+\nu}}}
  \\&
  +\E*{\cC\cR\cL_1\sqrt{\tsum_{k=1}^K \alpha_k^2\pi_k^2\brs*{f(x_k) - f^*}^2}}
  \\&\at{uses H\"older's inequality}\leq
  \cC\cR\sigma\sqrt{\tsum_{k=1}^K \alpha_k^2\pi_k^2}
  +\E*{\cC\cR\cL_1\sqrt{\tsum_{k=1}^K \alpha_k^2\pi_k^2\brs*{f(x_k) - f^*}^2}}
  \\&
  +\cC\cR\cL_0^{\frac{1}{1+\nu}}
  \brs*{\tsum_{k=1}^K [\alpha_k\pi_k]^{\frac{2}{1-\nu}}}^{\frac{1-\nu}{2(1+\nu)}}\brs*{\tsum_{k=1}^K \alpha_k\pi_k\E*{f(x_k) - f^*}}^{\frac{\nu}{1+\nu}
  }
  \\&\at{uses Young's inequality}\leq
  \cO\brs*{
    \cR\sigma\sqrt{\tsum_{k=1}^K \alpha_k^2\pi_k^2}
    +\cL_0\cR^{1+\nu}
    \brs*{\tsum_{k=1}^K [\alpha_k\pi_k]^{\frac{2}{1-\nu}}}^{\frac{1-\nu}{2}}
  }
  \\&
  +\tfrac{1}{4}\tsum_{k=1}^K \alpha_k\pi_k\E*{f(x_k) - f^*}
  +\cO\brs*{\E*{\cR\cL_1\sqrt{\tsum_{k=1}^K \alpha_k^2\pi_k^2\brs*{f(x_k) - f^*}^2}}},
\end{align*}
where \annotate.

Next, we define $J \in \N$, $k_s \in \N$ and $\cK_s \subset \N$ for $s \in \brf{1,\dots,J}$ similarly to the proof of \Cref{lem:regret_exp1} in \Cref{proof:lem:regret_exp1}. Consequently, it holds that $\pi_k \leq \cO\brs{\pi_{k_s}}$. Furthermore, for $k \in \cK_s$, using the fact that $\alpha_k \leq \alpha_{k_s}$ due to the definition in \cref{eq:alpha_uni}, we can upper-bound $\norm{x_k - x_{k_s}}$ as follows:
\begin{align*}
  \norm{x_k - x_{k_s}} \leq \cO\brs{\alpha_{k_s}T\cR} \leq \cO\brs{1/\cL_1}.
\end{align*}
Furthermore, using similar arguments to the proof of \Cref{lem:regret_exp2} in \Cref{proof:lem:regret_exp2}, we get the following:
\begin{align*}
  \tsum_{k \in \cK_s} \alpha_{k_s}^q \leq \cO\brs*{\tsum_{k \in \cK_s} \alpha_k^q}
  \quad\text{for all}\;\; q > 0 \text{\;and\;} s \in\brf{1,\dots,J}.
\end{align*}
Next, we can obtain the following inequality:
\begin{align*}
  \mi{0}\sqrt{\tsum_{k=1}^K \alpha_k^2\pi_k^2\brs*{f(x_k) - f^*}^2}
  =
  \sqrt{\tsum_{s=1}^J\tsum_{k\in\cK_s} \alpha_k^2\pi_k^2\brs*{f(x_k) - f^*}^2}
  \\&\leq
  \sqrt{\tsum_{s=1}^J\tsum_{k\in\cK_s} \alpha_k^2\pi_k^2\brs*{f(x_{k_s}) - f^*}^2}
  +\sqrt{\tsum_{s=1}^J\tsum_{k\in\cK_s} \alpha_k^2\pi_k^2\brs*{f(x_k) - f(x_{k_s})}^2}
  \\&\at{uses the definition in \cref{eq:h}}\leq
  \sqrt{\tsum_{s=1}^J\tsum_{k\in\cK_s} \alpha_k^2\pi_k^2\brs*{h(x_{k_s})}^{2(1+\nu)}}
  +\sqrt{\tsum_{s=1}^J\tsum_{k\in\cK_s} \alpha_k^2\pi_k^2\brs*{[h(x_k)]^{1+\nu} - [h(x_{k_s})]^{1+\nu}}^2}
  \\&\leq
  \sqrt{\tsum_{s=1}^J\tsum_{k\in\cK_s} \alpha_k^2\pi_k^2\brs*{h(x_{k_s})}^{2(1+\nu)}}
  \\&
  +\cO\brs*{
    \sqrt{\tsum_{s=1}^J\tsum_{k\in\cK_s} \alpha_k^2\pi_k^2\brs*{
        h(x_{k_s}) + \abs{h(x_k) - h(x_{k_s})}
    }^{2\nu}\brs*{h(x_k) - h(x_{k_s})}^2}
  }
  \\&\at{uses Young's inequality}\leq
  \cO\brs*{
    \sqrt{\tsum_{s=1}^J\tsum_{k\in\cK_s} \alpha_k^2\pi_k^2\brs*{h(x_{k_s})}^{2(1+\nu)}}
    +
    \sqrt{\tsum_{s=1}^J\tsum_{k\in\cK_s} \alpha_k^2\pi_k^2
      \brs*{h(x_k) - h(x_{k_s})}^{2(1+\nu)}
    }
  }
  \\&\at{uses \Cref{lem:SL01} and the upper-bound on $\norm{x_k - x_{k_s}}$ above}\leq
  \cO\brs*{\sqrt{\tsum_{s=1}^J\tsum_{k\in\cK_s} \alpha_k^2\pi_k^2\brs*{h(x_{k_s})}^{2(1+\nu)}}}
  \\&
  +\cO\brs*{\sqrt{\tsum_{s=1}^J\tsum_{k\in\cK_s} \alpha_k^2\pi_k^2\norm{x_k-x_{k_s}}^{2(1+\nu)}
      \brr*{\cL_0 + \brs{\cL_1 h(x_{k_s})}^{1+\nu}}^2
  }}
  \\&\at{uses the upper bounds on $\pi_k$ and $\norm{x_k - x_{k_s}}$}\leq
  \cO\brs*{\sqrt{\tsum_{s=1}^J\tsum_{k\in\cK_s} \alpha_k^2\pi_k^2\brs*{h(x_{k_s})}^{2(1+\nu)}}
    +\brs{T\cR}^{1+\nu}\cL_0 \sqrt{\tsum_{s=1}^J\tsum_{k\in\cK_s} \alpha_{k_s}^{2(1+\nu)}\alpha_k^2\pi_{k_s}^2}
  }
  \\&\at{uses the inequalities $\pi_k \leq \cO\brs{\pi_{k_s}}$ and $\alpha_k \leq \alpha_{k_s}$}\leq
  \cO\brs*{\sqrt{\tsum_{s=1}^J\tsum_{k\in\cK_s} \alpha_{k_s}^2\pi_{k_s}^2\brs*{h(x_{k_s})}^{2(1+\nu)}}
    +\brs{T\cR}^{1+\nu}\cL_0 \sqrt{\tsum_{s=1}^J\tsum_{k\in\cK_s} \alpha_{k_s}^{4+2\nu}\pi_{k_s}^2}
  }
  \\&\at{use the properties of $\cK_s$}\leq
  \cO\brs*{\sqrt{\tsum_{s=1}^J\tsum_{k\in\cK_s} \alpha_{k_s}^2\pi_{k_s}^2\brs*{h(x_{k_s})}^{2(1+\nu)}}
    +\brs{T\cR}^{1+\nu}\cL_0 \tsum_{s=1}^J\sqrt{T} \alpha_{k_s}^{2+\nu}\pi_{k_s}
  }
  \\&\at{use the properties of $\cK_s$}\leq
  \cO\brs*{\sqrt{\tsum_{s=1}^J\tsum_{k\in\cK_s} \alpha_{k_s}^2\pi_{k_s}^2\brs*{h(x_{k_s})}^{2(1+\nu)}}
    +\brs{T\cR}^{1+\nu}\cL_0 \tmfrac{1}{\sqrt{T}}\tsum_{s=1}^J\tsum_{k\in\cK_s} \alpha_{k_s}^{2+\nu}\pi_{k_s}
  }
  \\&\at{uses the inequality above with $q=2+\nu$ and the fact that $\pi_{k_s} \leq \pi_k$}\leq
  \cO\brs*{\sqrt{\tsum_{s=1}^J\tsum_{k\in\cK_s} \alpha_{k_s}^2\pi_{k_s}^2\brs*{h(x_{k_s})}^{2(1+\nu)}}
    +\brs{T\cR}^{1+\nu}\cL_0 \tmfrac{1}{\sqrt{T}}\tsum_{k=1}^K\alpha_k^{2+\nu}\pi_k
  }
  \\&\at{uses the definition in \cref{eq:h}}\leq
  \cO\brs*{
    \sqrt{\tsum_{s=1}^J\tsum_{k\in\cK_s} \alpha_{k_s}^2\pi_{k_s}^2\brs*{f(x_{k_s}) - f^*}^2}
    +
    \brs{T\cR}^{1+\nu}\cL_0 \tmfrac{1}{\sqrt{T}}\tsum_{k=1}^K\alpha_k^{2+\nu}\pi_k
    +
    \beta\sqrt{\tsum_{k=1}^K \alpha_k^2\pi_k^2}
  }
  \\&\at{uses the definition in \cref{eq:h} and the properties of $\cK_s$}\leq
  \cO\brs*{
    \tsum_{s=1}^J\sqrt{T} \alpha_{k_s}\pi_{k_s}\brs{f(x_{k_s}) - f^*}
    +
    \brs{T\cR}^{1+\nu}\cL_0 \tmfrac{1}{\sqrt{T}}\tsum_{k=1}^K\alpha_k^{2+\nu}\pi_k
    +
    \brs{\sigma/\cL_1}\sqrt{\tsum_{k=1}^K \alpha_k^2\pi_k^2}
  },
\end{align*}
where \annotate.
Furthermore, we can obtain the following inequality:
\begin{align*}
  \mi{0}\tsum_{s=1}^J \sqrt{T}\alpha_{k_s}\pi_{k_s}\brs*{f(x_{k_s}) - f^*}
  \\&\at{uses the definition in \cref{eq:h} and the properties of $\cK_s$ and the inequality above with $q = 1$}\leq
  \cO\brs*{
    \tsum_{s=1}^J\tsum_{k \in \cK_s}\tmfrac{\alpha_k\pi_{k_s}}{\sqrt{T}} \brs*{\brs{h(x_{k_s})}^{1+\nu} - \brs{h(x_k)}^{1+\nu} + f(x_k) - f^*}
  }
  \\&\leq
  \cO\brs*{
    \tsum_{s=1}^J\tsum_{k \in \cK_s}\tmfrac{\alpha_k\pi_{k_s}}{\sqrt{T}} \brs*{
      \brs{h(x_k) + \abs*{h(x_{k_s}) - h(x_k)}}^{\vphantom{1}\nu}\abs*{h(x_{k_s}) - h(x_k)}
      + f(x_k) - f^*
    }
  }
  \\&\at{use the definitions in \cref{eq:h} and Young's inequality}\leq
  \cO\brs*{
    \tsum_{s=1}^J\tsum_{k \in \cK_s}\tmfrac{\alpha_k\pi_{k_s}}{\sqrt{T}} \brs*{
      \abs{h(x_{k_s}) - h(x_k)}^{1+\nu}
      +\beta^{\frac{\nu}{1+\nu}}\abs*{h(x_{k_s}) - h(x_k)}
      + f(x_k) - f^*
    }
  }
  \\&\at{uses \Cref{lem:SL01} and the upper bound on $\norm{x_k - x_{k_s}}$ above}\leq
  \cO\brs*{
    \tsum_{s=1}^J\tsum_{k \in \cK_s}\tmfrac{\alpha_k\pi_{k_s}}{\sqrt{T}}
    \norm{x_k-x_{k_s}}^{1+\nu}\brr*{\cL_0 + \brs{\cL_1 h(x_k)}^{1+\nu}}
  }
  \\&
  +\cO\brs*{
    \tsum_{s=1}^J\tsum_{k \in \cK_s}\tmfrac{\alpha_k\pi_{k_s}}{\sqrt{T}} \brs*{
      \beta^{\frac{\nu}{1+\nu}}\norm{x_k-x_{k_s}}\brr*{\cL_0^{\frac{1}{1+\nu}} + \cL_1 h(x_k)}
      + f(x_k) - f^*
    }
  }
  \\&\at{use the definitions in \cref{eq:h} and Young's inequality}\leq
  \cO\brs*{
    \tsum_{s=1}^J\tsum_{k \in \cK_s}\tmfrac{\alpha_k\pi_{k_s}}{\sqrt{T}}
    \norm{x_k-x_{k_s}}^{1+\nu}\brr*{\cL_0 + \cL_1^{1+\nu}\beta + \cL_1^{1+\nu}\brs{f(x_k) - f^*}}
  }
  \\&
  +\cO\brs*{
    \tsum_{s=1}^J\tsum_{k \in \cK_s}\tmfrac{\alpha_k\pi_{k_s}}{\sqrt{T}} \brs*{
      \norm{x_k-x_{k_s}}\brr*{\brs{\cL_0\beta^\nu}^{\frac{1}{1+\nu}} + \cL_1\beta + \cL_1 \brs{f(x_k) - f^*}}
      + f(x_k) - f^*
    }
  }
  \\&\at{uses the definition in \cref{eq:h}}\leq
  \cO\brs*{
    \tsum_{s=1}^J\tsum_{k \in \cK_s}\tmfrac{\alpha_k\pi_{k_s}}{\sqrt{T}}
    \norm{x_k-x_{k_s}}^{1+\nu}\brr*{\cL_0 + \sigma\cL_1^{\nu} + \cL_1^{1+\nu}\brs{f(x_k) - f^*}}
  }
  \\&
  +\cO\brs*{
    \tsum_{s=1}^J\tsum_{k \in \cK_s}\tmfrac{\alpha_k\pi_{k_s}}{\sqrt{T}} \brs*{
      \norm{x_k-x_{k_s}}\brr*{\sigma + \cL_1 \brs{f(x_k) - f^*}}
      + f(x_k) - f^*
    }
  }
  \\&\at{use the upper bound on $\norm{x_k - x_{k_s}}$ above}\leq
  \cO\brs*{
    \tsum_{s=1}^J\tsum_{k \in \cK_s}\tmfrac{\alpha_{k}\pi_{k_s}}{\sqrt{T}}
    \brs*{
      \norm{x_k-x_{k_s}}^{1+\nu}\brr*{\cL_0 + \sigma\cL_1^{\nu}}
      +\norm{x_k-x_{k_s}}\sigma
      + f(x_k) - f^*
    }
  }
  \\&\at{use the upper bound on $\norm{x_k - x_{k_s}}$ above}\leq
  \cO\brs*{
    \tsum_{s=1}^J\tsum_{k \in \cK_s}\tmfrac{\alpha_k\pi_{k_s}}{\sqrt{T}}
    \brs*{
      \brs{\alpha_{k_s}T\cR}^{1+\nu}\brr*{\cL_0 + \sigma\cL_1^{\nu}}
      +\alpha_{k_s}T\cR\sigma
      + f(x_k) - f^*
    }
  }
  \\&\at{uses Young's inequality}\leq
  \cO\brs*{
    \tsum_{s=1}^J\tsum_{k \in \cK_s}
    \tmfrac{\pi_{k_s}}{\sqrt{T}}
    \brs*{
      \brs*{\alpha_k^{2+\nu} + \alpha_{k_s}^{2+\nu}}
      \brs{T\cR}^{1+\nu}\brr*{\cL_0 + \sigma\cL_1^{\nu}}
      +
      \brs{\alpha_k^2 + \alpha_{k_s}^2}T\cR\sigma
    }
  }
  \\&
  +\cO\brs*{
    \tsum_{s=1}^J\tsum_{k \in \cK_s}\tmfrac{\alpha_k\pi_{k_s}}{\sqrt{T}}
    \brs{
      f(x_k) - f^*
    }
  }
  \\&\at{uses the inequality above with $q = 2+\nu$ and $q=2$, and the fact that $\pi_{k_s} \leq \pi_k$}\leq
  \cO\brs*{
    \tsum_{k=1}^K
    \tmfrac{\alpha_k\pi_k}{\sqrt{T}}
    \brs*{
      \brs{\alpha_kT\cR}^{1+\nu}\brr*{\cL_0 + \sigma\cL_1^{\nu}}
      +
      \alpha_kT\cR\sigma
      +
      f(x_k) - f^*
    }
  },
\end{align*}
where \annotate.
Finally, we can upper-bound $\E{\reg_K}$ as follows:
\begin{align*}
  \mi{2}\E{\reg_K}
  \at{use the inequalities above}\leq
  \cO\brs*{
    \cR\sigma\sqrt{\tsum_{k=1}^K \alpha_k^2\pi_k^2}
    +\cL_0\cR^{1+\nu}
    \brs*{\tsum_{k=1}^K [\alpha_k\pi_k]^{\frac{2}{1-\nu}}}^{\frac{1-\nu}{2}}
  }
  \\&
  +\tfrac{1}{4}\tsum_{k=1}^K \alpha_k\pi_k\E*{f(x_k) - f^*}
  +\cO\brs*{\E*{\cR\cL_1\sqrt{\tsum_{k=1}^K \alpha_k^2\pi_k^2\brs*{f(x_k) - f^*}^2}}}
  \\&\at{use the inequalities above}\leq
  \cO\brs*{
    \cR\sigma\sqrt{\tsum_{k=1}^K \alpha_k^2\pi_k^2}
    +\cL_0\cR^{1+\nu}
    \brs*{\tsum_{k=1}^K [\alpha_k\pi_k]^{\frac{2}{1-\nu}}}^{\frac{1-\nu}{2}}
  }
  +\tfrac{1}{4}\tsum_{k=1}^K \alpha_k\pi_k\E*{f(x_k) - f^*}
  \\&
  +\cO\brs*{
    \cR\cL_1\tsum_{s=1}^J\sqrt{T} \alpha_{k_s}\pi_{k_s}\E*{f(x_{k_s}) - f^*}
    +
    \brs{T\cR}^{1+\nu}\cL_0 \tmfrac{\cR\cL_1}{\sqrt{T}}\tsum_{k=1}^K\alpha_k^{2+\nu}\pi_k
  }
  \\&\at{use the inequalities above}\leq
  \cO\brs*{
    \cR\sigma\sqrt{\tsum_{k=1}^K \alpha_k^2\pi_k^2}
    +\cL_0\cR^{1+\nu}
    \brs*{\tsum_{k=1}^K [\alpha_k\pi_k]^{\frac{2}{1-\nu}}}^{\frac{1-\nu}{2}}
  }
  \\&
  +\cO\brs*{
    \cR\cL_1\tsum_{k=1}^K
    \tmfrac{\alpha_k\pi_k}{\sqrt{T}}
    \brs*{
      \brs{\alpha_kT\cR}^{1+\nu}\brr*{\cL_0 + \sigma\cL_1^{\nu}}
      +
      \alpha_kT\cR\sigma
    }
  }
  \\&
  +\cO\brs*{
    \brs{T\cR}^{1+\nu}\cL_0 \tmfrac{\cR\cL_1}{\sqrt{T}}\tsum_{k=1}^K\alpha_k^{2+\nu}\pi_k
  }
  +\brs*{\tfrac{1}{4} + \cO\brs*{\tmfrac{\cR\cL_1}{\sqrt{T}}}}\tsum_{k=1}^K \alpha_k\pi_k\E*{f(x_k) - f^*}
  \\&\at{uses the definition of $T$}\leq
  \tfrac{1}{2}\tsum_{k=1}^K \alpha_k\pi_k\E*{f(x_k) - f^*}
  +\cO\brs*{
    \cR\sigma\sqrt{\tsum_{k=1}^K \alpha_k^2\pi_k^2}
    +\cL_0\cR^{1+\nu}
    \brs*{\tsum_{k=1}^K [\alpha_k\pi_k]^{\frac{2}{1-\nu}}}^{\frac{1-\nu}{2}}
  }
  \\&
  +\cO\brs*{
    T\cR\sigma\tsum_{k=1}^K
    \alpha_k^2\pi_k
    \brr*{
      \brs{\alpha_kT\cR\cL_1}^{\nu}
      +
      1
    }
    +
    \brs{T\cR}^{1+\nu}\cL_0\tsum_{k=1}^K
    \alpha_k^{2+\nu}\pi_k
  }
  \\&\at{uses the definition in \cref{eq:alpha_uni}}\leq
  \tfrac{1}{2}\tsum_{k=1}^K \alpha_k\pi_k\E*{f(x_k) - f^*}
  +\cO\brs*{
    \cR\sigma\sqrt{\tsum_{k=1}^K \alpha_k^2\pi_k^2}
    +\cL_0\cR^{1+\nu}
    \brs*{\tsum_{k=1}^K [\alpha_k\pi_k]^{\frac{2}{1-\nu}}}^{\frac{1-\nu}{2}}
  }
  \\&
  +\cO\brs*{
    T\cR\sigma\tsum_{k=1}^K
    \alpha_k^2\pi_k
    +
    \brs{T\cR}^{1+\nu}\cL_0\tsum_{k=1}^K
    \alpha_k^{2+\nu}\pi_k
  },
\end{align*}
where \annotate.\qed

\proofsubsection{thm:uni}

First, observe that the statement of \Cref{lem:regret_uni} holds for arbitrary $K \geq T$.
Next, we define $\Delta_k = \E{f(x_k) - f^*}$ and $\delta_k$ as follows:
\begin{align*}
  \delta_k = e\cF &+ \cO\brs*{
    [T\cR]^{1+\nu}\cL_0\tsum_{i=1}^k
    \alpha_i^{2+\nu}\pi_i
    +
    T\cR\sigma\tsum_{i=1}^k\alpha_i^2\pi_i
  }
  \\&
  +\cO\brs*{
    \cR\sigma\sqrt{\tsum_{i=1}^k \alpha_i^2\pi_i^2}
    +
    \cL_0\cR^{1+\nu}\brs*{\tsum_{i=1}^k [\alpha_i\pi_i]^{\frac{2}{1-\nu}}}^{(1-\nu)/2}
  }.
\end{align*}
Using \Cref{lem:regret_uni,lem:wd}, and using similar arguments to \Cref{proof:thm:exp1}, we can show that the following inequality holds for all $k \in \N$:
\begin{align*}
  \pi_k \Delta_k \leq \delta_k
  +\tfrac{1}{2}\tsum_{i=1}^k \alpha_i\pi_i \Delta_i.
\end{align*}
Hence, the conditions of \Cref{lem:tech} with $p=1/2$ are satisfied.
Furthermore, for $k \leq S T$, we can upper-bound $\delta_k$ as follows:
\begin{align*}
  \delta_k
  \at{uses the definitions in \cref{eq:alpha_uni,eq:fk}}\leq
  e\cF &+  \cO\brs*{
    [T\cR]^{1+\nu}\cL_0\alpha_1^{2+\nu}\pi_k\tsum_{i=0}^{\infty}(1-\alpha_1)^i
    +
    T\cR\sigma\alpha_1^2\pi_k\tsum_{i=0}^{\infty}(1-\alpha_1)^i
  }
  \\&
  +\cO\brs*{
    \cR\sigma\alpha_1\pi_k\sqrt{\tsum_{i=0}^{\infty}(1-\alpha_1)^{2i}}
    +
    \cL_0\cR^{1+\nu}\alpha_1\pi_k\brs*{\tsum_{i=0}^{\infty} (1-\alpha_1)^{\frac{2i}{1-\nu}}}^{(1-\nu)/2}
  }
  \\\leq
  e\cF &+  \cO\brs*{
    [T\cR]^{1+\nu}\cL_0\alpha_1^{1+\nu}\pi_k
    +
    T\cR\sigma\alpha_1\pi_k
  }
  \\&
  +\cO\brs*{
    \cR\sigma\sqrt{\alpha_1}\pi_k
    +
    \cL_0\cR^{1+\nu}\alpha_1^{\frac{1+\nu}{2}}\pi_k
  }
  \\\at{uses the definition in \cref{eq:alpha_uni}}\leq
  e\cF &+ \cO\brs*{
    \brs{\cL_0/\cL_1^{1+\nu}}\pi_k
    +
    \brs{\sigma/\cL_1}\pi_k
  }
  \\\at{uses the definition of $\hF$}=
  e\cF &+ \cO\brs*{\hF\pi_k}
\end{align*}
where \annotate.
For $k \geq ST+1$, we can upper-bound $\delta_k$ as follows:
\begin{align*}
  \delta_k
  &\at{uses the definitions in \cref{eq:alpha_uni,eq:fk}}\leq
  \delta_{ST}  + \cO\brs*{
    [T\cR]^{1+\nu}\cL_0\alpha_{ST+1}^{2+\nu}\pi_k\tsum_{i=0}^{\infty}(1-\alpha_{ST+1})^i
    +
    T\cR\sigma\alpha_{ST+1}^2\pi_k\tsum_{i=0}^{\infty}(1-\alpha_{ST+1})^i
  }
  \\&
  +\cO\brs*{
    \cR\sigma\alpha_{ST+1}\pi_k\sqrt{\tsum_{i=0}^{\infty}(1-\alpha_{ST+1})^{2i}}
    +
    \cL_0\cR^{1+\nu}\alpha_{ST+1}\pi_k\brs*{\tsum_{i=0}^{\infty} (1-\alpha_{ST+1})^{\frac{2i}{1-\nu}}}^{\frac{1-\nu}{2}}
  }
  \\&\leq
  \delta_{ST} +  \cO\brs*{
    \brs{T\cR}^{1+\nu}\cL_0\alpha_{ST+1}^{1+\nu}\pi_k
    +
    T\cR\sigma\alpha_{ST+1}\pi_k
  }
  \\&
  +\cO\brs*{
    \cR\sigma\sqrt{\alpha_{ST+1}}\pi_k
    +
    \cL_0\cR^{1+\nu}\alpha_{ST+1}^{\frac{1+\nu}{2}}\pi_k
  }
  \\&\at{uses the definition in \cref{eq:alpha_uni}}\leq
  \delta_{ST} + \cO\brs*{\epsilon\pi_k},
\end{align*}
where \annotate.
Finally, we can upper-bound $\Delta_K$ as follows:
\begin{align*}
  \Delta_K
  &\at{uses \Cref{lem:tech} with $p=1/2$}\leq
  \cO\brs*{\tmfrac{\delta_K}{\pi_K} + \tsum_{k=1}^K\tmfrac{\alpha_k\delta_k}{\sqrt{\pi_K\pi_{k-1}}}}
  \\&\at{uses the fact that $K \geq TS+1$}\leq
  \cO\brs*{
    \tmfrac{\delta_K}{\pi_K}
    +
    \tsum_{k=1}^{ST}\tmfrac{\alpha_k\delta_k}{\sqrt{\pi_K\pi_{k-1}}}
    +
    \tsum_{k=ST+1}^{K}\tmfrac{\alpha_k\delta_k}{\sqrt{\pi_K\pi_{k-1}}}
  }
  \\&\at{uses the inequalities above}\leq
  \cO\brs*{
    \tmfrac{\cF}{\pi_K}
    +
    \tmfrac{\hF\pi_{ST}}{\pi_K}
    +
    \epsilon
    +
    \tsum_{k=1}^{ST}\brs*{
      \tmfrac{\alpha_k\cF}{\sqrt{\pi_K\pi_{k-1}}}
      +
      \tmfrac{\hF\alpha_k\pi_k}{\sqrt{\pi_K\pi_{k-1}}}
  }}
  \\&
  +
  \cO\brs*{\tsum_{k=ST+1}^{K}\brs*{
      \tmfrac{\alpha_k\cF}{\sqrt{\pi_K\pi_{k-1}}}
      +
      \tmfrac{\hF\alpha_k\pi_{ST}}{\sqrt{\pi_K\pi_{k-1}}}
      +
      \tmfrac{\epsilon\alpha_k\pi_k}{\sqrt{\pi_K\pi_{k-1}}}
  }}
  \\&\at{uses the definitions in \cref{eq:fk,eq:alpha_uni}}\leq
  \cO\brs*{
    \cF[1-\alpha_1]^{ST/2}[1-\alpha_{ST+1}]^{(K-ST)/2}
    +
    \hF[1-\alpha_{ST+1}]^{(K-ST)/2}
    +
    \epsilon
  }
  \\&\at{uses the definitions in \cref{eq:KS_uni}}\leq
  \cO\brs*{\epsilon},
\end{align*}
where \annotate.\qed

\newpage

\section[Proofs for \crtCref{sec:quasar}]{Proofs for \Cref{sec:quasar}}

\subsection{Technical Lemma}

In this section, we present a technical lemma that will be used in the proof of \Cref{thm:quasar}.

\begin{lemma}<lem:tech2>
  Let $p \in (0,1/2]$, and let $\brf{\Delta_k}_{k \in \N} \subset \R_+$ and $\brf{\delta_k}_{k \in \N} \subset \R_+$ be two sequences of non-negative numbers satisfying the following condition for all $k \in \N$:
  \begin{equation}
    \pi_k \Delta_k \leq \delta_k + p\gamma\tsum_{i=1}^k \alpha_i\pi_i \Delta_i,
  \end{equation}
  where $\alpha_k,\pi_k$ satisfy \cref{eq:pi_quasar}. Then, the following inequalities hold for all $k \in \N$:
  \begin{equation}
    \tsum_{i=1}^k\alpha_i\pi_i \Delta_i
    \leq
    \tsum_{i=1}^k \alpha_i\delta_i\brs*{\pi_k/\pi_{i-1}}^p.
  \end{equation}
\end{lemma}

\begin{proof}

  We define $\tau_k \geq 0$ for $k \in \N$ as follows:
  \begin{align*}
    \tau_k = \tsum_{i=1}^k \alpha_i\pi_i\Delta_i.
  \end{align*}
  We can upper-bound $\tau_k$ as follows:
  \begin{align*}
    \tau_k
    \at{uses the definition of $\tau_k$}=
    \alpha_k\pi_k\Delta_k + \tau_{k-1}
    \at{uses the assumption of \Cref{lem:tech2}}\leq
    \alpha_k\delta_k + p\gamma\alpha_k\tau_k + \tau_{k-1}
    \at{uses the inequality $pt \leq 1 - (1-t)^p$ for $t < 1$ which is implied by the concavity of the function $t\mapsto t^p$}\leq
    \alpha_k\delta_k + \brs*{1 - (1-\gamma\alpha_k)^p}\tau_k + \tau_{k-1},
  \end{align*}
  where \annotate.
  After rearranging, we obtain the following:
  \begin{align*}
    \alpha_k\delta_k + \tau_{k-1}
    \geq
    (1-\gamma\alpha_k)^p\tau_k
    \at{uses the definition in \cref{eq:pi_quasar}}=
    \brs*{\pi_{k-1}/\pi_k}^p\tau_k
  \end{align*}
  where \annotate.
  After rearranging one more time, we obtain the following:
  \begin{align*}
    \brs*{1/\pi_k}^p\tau_k \leq \brs*{1/\pi_{k-1}}^p\brs*{\tau_{k-1} + \alpha_k\delta_k}.
  \end{align*}
  Hence, we can upper-bound $\tau_k$ as follows:
  \begin{align*}
    \tau_k \leq \tsum_{i=1}^k \alpha_i\delta_i\brs*{\pi_k/\pi_{i-1}}^p.
  \end{align*}
\end{proof}

\proofsubsection{lem:quasar}

\textbf{\Cref{eq:quasar} implies \Cref{eq:quasar_grad}.}
Since the function $f\gprime(x;\cdot\,)$ is subadditive and homogeneous, it is equal to the support function of the generalized subdifferential $\gpartial f(x)$. Hence, $\<w,g>\leq f\gprime(x;w)$ for all $g \in \gpartial f(x)$. Furthermore, we can upper-bound $\<x^*-x, \nabla f(x;w)>$ as follows:
\begin{align*}
  \<x^*-x, \nabla f(x;w)>
  &\at{uses the fact that $\nabla f(x;w) \in \gpartial f(x)$}\leq
  f\gprime(x;x^* - x)
  \\&\at{uses the definition in \cref{eq:gprime}}=
  \tlimsup_{x'\to x, t\to +0} \tfrac{1}{t}\brs{f(x' + t(x^* - x)) - f(x')}
  \\&\leq
  \tlimsup_{x'\to x, t\to +0} \tfrac{1}{t}\brs{f(x' + t(x^* - x)) - f(x' + t(x^* - x'))}
  \\&
  +\tlimsup_{x'\to x, t\to +0} \tfrac{1}{t}\brs{f(x' + t(x^* - x')) - f(x')}
  \\&\at{use the fact that the function $f(x)$ is locally Lipschitz}=
  \tlimsup_{x'\to x, t\to +0} \tfrac{1}{t}\brs{f(x' + t(x^* - x')) - f(x')}
  \\&\at{uses the $\gamma$-quasar convexity of the function $f(x)$ in \cref{eq:quasar}}\leq
  \tlimsup_{x'\to x, t\to +0} \gamma\brs{f^* - f(x')}
  \\&\at{use the fact that the function $f(x)$ is locally Lipschitz}=
  \gamma\brs{f^* - f(x)},
\end{align*}
where \annotate.

\textbf{\Cref{eq:quasar_grad} implies \Cref{eq:quasar}.}
Let the function $p(t)\colon [0,1] \to \R_+$ be defined as follows:
\begin{align*}
  p(t) = f(x^* + t(x-x^*)) - f^*.
\end{align*}
We can lower-bound $p(t') - p(t)$ for all $0 < t < t' \leq 1$ as follows:
\begin{align*}
  \mi{3}p(t') - p(t)
  \at{uses Lemma~3 of \citet{zhang2020complexity}}=
  \mint_0^1 f'(x^* + (t + s(t'-t))(x-x^*);(t' - t)(x-x^*)) \df s
  \\&\at{uses the property in \cref{eq:ggrad}}=
  \mint_0^1 \<\nabla f(x^* + (t + s(t'-t))(x-x^*);(t' - t)(x-x^*)),(t' - t)(x-x^*)> \df s
  \\&=
  \mint_0^1 \tmfrac{\<\nabla f(x^* + (t + s(t'-t))(x-x^*);(t'-t)(x-x^*)),[x^* + (t + s(t'-t))(x-x^*)] - x^*>\df s}{t/(t'-t) + s}
  \\&\at{uses the inequality in \cref{eq:quasar_grad}}\geq
  \mint_0^1\tmfrac{\gamma\brs{f(x^* + (t + s(t'-t))(x-x^*)) - f^*}(t'-t)\df s}{t + s(t'-t)}
  =
  \mint_t^{t'}\tmfrac{\gamma }{s}p(s)\df s \geq 0,
\end{align*}
where \annotate.
Hence, the function $p(t)$ is non-decreasing. Furthermore, let $c \in [0,1]$ such that $p(c) > 0$, and let the function $q(t)\colon [c,1] \to \R$ be defined as follows:
\begin{align*}
  q(t) = p(c) + \mint_c^t \tmfrac{\gamma}{s}p(s) ds.
\end{align*}
From the previously obtained inequalities, we can conclude that $p(c) = q(c) \leq q(t) \leq p(t)$ for all $t \in [c,1]$. Moreover, from the continuity of the function $p(t)$ it follows that the function $q(t)$ is continuously differentiable, and we can lower-bound its derivative $\dot{q}(t)$ as follows:
\begin{align*}
  \dot{q}(t)
  &\at{uses the definition of the function $q(t)$}=
  [\gamma/t]\cdot p(t)
  \at{uses the inequality $q(t)\leq p(t)$ above}\geq
  [\gamma/t]\cdot q(t),
\end{align*}
where \annotate.
This implies the following inequality:
\begin{align*}
  \dot{l}(t) \geq \gamma/t,
\end{align*}
where $l(t) = \ln q(t)$. After integrating this inequality on $[c,1]$, we obtain the following:
\begin{align*}
  \ln q(1) \geq  \ln q(c) - \gamma \ln c.
\end{align*}
By taking the exponent of both sides and rearranging, we obtain the following:
\begin{align*}
  q(c) \leq c^\gamma q(1) \leq (1 - \gamma(1-c))q(1).
\end{align*}
Finally, we can upper-bound $p(c)$ as follows:
\begin{align*}
  p(c) = q(c)  \leq  (1 - \gamma(1-c))q(1) \leq (1 - \gamma(1-c))p(1),
\end{align*}
which trivially holds even if $p(c) = 0$ and proves the inequality in \cref{eq:quasar} by choosing $c = 1-t$. \qed

\proofsubsection{thm:quasar}

First, using the definitions in \cref{eq:wd,eq:alpha_quasar,eq:g_quasar} and the fact that $z_k \in Q_\cR$, we can upper-bound $\norm{\ox_k - x_k}$ and $\norm{x_k - x_{k-1}}$ as follows:
\begin{align*}
  \max\brf{\norm{\ox_k - x_k},\norm{x_k - x_{k-1}}} \las \cO\brs{\alpha_k\cR} \leq \cO\brs{1/(T\cG_1)}.
\end{align*}
Next, we can upper-bound $(1-\alpha_k)\pi_k\E*{f(x_k) - f(x_{k-1})}$ as follows:
\begin{align*}
  \mi{3}(1-\alpha_k)\pi_k\E*{f(x_k) - f(x_{k-1})}
  \\&\at{Lemma~3 of \citet{zhang2020complexity} and the property in \cref{eq:ggrad}}=
  (1-\alpha_k)\pi_k\E*{\mint_0^1 \<\nabla f(x_{k-1} + \zeta(x_k - x_{k-1});x_k - x_{k-1}),x_k - x_{k-1}>\df \zeta}
  \\&\at{uses the definition of $\ox_k$ in \cref{eq:g_quasar} and the definition of $\zeta_k$ and its independence of $x_k$ and $x_{k-1}$}=
  (1-\alpha_k)\pi_k\E{\<\nabla f(\ox_k;x_k - x_{k-1}),x_k - x_{k-1}>}
  \\&\at{uses the definition in \cref{eq:wd}}=
  \alpha_k\pi_k\E{\<\nabla f(\ox_k;x_k - x_{k-1}),z_k - x_k>}
  \\&=
  \alpha_k\pi_k\E{\<\nabla f(\ox_k;x_k - x_{k-1}),x^* - \ox_k + \ox_k - x_k + z_k - x^*>}
  \\&\at{uses \Cref{lem:quasar}}\leq
  \alpha_k\pi_k\E{\gamma\brs{f^* - f(\ox_k)} + \<\nabla f(\ox_k;x_k - x_{k-1}), \ox_k -x_k + z_k - x^*>}
  \\&\at{uses the definition of $g_k$ in \cref{eq:g_quasar}, the unbiasedness in \Cref{ass:qG01}, and the fact that $\xi_k$ is independent of $\zeta_k$ and $z_k,x_k,x_{k-1}$}=
  \gamma\alpha_k\pi_k\E*{f^* - f(\ox_k)} + \alpha_k\pi_k\E{\<\nabla f(\ox_k;x_k - x_{k-1}), \ox_k -x_k>} + \E{\<g_k, z_k - x^*>},
\end{align*}
where \annotate.
After rearranging, we obtain the following inequality:
\begin{align*}
  \mi{3}\pi_k\E*{f(x_k) - f^*}
  \\&\leq
  (1-\gamma\alpha_k)\pi_k\E*{f(x_{k-1}) - f^*}
  +\E{\<g_k, z_k - x^*>}
  \\&
  +(1-\gamma)\alpha_k\pi_k\E*{f(x_k) - f(x_{k-1})}
  +\gamma\alpha_k\pi_k\E*{f(x_k) - f(\ox_k)}
  \\&
  +\alpha_k\pi_k\E{\<\nabla f(\ox_k;x_k - x_{k-1}), \ox_k -x_k>}
  \\&\at{uses the definition in \cref{eq:pi_quasar}}\leq
  \pi_{k-1}\E*{f(x_{k-1}) - f^*}
  +\E{\<g_k, z_k - x^*>}
  \\&
  +(1-\gamma)\alpha_k\pi_k\E*{f(x_k) - f(x_{k-1})}
  +\gamma\alpha_k\pi_k\E*{f(x_k) - f(\ox_k)}
  \\&
  +\alpha_k\pi_k\E{\<\nabla f(\ox_k;x_k - x_{k-1}), \ox_k -x_k>}
  \\&\at{uses \Cref{ass:qG01} and the upper bound on $\norm{\ox_k - x_k}$ above}\leq
  \pi_{k-1}\E*{f(x_{k-1}) - f^*}
  +\E{\<g_k, z_k - x^*>}
  \\&
  +(1-\gamma)\alpha_k\pi_k\E*{f(x_k) - f(x_{k-1})}
  +\gamma\alpha_k\pi_k\E*{f(x_k) - f(\ox_k)}
  \\&
  +\cO\brs*{\alpha_k^2\pi_k\cR\cG_0 + \tmfrac{\alpha_k\pi_k}{T}\E{f(\ox_k) - f^*}}
  \\&\leq
  \pi_{k-1}\E*{f(x_{k-1}) - f^*}
  +\E{\<g_k, z_k - x^*>}
  +\cO\brs*{\alpha_k^2\pi_k\cR\cG_0 + \tmfrac{\alpha_k\pi_k}{T}\E{f(x_k) - f^*}}
  \\&
  +\cO\brs*{\alpha_k\pi_k\E{\abs{f(x_k) - f(x_{k-1})} + \abs{f(x_k) - f(\ox_k)}}}
  \\&\at{uses \Cref{ass:qG01}, \Cref{rem:qGM01}, and \Cref{lem:M01}}\leq
  \pi_{k-1}\E*{f(x_{k-1}) - f^*}
  +\E{\<g_k, z_k - x^*>}
  +\cO\brs*{\alpha_k^2\pi_k\cR\cG_0 + \tmfrac{\alpha_k\pi_k}{T}\E{f(x_k) - f^*}}
  \\&
  +\cO\brs*{\alpha_k\pi_k\E*{\brr{\cG_0 + \cG_1\brs{f(x_k) - f^*}}\exp\brr{\cG_1\norm{\ox_k - x_k}}\norm{\ox_k - x_k}}}
  \\&
  +\cO\brs*{\alpha_k\pi_k\E*{\brr{\cG_0 + \cG_1\brs{f(x_k) - f^*}}\exp\brr{\cG_1\norm{x_{k-1} - x_k}}\norm{x_{k-1} - x_k}}}
  \\&\at{uses the upper bounds on $\norm{\ox_k - x_k}$ and $\norm{x_{k-1} - x_k}$ above}\leq
  \pi_{k-1}\E*{f(x_{k-1}) - f^*}
  +\E{\<g_k, z_k - x^*>}
  +\cO\brs*{\alpha_k^2\pi_k\cR\cG_0 + \tmfrac{\alpha_k\pi_k}{T}\E{f(x_k) - f^*}}
  \\&\at{uses the definition of $T$}\leq
  \pi_{k-1}\E*{f(x_{k-1}) - f^*} + \E{\<g_k, z_k - x^*>}
  +\tfrac{1}{4}\gamma\alpha_k\pi_k\E*{f(x_k) - f^*}
  +\cO\brs*{\alpha_k^2\pi_k \cR\cG_0}
\end{align*}
where \annotate.
After summing these inequalities for $k = 1,\dots,K$, we obtain the following:
\begin{align*}
  \pi_K\E{f(x_K) - f^*}
  &\leq
  \brs{f(x_0) - f^*} + \E{\reg_K}
  \\&
  +\tfrac{1}{4}\gamma\tsum_{k=1}^{K}\alpha_k\pi_k\E*{f(x_k) - f^*}
  +  \cO\brs*{\cR\cG_0\tsum_{k=1}^{K}\alpha_k^2\pi_k},
\end{align*}
Next, using \Cref{ass:regret,ass:qG01} and similar arguments to the proof of \Cref{lem:regret} in \Cref{proof:lem:regret}, we can upper-bound $\E{\reg_K}$ as follows:
\begin{align*}
  \E{\reg_K}
  &\leq
  \cO\brs*{\cR\cG_0\sqrt{\tsum_{k=1}^K \alpha_k^2\pi_k^2}
  + \E*{\cR\cG_1\sqrt{\tsum_{k=1}^K \alpha_k^2\pi_k^2\brs*{f(\ox_k) - f^*}^2}}}
  \\&\at{uses \Cref{ass:qG01}, \Cref{rem:qGM01}, and \Cref{lem:M01}}\leq
  \cO\brs*{\cR\cG_0\sqrt{\tsum_{k=1}^K \alpha_k^2\pi_k^2}
  +\E*{\cR\cG_1\sqrt{\tsum_{k=1}^K \alpha_k^2\pi_k^2\brs*{f(x_k) - f^*}^2}}}
  \\&
  +\cO\brs*{\E*{\cR\cG_1\sqrt{\tsum_{k=1}^K \alpha_k^2\pi_k^2\brr*{\cG_0 + \cG_1\brs{f(x_k) - f^*}}^2\exp\brr{2\cG_1\norm{\ox_k - x_k}}\sqn{\ox_k - x_k}}}}
  \\&\at{uses the upper bound on $\norm{\ox_k - x_k}$ above}\leq
  \cO\brs*{\cR\cG_0\sqrt{\tsum_{k=1}^K \alpha_k^2\pi_k^2}
  +\E*{\cR\cG_1\sqrt{\tsum_{k=1}^K \alpha_k^2\pi_k^2\brs*{f(x_k) - f^*}^2}}}
  \\&
  +\cO\brs*{\E*{\tmfrac{\cR}{T}\sqrt{\tsum_{k=1}^K \alpha_k^2\pi_k^2\brr*{\cG_0 + \cG_1\brs{f(x_k) - f^*}}^2}}}
  \\&\leq
  \cO\brs*{\cR\cG_0\sqrt{\tsum_{k=1}^K \alpha_k^2\pi_k^2}
  +\E*{\cR\cG_1\sqrt{\tsum_{k=1}^K \alpha_k^2\pi_k^2\brs*{f(x_k) - f^*}^2}}},
\end{align*}
where \annotate.
Moreover, using similar arguments to the proof of \Cref{lem:regret_exp2} in \Cref{proof:lem:regret_exp2}, with the modified weight recursion in \cref{eq:pi_quasar}, and observing that the blockwise bound $\pi_k \leq \cO\brs{\pi_{k_s}}$ still holds, we can further upper-bound $\E{\reg_K}$ as follows:
\begin{align*}
  \E{\reg_K}
  &\leq
  \cO\brs*{\cR\cG_0\sqrt{\tsum_{k=1}^K \alpha_k^2\pi_k^2}
    +\sqrt{T}\cR^2\cG_0\cG_1\tsum_{k=1}^K\alpha_k^2\pi_k
  +\tmfrac{\cR\cG_1}{\sqrt{T}}\tsum_{k=1}^K  \alpha_k\pi_k\E*{f(x_k) - f^*}}
  \\&\at{uses the definition of $T$}\leq
  \tfrac{1}{4}\gamma\tsum_{k=1}^K  \alpha_k\pi_k\E*{f(x_k) - f^*}
  +\cO\brs*{\cR\cG_0\sqrt{\tsum_{k=1}^K \alpha_k^2\pi_k^2}
  +\sqrt{T}\cR^2\cG_0\cG_1\tsum_{k=1}^K\alpha_k^2\pi_k}
\end{align*}
where \annotate. Now, we combine this with the previously obtained upper bound on $\pi_K\E{f(x_K) - f^*}$ and get the following inequality:
\begin{align*}
  \pi_K\E{f(x_K) - f^*} &\leq \brs{f(x_0) - f^*}
  +\tfrac{1}{2}\gamma\tsum_{k=1}^K  \alpha_k\pi_k\E*{f(x_k) - f^*}
  \\&
  +\cO\brs*{\cR\cG_0\sqrt{\tsum_{k=1}^K \alpha_k^2\pi_k^2}
  +\sqrt{T}\cR^2\cG_0\cG_1\tsum_{k=1}^K\alpha_k^2\pi_k},
\end{align*}
which holds for all $K \geq T$.
Finally, let $\Delta_k = \E{f(x_k) - f^*}$, and let $\delta_k$ be defined as follows:
\begin{align*}
  \delta_k
  =
  e\cF + \cO\brs*{\cR\cG_0\sqrt{\tsum_{i=1}^k \alpha_i^2\pi_i^2}
  +T\cR\cG_0\tsum_{i=1}^k\alpha_i^2\pi_i}.
\end{align*}
Using the previous inequality and similar arguments to the proof of \Cref{thm:exp1} in \Cref{proof:thm:exp1}, for all $K \in \N$, we obtain the following inequality:
\begin{align*}
  \pi_K \Delta_K \leq \delta_K + \tfrac{1}{2}\gamma\tsum_{k=1}^K\alpha_k\pi_k\Delta_k.
\end{align*}
Next, for $k \leq S T$, we can upper-bound $\delta_k$ as follows:
\begin{align*}
  \delta_k
  &\at{uses the definitions in \cref{eq:alpha_quasar,eq:pi_quasar}}=
  e\cF + \cO\brs*{\cR\cG_0\sqrt{\tsum_{i=1}^k \alpha_1^2(1-\gamma\alpha_1)^{-2i}}
  +T\cR\cG_0\tsum_{i=1}^k\alpha_1^2(1-\gamma\alpha_1)^{-i}}
  \\&\leq
  e\cF + \cO\brs*{
    \sqrt{\alpha_1/\gamma}\pi_k\cR\cG_0
    +
    T\cR\cG_0 [\alpha_1/\gamma] \pi_k
  }
  \\&\at{uses the definition in \cref{eq:alpha_quasar}}\leq
  e\cF + \cO\brs*{\pi_k\cG_0/[\gamma\cG_1]},
\end{align*}
where \annotate.
Furthermore, for $k \geq ST+1$, we can upper-bound $\delta_k$ as follows:
\begin{align*}
  \delta_k
  \at{uses the definitions in \cref{eq:alpha_quasar,eq:pi_quasar}}\leq
  \delta_{ST}
  &+\cO\brs*{\cR\cG_0\sqrt{\tsum_{i=1}^{k-ST} \alpha_{ST+1}^2 \pi_{ST}^2(1-\gamma\alpha_{ST+1})^{-2i}}}
  \\
  &+\cO\brs*{T\cR\cG_0\tsum_{i=1}^{k-ST}\alpha_{ST+1}^2\pi_{ST}(1-\gamma\alpha_{ST+1})^{-i}}
  \\\leq
  \delta_{ST}
  &+
  \cO\brs*{\cR\cG_0\sqrt{\alpha_{ST+1}/\gamma}\pi_k
  +T\cR\cG_0[\alpha_{ST+1}/\gamma]\pi_k}
  \\\at{uses the definition in \cref{eq:alpha_quasar}}\leq
  \delta_{ST}
  &+
  \cO\brs*{\epsilon \pi_k}
\end{align*}
where \annotate.
Finally, we can upper-bound $\Delta_K$ as follows:
\begin{align*}
  \Delta_K
  &\at{uses \Cref{lem:tech2} with $p=1/2$}\leq
  \cO\brs*{
    \tmfrac{\delta_K}{\pi_K} + \tsum_{k=1}^K\tmfrac{\gamma\alpha_k\delta_k}{\sqrt{\pi_K\pi_{k-1}}}
  }
  \\&\at{uses the fact that $K \geq TS+1$}=
  \cO\brs*{
    \tmfrac{\delta_K}{\pi_K}
    +
    \tsum_{k=1}^{ST}\tmfrac{\gamma\alpha_k\delta_k}{\sqrt{\pi_K\pi_{k-1}}}
    +
    \tsum_{k=ST+1}^{K}\tmfrac{\gamma\alpha_k\delta_k}{\sqrt{\pi_K\pi_{k-1}}}
  }
  \\&\at{uses the inequalities above}\leq
  \cO\brs*{
    \tmfrac{\cF}{\pi_K} + \tmfrac{\cG_0\pi_{ST}}{\gamma\cG_1\pi_K} + \epsilon
    +
    \tsum_{k=1}^{ST}\brs*{
      \tmfrac{\gamma\alpha_k\cF}{\sqrt{\pi_K\pi_{k-1}}}
      +
      \tmfrac{\gamma\cG_0\alpha_k\pi_k}{\gamma\cG_1\sqrt{\pi_K\pi_{k-1}}}
    }
  }
  \\&
  +\cO\brs*{
    \tsum_{k=ST+1}^{K}\brs*{
      \tmfrac{\gamma\alpha_k\cF}{\sqrt{\pi_K\pi_{k-1}}}
      +
      \tmfrac{\gamma\cG_0\alpha_k\pi_{ST}}{\gamma\cG_1\sqrt{\pi_K\pi_{k-1}}}
      +
      \tmfrac{\gamma\epsilon\alpha_k\pi_k}{\sqrt{\pi_K\pi_{k-1}}}
    }
  }
  \\&\at{use the definitions in \cref{eq:pi_quasar,eq:alpha_quasar}}\leq
  \cO\brs*{
    \tmfrac{\cF}{\pi_K} + \tmfrac{\cG_0\pi_{ST}}{\gamma\cG_1\pi_K} + \epsilon
    +
    \tmfrac{\gamma\cF\alpha_1}{\sqrt{\pi_K}}\tsum_{k=0}^{ST-1}\brs{1-\gamma\alpha_1}^{k/2}
    +
    \tmfrac{\gamma\cG_0\alpha_1}{\gamma\cG_1\sqrt{\pi_K}}\tsum_{k=0}^{ST-1}\brs{1-\gamma\alpha_1}^{-k/2-1}
  }
  \\&
  +
  \cO\brs*{
    \brs*{
      \tmfrac{\gamma\alpha_{ST+1}\cF}{\sqrt{\pi_K\pi_{ST}}}
      +
      \tmfrac{\gamma\cG_0\alpha_{ST+1}\sqrt{\pi_{ST}}}{\gamma\cG_1\sqrt{\pi_K}}
    }
    \tsum_{k=0}^{K-ST-1}[1-\gamma\alpha_{ST+1}]^{k/2}
  }
  \\&
  +
  \cO\brs*{
    \tmfrac{\gamma\epsilon\alpha_{ST+1}\sqrt{\pi_{ST}}}{2\sqrt{\pi_K}}
    \tsum_{k=0}^{K-ST-1}
    [1-\gamma\alpha_{ST+1}]^{-k/2-1}
  }
  \\&\at{use the definitions in \cref{eq:pi_quasar,eq:alpha_quasar}}\leq
  \cO\brs*{
    \cF[1-\gamma\alpha_1]^{ST/2}[1-\gamma\alpha_{ST+1}]^{[K-ST]/2}
    +
    \brs{\cG_0/[\gamma\cG_1]}[1-\gamma\alpha_{ST+1}]^{[K-ST]/2}
    +
    \epsilon
  }
  \\&\at{uses the definitions in \cref{eq:KS_quasar}}\leq
  \cO\brs*{\epsilon},
\end{align*}
where \annotate.\qed

\end{document}